\documentclass[reqno]{amsart}

\usepackage{amsmath}
\usepackage{amsthm}
\usepackage{amsfonts}
\usepackage{amssymb}
\usepackage{enumerate}
\usepackage{graphicx}

\newcommand{\indicator}[1]{\ensuremath{\mathbf{1}_{\{#1\}}}}
\newcommand{\oindicator}[1]{\ensuremath{\mathbf{1}_{{#1}}}}

\numberwithin{equation}{section}

\newcommand{\E}{\mathbb{E}}
\renewcommand{\P}{\mathbb{P}}
\newcommand{\Prob}{\mathbb{P}}
\newcommand{\C}{\mathbb{C}}

\DeclareMathOperator{\var}{Var}
\DeclareMathOperator{\tr}{tr}

\newcommand{\stcomp}{\mathsf{c}}

\renewcommand\Re{\operatorname{Re}}
\renewcommand\Im{\operatorname{Im}}
\newcommand{\eps}{\varepsilon}
\DeclareMathOperator{\dist}{dist}
\DeclareMathOperator{\rank}{rank}

\theoremstyle{plain}
  \newtheorem{theorem}{Theorem}[section]

  \newtheorem{lemma}[theorem]{Lemma}
  \newtheorem{corollary}[theorem]{Corollary}

\theoremstyle{definition}
  \newtheorem{definition}[theorem]{Definition}
  
  \newtheorem{remark}[theorem]{Remark}

\begin{document}
\title{Low rank perturbations of large elliptic random matrices}

\author[S. O'Rourke]{Sean O'Rourke}
\address{Department of Mathematics, Yale University, New Haven, CT 06520, USA  }
\thanks{S. O'Rourke has been supported by grant AFOSAR-FA-9550-12-1-0083}
\email{sean.orourke@yale.edu}

\author[D. Renfrew]{David Renfrew} 
\address{Department of Mathematics, UCLA  }
\thanks{D. Renfrew has been supported by grant DMS-0838680}
\email{dtrenfrew@math.ucla.edu}

\begin{abstract}
We study the asymptotic behavior of outliers in the spectrum of bounded rank perturbations of large random matrices.  In particular, we consider perturbations of elliptic random matrices which generalize both Wigner random matrices and non-Hermitian random matrices with iid entries.  As a consequence, we recover the results of Capitaine, Donati-Martin, and F\'eral for perturbed Wigner matrices as well as the results of Tao for perturbed random matrices with iid entries.  Along the way, we prove a number of interesting results concerning elliptic random matrices whose entries have finite fourth moment; these results include a bound on the least singular value and the asymptotic behavior of the spectral radius.  
\end{abstract}

\maketitle

\section{Introduction}

In this note, we investigate the asymptotic behavior of outliers in the spectrum of bounded rank perturbations of large random matrices.  We begin by introducing the empirical spectral distribution of a square matrix.  

The \emph{eigenvalues} of a $N \times N$ matrix $M$ are the roots in $\mathbb{C}$ of the characteristic polynomial $\det( M - z I)$, where $I$ is the identity matrix.  We let $\lambda_1(M), \ldots, \lambda_N(M)$ denote the eigenvalues of $M$.  In this case, the \emph{empirical spectral measure} $\mu_{M}$ is given by
$$ \mu_M := \frac{1}{N} \sum_{i=1}^N \delta_{\lambda_i(M)}. $$
The corresponding \emph{empirical spectral distribution} (ESD) is given by
$$ F^{M}(x,y) := \frac{1}{N} \# \left\{1 \leq i \leq N: \Re(\lambda_i(M)) \leq x, \Im(\lambda_i(M)) \leq y \right\}. $$
Here $\# E$ denotes the cardinality of the set $E$.  If the matrix $M$ is Hermitian, then the eigenvalues $\lambda_1(M), \ldots, \lambda_N(M)$ are real.  In this case the ESD is given by 
$$ F^{M}(x) := \frac{1}{N} \# \left\{ 1 \leq i \leq N : \lambda_i(M) \leq x \right\}. $$

Given a random $N \times N$ matrix $Y_N$, an important problem in random matrix theory is to study the limiting distribution of the empirical spectral measure as $N$ tends to infinity. We will use asymptotic notation, such as $O,o,\Omega$, under the assumption that $N \rightarrow \infty$.  See Section \ref{sec:notation} for a complete description of our asymptotic notation.

\subsection{Random matrices with independent entries}

We consider two ensembles of random matrices with independent entries.  We first define a class of Hermitian random matrices with independent entries originally introduced by Wigner \cite{W}.  

\begin{definition}[Wigner random matrices]
Let $\xi$ be a complex random variable with mean zero and unit variance, and let $\zeta$ be a real random variables with mean zero and finite variance.  We say $Y_N$ is a \emph{Wigner matrix} of size $N$ with atom variables $\xi,\zeta$ if $Y_N = (y_{ij})_{i,j=1}^N$ is a random Hermitian $N \times N$ matrix that satisfies the following conditions.
\begin{itemize}
\item $\{y_{ij} : 1 \leq i \leq j \leq N\}$ is a collection of independent random variables.
\item $\{y_{ij} : 1 \leq i < j \leq N\}$ is a collection of independent and identically distributed (iid) copies of $\xi$.
\item $\{y_{ii} : 1 \leq i \leq N\}$ is a collection of iid copies of $\zeta$.  
\end{itemize}

\end{definition}

The prototypical example of a Wigner real symmetric matrix is the \emph{Gaussian orthogonal ensemble} (GOE).  The GOE is defined by the probability distribution  
\begin{equation} \label{eq:distGOE}
	\Prob(d M) = \frac{1}{Z^{(\beta)}_N} \exp\left({-\frac{\beta}{4}\tr{M^2}}\right) d M
\end{equation}
on the space of $N \times N$ real symmetric matrices when $\beta = 1$ and $d M$ refers to the Lebesgue measure on the $N(N+1)/2$ different elements of the matrix.  Here $Z_N^{(\beta)}$ denotes the normalization constant.  So for a matrix $Y_N = (y_{ij})_{i,j=1}^N$ drawn from the GOE, the elements $\{ y_{ij} : 1 \leq i \leq j \leq N \}$ are independent Gaussian random variables with mean zero and variance $1+\delta_{ij}$.  

The classical example of a Wigner Hermitian matrix is the \emph{Gaussian unitary ensemble} (GUE).  The GUE is defined by the probability distribution given in \eqref{eq:distGOE} with $\beta = 2$, but on the space of $N \times N$ Hermitian matrices.  Thus, for a matrix $Y_N = (y_{ij})_{i,j=1}^N$ drawn from the GUE, the $N^2$ different real elements of the matrix, 
\begin{equation*}
	\{\Re (y_{ij}) : 1 \leq i \leq j \leq N \} \cup \{\Im (y_{ij}) : 1 \leq i < j \leq N \},
\end{equation*}
are independent Gaussian random variables with mean zero and variance $(1+\delta_{ij})/2$.

A classical result for Wigner random matrices is Wigner's semicircle law \cite[Theorem 2.5]{BSbook}.  
\begin{theorem}[Wigner's Semicircle law] \label{thm:semicirclelaw}
Let $\xi$ be a complex random variable with mean zero and unit variance, and let $\zeta$ be a real random variables with mean zero and finite variance.  For each $N \geq 1$, let $Y_N$ be a Wigner matrix of size $N$ with atom variables $\xi,\zeta$, and let $A_N$ be a deterministic $N \times N$ Hermitian matrix with rank $o(N)$.  Then the ESD of $\frac{1}{\sqrt{N}}\left(Y_N+A_N\right)$ converges almost surely to the semicircle distribution $F_{\mathrm{sc}}$ as $N \rightarrow \infty$, where 
$$ F_{\mathrm{sc}}(x) := \int_{-\infty}^x \rho_{\mathrm{sc}}(t) dt, \quad \rho_{\mathrm{sc}}(t) :=  \left\{
	\begin{array}{ll}
       		\frac{1}{2\pi} \sqrt{4 - t^2}, & \text{if } |t| \leq 2 \\
       		0, & \text{if } |t| > 2.
     	\end{array}
   	\right. $$
\end{theorem}
\begin{remark}
Wigner's semicircle law holds in the case when the entries of $Y_N$ are not identically distributed (but still independent) provided the entries satisfy a Lindeberg-type condition.  See \cite[Theorem 2.9]{BSbook} for further details.  
\end{remark}

We now consider an ensemble of random matrices with iid entries.  
\begin{definition}[iid random matrices]
Let $\xi$ be a complex random variable.  We say $Y_N$ is an \emph{iid random matrix} of size $N$ with atom variable $\xi$ if $Y_N$ is a $N \times N$ matrix whose entries are iid copies of $\xi$.   
\end{definition}

When $\xi$ is a standard complex Gaussian random variable, $Y_N$ can be viewed as a random matrix drawn from the probability distribution
$$ \Prob(d M) = \frac{1}{\pi^{N^2}} e^{- \tr (M M^\ast)} d M $$
on the set of complex $N \times N$ matrices.  Here $d M$ denotes the Lebesgue measure on the $2N^2$ real entries of $M$.  This is known as the {\it complex Ginibre ensemble}.  The {\it real Ginibre ensemble} is defined analogously.  Following Ginibre \cite{Gi}, one may compute the joint density of the eigenvalues of a random $N \times N$ matrix $Y_N$ drawn from the complex Ginibre ensemble. 

Mehta \cite{M,M:B} used the joint density function obtained by Ginibre to compute the limiting spectral measure of the complex Ginibre ensemble.  In particular, he showed that if $Y_N$ is drawn from the complex Ginibre ensemble, then the ESD of $\frac{1}{\sqrt{N}} Y_N$ converges to the {\it circular law} $F_{\mathrm{circ}}$, where
$$ F_{\mathrm{circ}}(x,y) := \mu_{\mathrm{circ}} \left( \left\{ z \in \mathbb{C} : \Re(z) \leq x, \Im(z) \leq y \right\} \right) $$
and $\mu_{\mathrm{circ}}$ is the uniform probability measure on the unit disk in the complex plane.  Edelman \cite{Ed-cir} verified the same limiting distribution for the real Ginibre ensemble.

For the general (non-Gaussian) case, there is no formula for the joint distribution of the eigenvalues and the problem appears much more difficult.  The universality phenomenon in random matrix theory asserts that the spectral behavior of a random matrix does not depend on the distribution of the atom variable $\xi$ in the limit $N \rightarrow \infty$.  In other words, one expects that the circular law describes the limiting ESD of a large class of random matrices (not just Gaussian matrices)

The first rigorous proof of the circular law for general (non-Gaussian) distributions was by Bai \cite{Bcirc,BSbook}.  He proved the result under a number of assumptions on the moments and smoothness of the atom variable $\xi$.  Important results were obtained more recently by Pan and Zhou \cite{PZ} and G\"otze and Tikhomirov \cite{GTcirc}.  Using techniques from additive combinatorics, Tao and Vu \cite{TVcirc} were able to prove the circular law under the assumption that $\E|\xi|^{2+\eps} < \infty$ for some $\eps > 0$.  Recently, Tao and Vu \cite{TVbull,TVesd} established the law assuming only that $\xi$ has finite variance.  

For any matrix $M$, we denote the Hilbert-Schmidt norm $\|M\|_2$ by the formula
\begin{equation} \label{eq:def:hs}
	\|M\|_2 := \sqrt{ \tr (M M^\ast) } = \sqrt{ \tr (M^\ast M)}. 
\end{equation}

\begin{theorem}[Tao-Vu, \cite{TVesd}] \label{thm:tvcirc}
Let $\xi$ be a complex random variable with mean zero and unit variance.  For each $N \geq 1$, let $Y_N$ be a $N \times N$ matrix whose entries are iid copies of $\xi$, and let $A_N$ be a $N \times N$ deterministic matrix.  If $\rank(A_N) = o(N)$ and $\sup_{N \geq 1} \frac{1}{N^2} \| A_N \|_2^2 < \infty$, then the ESD of $\frac{1}{\sqrt{N}} (Y_N + A_N)$ converges almost surely to the circular law $F_{\mathrm{circ}}$ as $N \rightarrow \infty$.  
\end{theorem}

\subsection{Outliers in the spectrum}

From Theorem \ref{thm:semicirclelaw} and Theorem \ref{thm:tvcirc}, we see that the low rank perturbation $A_N$ does not effect the limiting ESD.  In other words, the majority of the eigenvalues remain distributed according to semicircle law or circular law, respectively.  However, the perturbation $A_N$ may create one or more outliers.  

Let $Y_N$ be a $N \times N$ random matrix whose entries are iid copies of $\xi$.  When the atom variable $\xi$ has finite fourth moment, one can compute the asymptotic behavior of the spectral radius \cite[Thoerem 5.18]{BSbook}.  We remind the reader that the spectral radius of a square matrix is the largest eigenvalue in absolute value.  

\begin{theorem}[No outliers for iid matrices] \label{thm:no:iid}
Let $\xi$ be a complex random variable with mean zero, unit variance, and finite fourth moment.  For each $N \geq 1$, let $Y_N$ be a $N \times N$ random matrix whose entries are iid copies of $\xi$.  Then the spectral radius of $\frac{1}{\sqrt{N}} Y_N$ converges to $1$ almost surely as $N \rightarrow \infty$.  
\end{theorem}

In \cite{Tout}, Tao computes the asymptotic location of the outlier eigenvalues for bounded rank perturbations of iid random matrices.  

\begin{theorem}[Outliers for small low rank perturbations of iid matrices, \cite{Tout}] \label{thm:outlier:iid}
Let $\xi$ be a complex random variable with mean zero, unit variance, and finite fourth moment.  For each $N \geq 1$, let $Y_N$ be a $N \times N$ random matrix whose entries are iid copies of $\xi$, and let $C_N$ be a deterministic matrix with rank $O(1)$.  Let $\eps > 0$, and suppose that for all sufficiently large $N$, there are no eigenvalues of $C_N$ in the band $\{z \in \C : 1 + \eps < |z| < 1 + 3\eps\}$, and there are $j$ eigenvalues $\lambda_1(C_N), \ldots, \lambda_j(C_N)$ for some $j = O(1)$ in the region $\{z \in \C : |z| \geq 1+3 \eps\}$.  Then, almost surely, for sufficiently large $N$, there are precisely $j$ eigenvalues $\lambda_1(\frac{1}{\sqrt{N}} Y_N + C_N), \ldots, \lambda_j(\frac{1}{\sqrt{N}} Y_N + C_N)$ of $\frac{1}{\sqrt{N}}Y_N + C_N$ in the region $\{z \in \C : |z| \geq 1 + 2\eps\}$, and after labeling these eigenvalues properly, 
$$ \lambda_i\left( \frac{1}{\sqrt{N}} Y_N + C_N \right) = \lambda_i(C_N) + o(1) $$
as $N \rightarrow \infty$ for each $1 \leq i \leq j$.  
\end{theorem}

Recently, Benaych-Georges and Rochet \cite{BGR} obtained an analogous result for finite rank perturbations of random matrices whose distributions are invariant under the left and right actions of the unitary group.  Benaych-Georges and Rochet also study the fluctuations of the outlier eigenvalues.  

Similar results have also been obtained for Wigner random matrices.  When the atom variables have finite fourth moment, the asymptotic behavior of the spectral radius can be computed \cite[Theorem 5.2]{BSbook}.

\begin{theorem}[No outliers for Wigner matrices] \label{thm:no:wigner}
Let $\xi$ be a complex random variable with mean zero, unit variance, and finite fourth moment, and let $\zeta$ be a real random variables with mean zero and finite variance.  For each $N \geq 1$, let $Y_N$ be a Wigner matrix of size $N$ with atom variables $\xi,\zeta$.  Then the spectral radius of $\frac{1}{\sqrt{N}} Y_N$ converges to $2$ almost surely as $N \rightarrow \infty$.  
\end{theorem}

The asymptotic location of the outliers for bounded rank perturbations of Wigner matrices and other classes of self adjoint random matrices have also been determined.  In fact, the fluctuations of the outlier eigenvalues can be explicitly computed.  We refer the reader to \cite{BR,BGM,BGM2,CDF1,CDF,CDFF,FP,FK,KY,KY2,P,PRS,PRS2} and references therein for further details.  

\begin{theorem}[Outliers for small low rank perturbations of Wigner matrices, \cite{PRS2}] \label{thm:outlier:wigner}
Let $\xi$ be a real random variable with mean zero, unit variance, and finite fourth moment, and let $\zeta$ be a real random variables with mean zero and finite variance.  For each $N \geq 1$, let $Y_N$ be a Wigner matrix of size $N$ with atom variables $\xi,\zeta$.  Let $k \geq 1$.  For each $N \geq k$, let $C_N$ be a $N \times N$ deterministic Hermitian matrix with rank $k$ and nonzero eigenvalues $\lambda_1(C_N), \ldots, \lambda_k(C_N)$, where $k, \lambda_1(C_N), \ldots, \lambda_k(C_N)$ are independent of $N$.  Let $S = \{1 \leq i \leq k : |\lambda_i(C_N)| > 1\}$.  Then we have the following.
\begin{itemize}
\item For all $i \in S$, after labeling the eigenvalues of $\frac{1}{\sqrt{N}} Y_N+C_N$ properly,
$$ \lambda_i \left( \frac{1}{\sqrt{N}} Y_N + C_N \right) \longrightarrow \lambda_i(C_N) + \frac{1}{\lambda_i(C_N)} $$
in probability as $N \rightarrow \infty$.  
\item For all $i \in \{1,\ldots,k\} \setminus S$, after labeling the eigenvalues of $\frac{1}{\sqrt{N}} Y_N+C_N$ properly,
$$ \left|\lambda_i \left( \frac{1}{\sqrt{N}} Y_N + C_N \right)\right| \longrightarrow  2 $$
in probability as $N \rightarrow \infty$.  
\end{itemize}
\end{theorem}

\begin{remark}
Under additional assumptions on the atom variables $\xi,\zeta$, the convergence in Theorem \ref{thm:outlier:wigner} can be strengthened to almost sure convergence \cite{CDF1}.  
\end{remark}

Non-Hermitian finite rank perturbations of random Hermitian matrices have been studied in the mathematical physics literature.  We refer the reader to \cite{FS,FS2,FKstat} and references therein for further details.

\subsection{Elliptic random matrices}

We consider the following class of random matrices with dependent entries that generalizes the ensembles introduced above.  These so-called elliptic random matrices were originally introduced by Girko \cite{Gorig,Gten}.  

\begin{definition}[Condition {\bf C1}] \label{def:C1}
Let $(\xi_1,\xi_2)$ be a random vector in $\mathbb{R}^2$, where both $\xi_1,\xi_2$ have mean zero and unit variance.  We set $\rho := \E[\xi_1 \xi_2]$.  Let $\{y_{ij}\}_{i,j \geq 1}$ be an infinite double array of real random variables.  For each $N \geq 1$, we define the $N \times N$ random matrix $Y_N = (y_{ij})_{i,j = 1}^N$. We say the sequence of random matrices $\{Y_N\}_{N \geq 1}$ satisfies condition {\bf C1} with atom variables $(\xi_1,\xi_2)$ if the following hold: 
\begin{itemize}
\item $\{ y_{ii} : 1 \leq i \} \cup \{ (y_{ij}, y_{ji}) : 1 \leq i < j \}$ is a collection of independent random elements,
\item $\{ (y_{ij}, y_{ji}) : 1 \leq i < j \}$ is a collection of iid copies of $(\xi_1, \xi_2)$,  
\item $\{ y_{ii} : 1 \leq i \}$ is a collection of iid random variables with mean zero and finite variance.  
\end{itemize}
\end{definition}

\begin{remark}
Let $\{Y_N\}_{N \geq 1}$ be a sequence of random matrices that satisfy condition {\bf C1} with atom variables $(\xi_1,\xi_2)$.  If $\rho := \E[\xi_1 \xi_2] = 1$, then $\{Y_N\}_{N \geq 1}$ is a sequence of Wigner real symmetric matrices.  
\end{remark}
\begin{remark}
Let $\xi$ be a real random variable with mean zero and unit variance.  For each $N \geq 1$, let $Y_N$ be a $N \times N$ random matrix whose entries are iid copies of $\xi$.  Then $\{Y_N\}_{N \geq 1}$ is a sequence of random matrices that satisfy condition {\bf C1}.  
\end{remark}

If $\{Y_N\}_{N \geq 1}$ is a sequence of random matrices that satisfy condition {\bf C1}, then it was shown in \cite{NgO} that the limiting ESD of $\frac{1}{\sqrt{N}} Y_N$ is given by the uniform distribution on the interior of an ellipse.  The same conclusion was shown to hold by Naumov \cite{Nell} for elliptic random matrices whose atom variables satisfy additional moment assumptions.  

For $-1 < \rho < 1$, define the ellipsoid
\begin{equation} \label{eq:def:ellipsoid}
	\mathcal{E}_\rho := \left \{ z \in \C : \frac{\Re(z)^2}{(1+\rho)^2} + \frac{ \Im(z)^2 }{ (1-\rho)^2} \leq 1 \right \}. 
\end{equation}
Let 
$$ F_{\rho}(x,y) := \mu_{\rho} \left( \{ z \in \C : \Re(z) \leq x, \Im(z) \leq y \} \right), $$ 
where $\mu_\rho$ is the uniformly probability measure on $\mathcal{E}_\rho$.  
It will also be convenient to define $\mathcal{E}_\rho$ when $\rho = \pm 1$.  For $\rho = 1$, let $\mathcal{E}_1$ be the line segment $[-2, 2]$, and for $\rho=-1$, we let $\mathcal{E}_{-1}$ be the line segment $[-2,2]\sqrt{-1}$ on the imaginary axis\footnote{We use $\sqrt{-1}$ to denote the imaginary unit and reserve $i$ as an index.}.  

\begin{theorem} [Elliptic law, \cite{NgO}] \label{thm:ellipticlaw}
Let $\{Y_N\}_{N \geq 1}$ be a sequence of random matrices that satisfies condition {\bf C1} with atom variables $(\xi_1, \xi_2)$, where $\rho = \E[\xi_1 \xi_2]$, and assume $-1 < \rho < 1$.  For each $N \geq 1$, let $A_N$ be a $N \times N$ matrix, and assume the sequence $\{A_N\}_{N \geq 1}$ satisfies $\rank(A_N) = o(N)$ and $\sup_{N \geq 1} \frac{1}{N^2} \|A_N\|_2^2 < \infty$.  Then the ESD of $\frac{1}{\sqrt{N}}( Y_N + A_N)$ converges almost surely to $F_\rho$ as $N \rightarrow \infty$.  
\end{theorem}

\begin{remark}
A version of Theorem \ref{thm:ellipticlaw} holds when $\xi_1,\xi_2$ are complex random variables \cite{NgO}.  However, this note will only focus on real elliptic random matrices.  
\end{remark}

\section{Main results} \label{sec:mainresults}

In this note, we consider the outliers of perturbed elliptic random matrices.  In particular, we consider versions of Theorem \ref{thm:no:iid}, Theorem \ref{thm:outlier:iid}, Theorem \ref{thm:no:wigner}, and Theorem \ref{thm:outlier:wigner} for elliptic random matrices whose entries have finite fourth moment.  

\begin{definition}[Condition {\bf C0}] \label{def:C0}
Let $(\xi_1,\xi_2)$ be a random vector in $\mathbb{R}^2$, where both $\xi_1,\xi_2$ have mean zero and unit variance.  We set $\rho := \E[\xi_1 \xi_2]$.  For each $N \geq 1$, let $Y_N$ be a $N \times N$ random matrix.  We say the sequence of random matrices $\{Y_N\}_{N \geq 1}$ satisfies condition {\bf C0} with atom variables $(\xi_1,\xi_2)$ if the following conditions hold: 
\begin{itemize}
\item The sequence $\{Y_N\}_{N \geq 1}$ satisfies condition {\bf C1} with atom variables $(\xi_1, \xi_2)$,
\item We have
$$ M_4 := \max\{ \E|\xi_1|^4, \E|\xi_2|^4 \} < \infty. $$
\end{itemize}
\end{definition}

We will also define the neighborhoods 
$$ \mathcal{E}_{\rho,\delta} := \left\{ z \in \mathbb{C} : \dist(z,\mathcal{E}_{\rho}) \leq \delta \right\} $$
for any $\delta > 0$.  

We first consider a version of Theorem \ref{thm:no:iid} and Theorem \ref{thm:no:wigner} for elliptic random matrices.  Because of the elliptic shape of the limiting ESD, it is not enough to just consider the spectral radius.  

\begin{theorem}[No outliers for elliptic random matrices] \label{thm:nooutliers}
Let $\{Y_N\}_{N \geq 1}$ be a sequence of random matrices that satisfies condition {\bf C0} with atom variables $(\xi_1, \xi_2)$, where $\rho = \E[\xi_1 \xi_2]$.  Let $\delta > 0$.  Then, almost surely, for $N$ sufficiently large, all the eigenvalues of $\frac{1}{\sqrt{N}} Y_N$ are contained in $\mathcal{E}_{\rho, \delta}$.  
\end{theorem}

Theorem \ref{thm:ellipticlaw} and Theorem \ref{thm:nooutliers} immediately imply the following asymptotic behavior for the spectral radius of elliptic random matrices.  

\begin{corollary}[Spectral radius of elliptic random matrices]
Let $\{Y_N\}_{N \geq 1}$ be a sequence of random matrices that satisfies condition {\bf C0} with atom variables $(\xi_1, \xi_2)$, where $\rho = \E[\xi_1 \xi_2]$.  Then the spectral radius of $\frac{1}{\sqrt{N}} Y_N$ converges almost surely to $1+|\rho|$ as $N \rightarrow \infty$.
\end{corollary}

We now consider the analogue of Theorem \ref{thm:outlier:iid} and Theorem \ref{thm:outlier:wigner} for elliptic random matrices.  Figure \ref{fig:ellipticpert} shows an eigenvalue plot of a perturbed elliptic random matrix as well as the location of the outlier eigenvalues predicted by the following theorem.

\begin{figure} 
	\begin{center}
	\includegraphics[width=7cm]{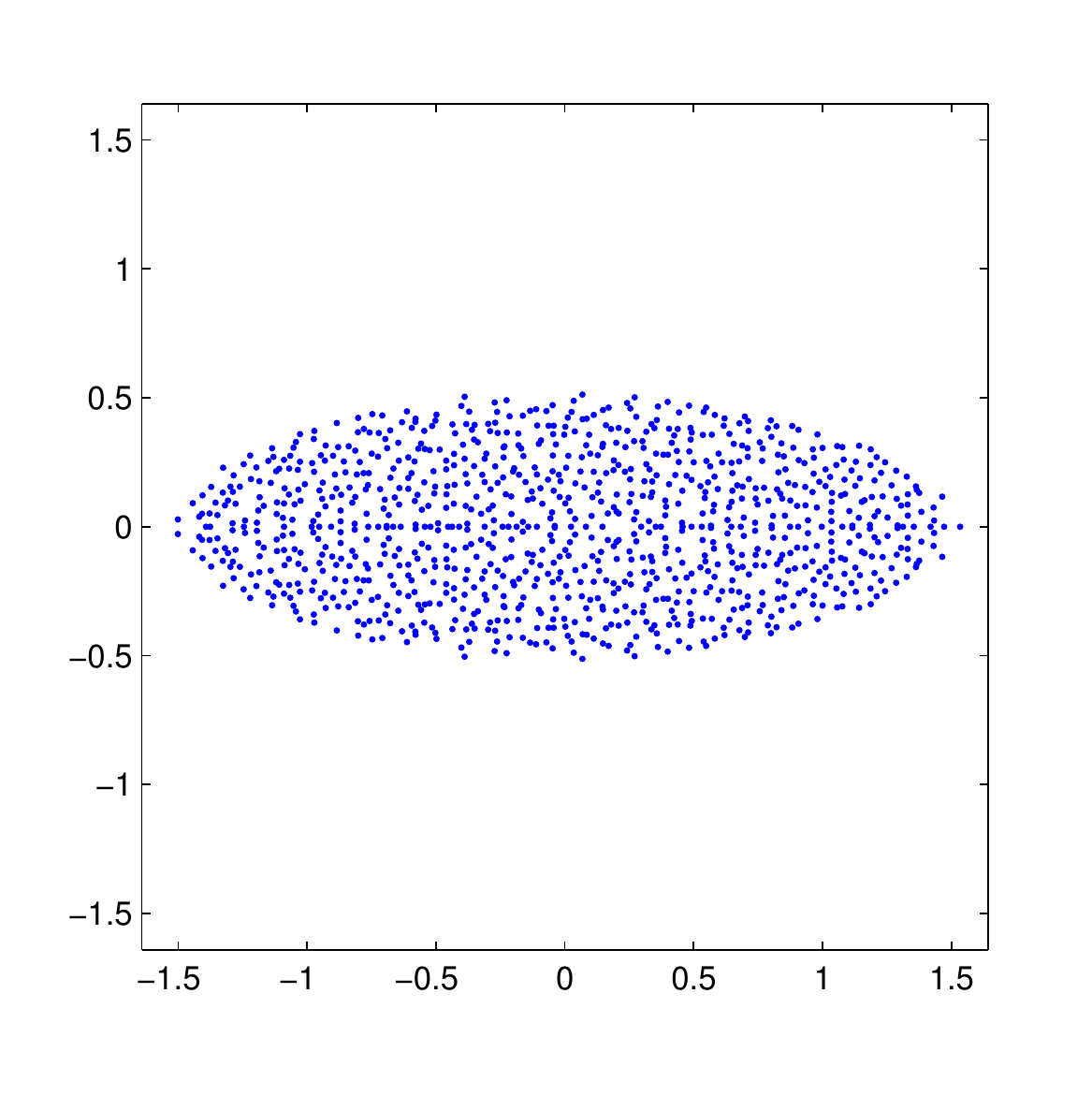}
  \includegraphics[width=9cm]{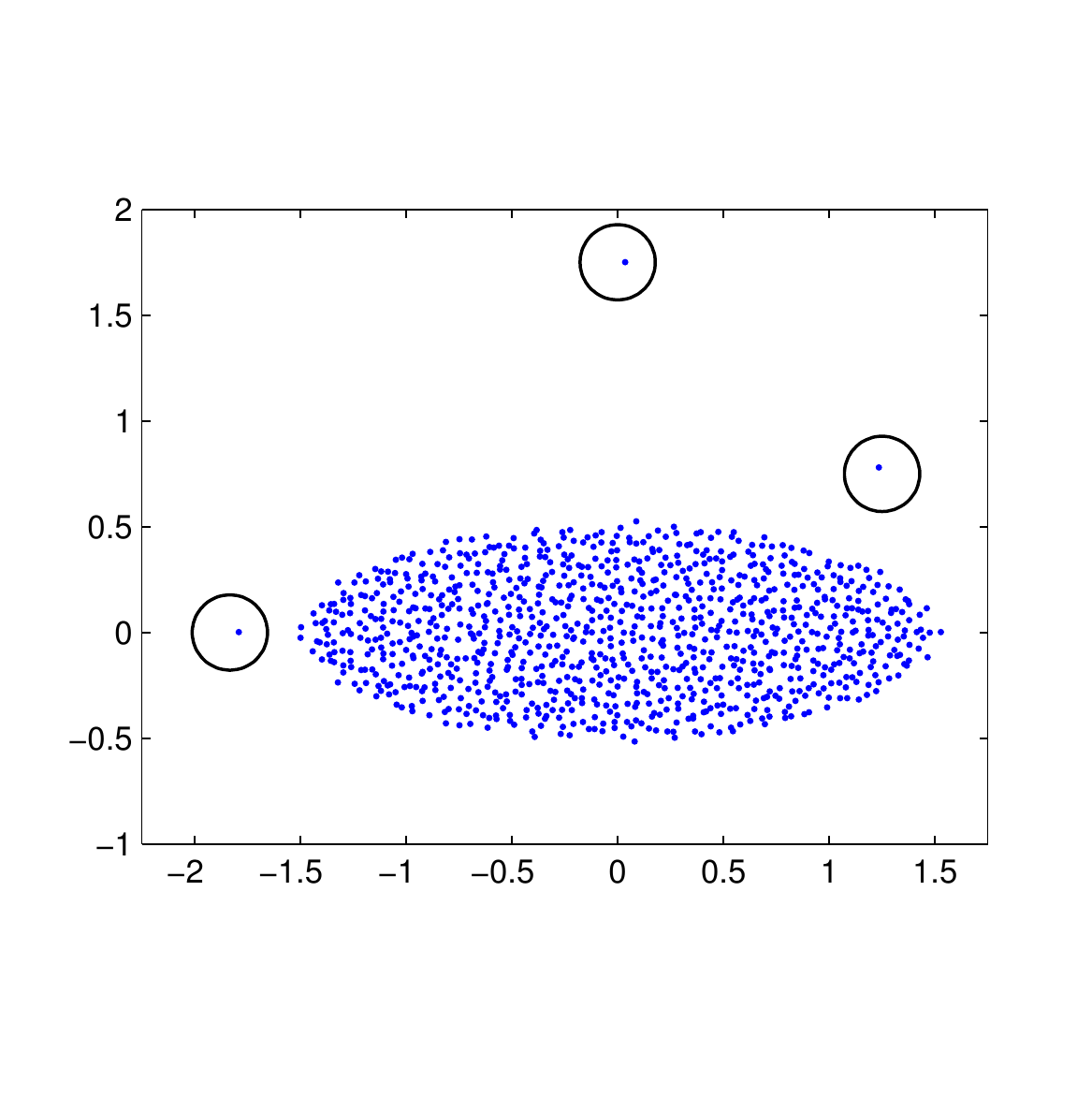}
   	\caption{The plot on top shows the eigenvalues of a single $N \times N$ elliptic random matrix with Gaussian entries when $N=1000$ and $\rho=1/2$.  The plot on bottom shows the eigenvalues of the same elliptic matrix after perturbing it by a diagonal matrix with three nonzero eigenvalues: $2\sqrt{-1}, -\frac{3}{2}$, and $1+\sqrt{-1}$.  The three circles are centered at $\frac{7}{4}\sqrt{-1}, -\frac{11}{6}$, and $\frac{5}{4}+\frac{3}{4}\sqrt{-1}$, respectively, and each have radius $N^{-1/4}$.}
   	\label{fig:ellipticpert}
 	\end{center}
\end{figure}

\begin{theorem}[Outliers for low rank perturbations of elliptic random matrices] \label{thm:main}
Let $k \geq 1$ and $\delta > 0$.  Let $\{Y_N\}_{N \geq 1}$ be a sequence of random matrices that satisfies condition {\bf C0} with atom variables $(\xi_1, \xi_2)$, where $\rho = \E[\xi_1 \xi_2]$.  For each $N \geq 1$, let $C_N$ be a deterministic $N \times N$ matrix, where $\sup_{N \geq 1} \rank(C_N) \leq k$ and $\sup_{N \geq 1} \| C_N \| = O(1)$.  Suppose for $N$ sufficiently large, there are no nonzero eigenvalues of $C_N$ which satisfy 
\begin{equation} \label{eq:reqclose}
	\lambda_i(C_N) + \frac{\rho}{\lambda_i(C_N)} \in \mathcal{E}_{\rho,3\delta} \setminus \mathcal{E}_{\rho,\delta} \quad \text{with} \quad |\lambda_i(C_N)| > 1, 
\end{equation}
and there are $j$ eigenvalues $\lambda_1(C_N), \ldots, \lambda_j(C_N)$ for some $j \leq k$ which satisfy 
$$ \lambda_i(C_N) + \frac{\rho}{\lambda_i(C_N)} \in \C \setminus \mathcal{E}_{\rho,3\delta}  \quad \text{with} \quad |\lambda_i(C_N)| > 1. $$
Then, almost surely, for $N$ sufficiently large, there are exactly $j$ eigenvalues of $\frac{1}{\sqrt{N}} Y_N + C_N$ in the region $\C \setminus \mathcal{E}_{\rho,2\delta}$, and after labeling the eigenvalues properly, 
$$ \lambda_i \left( \frac{1}{\sqrt{N}} Y_N + C_N \right) = \lambda_i(C_N) + \frac{\rho}{\lambda_i(C_N)} + o(1) $$
for each $1 \leq i \leq j$.  
\end{theorem}

\begin{remark}
Theorem \ref{thm:main} generalizes the results of both Theorem \ref{thm:outlier:iid} and Theorem \ref{thm:outlier:wigner}.  Indeed, if $\rho=1$, then $\{Y_N\}_{N \geq 1}$ is a sequence of Wigner real symmetric matrices.  In this case, Theorem \ref{thm:main} implies the almost sure convergence of the outlier eigenvalues to the locations described by Theorem \ref{thm:outlier:wigner}.  Additionally, Theorem \ref{thm:main} also deals with the case when $C_N$ is non-Hermitian.  On the other hand, if $\{Y_N\}_{N \geq 1}$ is a sequence of random matrices whose entries are iid random variables, then $\rho = 0$, and Theorem \ref{thm:main} gives precisely the results of Theorem \ref{thm:outlier:iid}.  
\end{remark}

\begin{remark}
In \cite{CDF1, CDFF}, Capitaine, Donati-Martin, F\'eral, and F\'evrier consider spiked deformations of Wigner random matrices plus deterministic matrices.  Theorem \ref{thm:main} can be viewed as a non-Hermitian extension of the results in \cite{CDF1,CDFF}.  Indeed, the subordination functions appearing in \cite{CDFF} appears very naturally in our analysis; see Remark \ref{rem:subord} for further details.  
\end{remark}

\begin{remark}
Theorem \ref{thm:main} requires that there are no nonzero eigenvalues of $C_N$ which satisfy \eqref{eq:reqclose}. Since $\delta$ is arbitrary, if the eigenvalues of $C_N$ do not change with $N$, this condition can be ignored. This condition is analogous to the requirements of Theorem \ref{thm:outlier:iid}.  Indeed, Theorem \ref{thm:outlier:iid} requires that there are no eigenvalues of $C_N$ in the band $\{z \in \mathbb{C} : 1+\eps < |z| < 1+3\eps\}$.
\end{remark}

We now consider the case of elliptic random matrices with nonzero mean, which we write as $\frac{1}{\sqrt{N}} Y_N + \mu \sqrt{N} \varphi_N \varphi_N^\ast$, where $\{Y_N\}_{N \geq 1}$ is a sequence of random matrices that satisfies condition {\bf C0} with atom variables $(\xi_1,\xi_2)$, $\mu$ is a fixed nonzero complex number (independent of $N$), and $\varphi_N$ is the unit vector $\varphi_N := \frac{1}{\sqrt{N}} (1,\ldots,1)^\ast$.  This corresponds to shifting the entries of $Y_N$ by $\mu$ (so they have mean $\mu$ instead of mean zero).  The elliptic law still holds for this rank one perturbation of $\frac{1}{\sqrt{N}} Y_N$, thanks to Theorem \ref{thm:ellipticlaw}.  In view of Theorem \ref{thm:main}, we show there is a single outlier for this ensemble near $\mu \sqrt{N}$.  

\begin{theorem}[Outlier for elliptic random matrices with nonzero mean] \label{thm:mean}
Let $\delta > 0$.  Let $\{Y_N\}_{N \geq 1}$ be a sequence of random matrices that satisfies condition {\bf C0} with atom variables $(\xi_1, \xi_2)$, where $\rho = \E[\xi_1 \xi_2]$, and let $\mu$ be a nonzero complex number independent of $N$.  Then almost surely, for sufficiently large $N$, all the eigenvalues of $\frac{1}{\sqrt{N}} Y_N + \mu \sqrt{N} \varphi_N \varphi_N^\ast$ lie in $\mathcal{E}_{\rho,\delta}$, with a single exception taking the value $\mu \sqrt{N} + o(1)$.  
\end{theorem}

\begin{remark}
A version of Theorem \ref{thm:mean} was proven by F\"uredi and Koml\'os in \cite{FK} for a class of real symmetric Wigner matrices.  Moreover,  F\"uredi and Koml\'os study the fluctuations of the outlier eigenvalue.  Tao \cite{Tout} verified Theorem \ref{thm:mean} when $Y_N$ is a random matrix with iid entries.  
\end{remark}

One of the keys to proving Theorem \ref{thm:main} and Theorem \ref{thm:mean} is to control the least singular value of a perturbed elliptic random matrix.  Let $M$ be a $N \times N$ matrix.  The \emph{singular values} of $M$ are the eigenvalues of $|M| := \sqrt{MM^\ast}$.  We let $\sigma_1(M) \geq \cdots \geq \sigma_N(M) \geq 0$ denote the singular values of $M$.  In particular, the largest and smallest singular values are
$$ \sigma_1(M) := \sup_{\|x\| = 1} \| M x \|, \qquad \sigma_N(M) := \inf_{\|x\| = 1} \| M x \|, $$
where $\|x\|$ denotes the Euclidian norm of the vector $x$.  We let $\|M \|$ denote the spectral norm of $M$.  It follows that the largest and smallest singular values can be written in terms of the spectral norm.  Indeed, $\sigma_1(M) = \|M\|$ and $\sigma_N(M) = 1/\|M^{-1} \|$ provided $M$ is invertible.  

We now consider a lower bound for the least singular value of perturbed elliptic random matrices of the form $\frac{1}{\sqrt{N}} Y_N - z I$, where $I$ denotes the identity matrix.  A lower bound of the form
$$ \sigma_N \left( \frac{1}{\sqrt{N}}Y_N - z I \right) \geq N^{-A}, $$
for some $A > 0$, was shown to hold with high probability in \cite{Nell,NgO}.  Below, we consider only the case when $z$ is outside the ellipse $\mathcal{E}_\rho$ and thus obtain a constant lower bound independent of $N$.  

\begin{theorem}[Least singular value bound] \label{thm:least}
Let $\{Y_N\}_{N \geq 1}$ be a sequence of random matrices that satisfies condition {\bf C0} with atom variables $(\xi_1, \xi_2)$, where $\rho = \E[\xi_1 \xi_2]$.  Let $\delta > 0$.  Then there exists $c>0$ such that almost surely, for $N$ sufficiently large,
$$ \inf_{\dist(z, \mathcal{E}_\rho) \geq \delta} \sigma_N \left( \frac{1}{\sqrt{N}}  Y_N - z I \right) \geq c. $$
\end{theorem}

In fact, Theorem \ref{thm:nooutliers} follows immediately from Theorem \ref{thm:least}.  

\begin{proof}[Proof of Theorem \ref{thm:nooutliers}]
We note that $z$ is an eigenvalue of $\frac{1}{\sqrt{N}}Y_N$ if and only if 
$$ \det \left( \frac{1}{\sqrt{N}} Y_N - z I \right) = 0. $$
On the other hand, 
$$ \left| \det \left( \frac{1}{\sqrt{N}} Y_N - z I \right) \right| = \prod_{i=1}^N \sigma_i \left( \frac{1}{\sqrt{N}} Y_N - zI \right). $$
Thus, we conclude that $z$ is an eigenvalue of $\frac{1}{\sqrt{N}} Y_N$ if and only if
$$ \sigma_N \left( \frac{1}{\sqrt{N}} Y_N - z I \right) = 0. $$
The claim therefore follows from Theorem \ref{thm:least}.  
\end{proof}

The condition number $\sigma_1(M)/\sigma_N(M)$ of a $N \times N$ matrix $M$ plays an important role in numerical linear algebra (see for example \cite{BT}).  As a consequence of Theorem \ref{thm:least}, we obtain the following bound for the condition number of perturbed elliptic random matrices that satisfy condition {\bf C0}.

\begin{corollary}[Condition number bound] \label{cor:condition}
Let $\{Y_N\}_{N \geq 1}$ be a sequence of random matrices that satisfies condition {\bf C0} with atom variables $(\xi_1, \xi_2)$, where $\rho = \E[\xi_1 \xi_2]$.  Fix $z \notin \mathcal{E}_\rho$.  Then there exists $C>0$ (depending on $z$) such that almost surely, for $N$ sufficiently large,
$$ \frac{\sigma_1 \left( \frac{1}{\sqrt{N}} Y_N - zI \right)}{\sigma_N( \frac{1}{\sqrt{N}} Y_N - z I ) } \leq C. $$
\end{corollary}
\begin{proof}
In view of Theorem \ref{thm:least}, it suffices to show that almost surely 
$$ \sigma_1 \left( \frac{1}{\sqrt{N}} Y_N - zI \right) \leq C $$
for $N$ sufficiently large.  Since 
$$ \sigma_1 \left( \frac{1}{\sqrt{N}} Y_N - zI \right) = \left \|  \frac{1}{\sqrt{N}} Y_N - zI  \right \| \leq \frac{1}{\sqrt{N}} \| Y_N \| + |z|, $$
it suffices to show that almost surely
$$ \frac{1}{\sqrt{N}} \| Y_N \| \leq C $$
for $N$ sufficiently large.  The claim now follow from Lemma \ref{lemma:normbnd} below.  Indeed, the bound on the spectral norm of $Y_N$ has previously been obtained in \cite{Nell} and follows from \cite[Theorem 5.2]{BSbook}.  
\end{proof}

\subsection{Overview}

In order to prove Theorem \ref{thm:main}, we will make use of Sylvester's determinant identity: 
\begin{equation} \label{eq:sylvdet}
	\det( I + AB) = \det( I + BA),
\end{equation}
where $A$ is a $N \times k$ matrix and $B$ is a $k \times N$ matrix.  In particular, the left-hand side of \eqref{eq:sylvdet} is a $N \times N$ determinant, while the right-hand side is a $k \times k$ determinant.  

To outline the main idea, which is based on the arguments of Benaych-Georges and Rao \cite{BR}, consider the rank one perturbation $C_N = v u^\ast$.  In order to study the outlier eigenvalues, we will need to solve the equation
\begin{equation} \label{eq:detsolvez}
	\det \left( \frac{1}{\sqrt{N}} Y_N + C_N - z I \right) = 0
\end{equation}
for $z \notin \mathcal{E}_\rho$.  Assume $z$ is not an eigenvalue of $\frac{1}{\sqrt{N}} Y_N$, then we can rewrite \eqref{eq:detsolvez} as
$$ \det\left( I + \left( \frac{1}{\sqrt{N}} Y_N - zI \right)^{-1} C_N \right) = 0. $$
From \eqref{eq:sylvdet}, we find that this is equivalent to solving 
$$ 1 + u^\ast \left( \frac{1}{\sqrt{N}} Y_N - z I \right)^{-1}v = 0. $$
Thus, the problem of locating the outlier eigenvalues reduces to studying the resolvent 
$$ G_N(z) := \left( \frac{1}{\sqrt{N}} Y_N - z I \right)^{-1}. $$
We develop an isotropic limit law in Section \ref{sec:main} to compute the limit of $u^\ast G_N v$; this limit law is inspired by the isotropic semicircle law developed by Knowles and Yin \cite{KY,KY2} for Wigner random matrices. Namely in Theorem \ref{thm:mainuv} we show that not only does the trace of $G_N(z)$ almost surely converge to some function $m(z)$ (defined in \eqref{eq:def:mz}) but arbitrary bilinear forms $u^\ast G_N v$ almost surely converge to $m(z) u^\ast v$. 

However, instead of working with $G_N$ directly, it will often be more convenient to work with the $2N \times 2N$ Hermitian matrix\footnote{Actually, for notational convenience we will work with $ \Xi_N$ conjugated by a permutation matrix (see Section \ref{sec:hermitization} for complete details).}
$$ \Xi_N := \begin{bmatrix} 0 & \frac{1}{\sqrt{N}} Y_N - z I \\ \frac{1}{\sqrt{N}} Y_N^\ast - \bar{z} I & 0 \end{bmatrix} $$
and its resolvent $ ( \Xi_N - \eta I)^{-1}$.  In fact, the eigenvalues of $\Xi_N$ are given by the singular values 
$$ \pm \sigma_1\left( \frac{1}{\sqrt{N}} Y_N - z I \right), \ldots,  \pm \sigma_N\left( \frac{1}{\sqrt{N}} Y_N - z I \right). $$
Thus, for $\Im(\eta) > 0$, the matrix $\Xi_N - \eta I$ is always invertible.  Moreover, when $\eta = 0$, the resolvent becomes 
$$ \begin{bmatrix} 0 & \left( \frac{1}{\sqrt{N}} Y_N^\ast - \bar{z} I \right)^{-1} \\ \left( \frac{1}{\sqrt{N}}Y_N - z I \right)^{-1} & 0 \end{bmatrix}. $$
In other words, we can recover $G_N$ by letting $\eta$ tend to zero.  Similarly, we will bound the least singular value of $\frac{1}{\sqrt{N}} Y_N - z I$ and prove Theorem \ref{thm:least} by studying the eigenvalues of the resolvent $(\Xi_N - \eta I)^{-1}$ when $\Im(\eta)=N^{-\beta}$ for some $\beta > 0$.

The paper is organized as follows.  We present our preliminary tools in Section \ref{sec:tools} and Section \ref{sec:stability}.  In particular, Section \ref{sec:tools} contains a standard truncation lemma; in Section \ref{sec:stability}, we study the stability of a fixed point equation which will determine the asymptotic behavior of the diagonal entries of $G_N$.  In Section \ref{sec:main}, we apply the truncation lemma from Section \ref{sec:tools} to reduce both Theorem \ref{thm:main} and Theorem \ref{thm:least} to the case where we only need to consider elliptic random matrices whose entries are bounded.  We also introduce an isotropic limit law for $G_N$ and prove Theorem \ref{thm:mean} in Section \ref{sec:main}.  Finally, we complete the proof of Theorem \ref{thm:least} in Section \ref{sec:least} and complete the proof of Theorem \ref{thm:main} in Section \ref{sec:mainuv}.  

A number of auxiliary proofs and results are contained in the appendix.  Appendix \ref{app:truncation} contains a somewhat standard proof of the truncation lemma from Section \ref{sec:tools}.  Appendix \ref{section:lde} contains a large deviation estimate for bilinear forms.  In Appendix \ref{app:prop}, we study some additional properties of a limiting spectral measure which was analyzed in \cite{NgO}.

\subsection{Notation} \label{sec:notation}

We use asymptotic notation (such as $O,o,\Omega$) under the assumption that $N \rightarrow \infty$.  We use $X \ll Y, Y \gg X, Y=\Omega(X)$, or $X = O(Y)$ to denote the bound $X \leq CY$ for all sufficiently large $N$ and for some constant $C$.  Notations such as $X \ll_k Y$ and $X=O_k(Y)$ mean that the hidden constant $C$ depends on another constant $k$.  $X=o(Y)$ or $Y=\omega(X)$ means that $X/Y \rightarrow 0$ as $N \rightarrow \infty$.  

An event $E$, which depends on $N$, is said to hold with \emph{overwhelming probability} if $\Prob(E) \geq 1 - O_C(N^{-C})$ for every constant $C > 0$.  We let $\oindicator{E}$ denote the indicator function of the event $E$.  $E^{\stcomp}$ denotes the complement of the event $E$.  

We let $\|M\|$ denote the spectral norm of $M$.  $\|M\|_2$ denotes the Hilbert-Schmidt norm of $M$ (defined in \eqref{eq:def:hs}).  We let $I_N$ denote the $N \times N$ identity matrix.  Often we will just write $I$ for the identity matrix when the size can be deduced from the context.  For a square matrix $M$, we let $\tr_N M := \frac{1}{N} \tr M$.  

We write a.s., a.a., and a.e. for almost surely, Lebesgue almost all, and Lebesgue almost everywhere respectively.  We use $\sqrt{-1}$ to denote the imaginary unit and reserve $i$ as an index.

We let $C$ and $K$ denote constants that are non-random and may take on different values from one appearance to the next.  The notation $K_p$ means that the constant $K$ depends on another parameter $p$.  

\subsection*{Acknowledgments}
The authors are grateful to Alexander Soshnikov for many useful discussions and Yan Fyodorov for references.  They are particularly thankful to Terry Tao for helpful discussions and enthusiastic encouragement.  The authors would also like to thank the anonymous referees for their valuable comments and corrections.

\section{Preliminary tools and notation} \label{sec:tools}

In this section, we consider a number of tools we will need to prove our main results.  We also introduce some new notation, which we will use throughout the paper.

Let $\{Y_N\}_{N \geq 1}$ be a sequence of random matrices that satisfies condition {\bf C0} with atom variables $(\xi_1,\xi_2)$.  We will work with the resolvent $G_N$ defined by
\begin{equation} \label{eq:def:G_N}
	G_N = G_N(z) := \left( \frac{1}{\sqrt{N}} Y_N - z I \right)^{-1}. 
\end{equation}
and it's trace, denoted $m_N(z)$
\begin{equation} \label{eq:def:mN}
	m_N(z) := \frac{1}{N} \tr G_N(z). 
\end{equation}
In order to work with the resolvent, we will need control of the spectral norm $\| G_N \|$.  We bound the spectral norm of $G_N$ for $z$ sufficiently large by bounding the spectral norm of $\frac{1}{\sqrt{N}} Y_N$ in the next subsection.  

When working with $G_N$, we will take advantage of the following well known resolvent identity: for any invertible $N \times N$ matrices $A$ and $B$,
\begin{equation} \label{eq:generalresolventid}
	A^{-1} - B^{-1} = A^{-1} (B - A) B^{-1}. 
\end{equation}
%
%

Suppose $A$ is an invertible square matrix.  Let $u,v$ be vectors.  If $1 + v^\ast A^{-1} u \neq 0$, from \eqref{eq:generalresolventid} one can deduce the Sherman--Morrison rank one perturbation formula (see \cite[Section 0.7.4]{HJ2}): 
\begin{equation} \label{eq:rank1}
	(A + uv^\ast)^{-1} = A^{-1} - \frac{ A^{-1} uv^\ast A^{-1}}{1 + v^\ast A^{-1} u } 
\end{equation}
and
\begin{equation} \label{eq:rank1vec}
	(A + uv^\ast)^{-1}u = \frac{ A^{-1} u}{1 + v^\ast A^{-1} u}.
\end{equation}

From \cite[Section 0.7.3]{HJ2}, we obtain the inverse of a block matrix and Schur's complement:
\begin{equation}
\label{Schur1}
	\begin{bmatrix} A & B \\ C & D \end{bmatrix}^{-1} = \begin{bmatrix} (A-BD^{-1}C)^{-1} & -(A - BD^{-1}C)^{-1} BD^{-1} \\ -D^{-1} C(A-BD^{-1}C)^{-1} & D^{-1} + D^{-1} C(A-BD^{-1}C)^{-1} BD^{-1} \end{bmatrix},
\end{equation}
where $A,B,C,D$ are matrix sub-blocks and $D,A-BD^{-1}C$ are non-singular.  In the case that $A,D-CA^{-1}B$ are invertible, we obtain
\begin{equation*}
\label{Schur2}
\begin{bmatrix} A & B \\ C & D \end{bmatrix}^{-1} = \begin{bmatrix} A^{-1} + A^{-1}B(D-CA^{-1}B)^{-1}CA^{-1} & -A^{-1}B(D-CA^{-1}B)^{-1} \\ -(D-CA^{-1}B)^{-1} CA^{-1} & (D-CA^{-1}B)^{-1} \end{bmatrix}. 
\end{equation*}

It follows from the block matrix inversion formula that
\begin{equation}
\label{trident}
 \tr \begin{pmatrix} A & B \\ C & D \end{pmatrix}^{-1} = \tr( A^{-1} ) + \tr (( D - C A^{-1} B)^{-1} (I + C A^{-2} B)) 
\end{equation}
provided $A, D-CA^{-1}B$ are invertible.

\subsection{Bounds on the spectral norm}

We begin with the following deterministic bound.

\begin{lemma}[Spectral norm of the resolvent for large $|z|$]  \label{lemma:detresolventbnd}
Let $M$ be a $N \times N$ matrix that satisfies $\|M\| \leq \mathcal{K}$.  Then
$$ \left\| \left( M - z I \right)^{-1} \right\| \leq \frac{1}{\eps} $$
for all $z \in \C$ with $|z| \geq \mathcal{K} + \eps$.  
\end{lemma}
\begin{proof}
By writing out the Neumann series, we obtain
\begin{align*}
	\left\| \left( M - z I \right)^{-1} \right\| \leq \left\| \frac{1}{z} + \frac{1}{z} \sum_{k=1}^\infty \frac{ M^k }{z^k} \right\| \leq \frac{1}{\mathcal{K} + \eps} \sum_{k=0}^\infty \left(\frac{\mathcal{K}}{\mathcal{K} + \eps}\right)^k \leq \frac{1}{\eps}
\end{align*}
for $|z| \geq \mathcal{K} + \eps$.  
\end{proof}

\begin{remark}
If $H$ is a Hermitian matrix, we have 
\begin{equation}
\label{symresest}
\|(H-zI)^{-1}\| \leq |\Im(z)|^{-1},
\end{equation}
provided $\Im(z) \neq 0$.  
\end{remark}

We will use the following estimate for the spectral norm.  We note that the bound in Lemma \ref{lemma:normbnd} below is not sharp, but will suffice for our purposes.  

\begin{lemma}[Spectral norm bound] \label{lemma:normbnd}
Let $\{Y_N\}_{N \geq 1}$ be a sequence of random matrices that satisfies condition {\bf C0} with atom variables $(\xi_1, \xi_2)$.  Then a.s.
\begin{equation} \label{eq:normbnd}
	\limsup_{N \rightarrow \infty} \left\| \frac{1}{\sqrt{N}} Y_N \right\| \leq 4. 
\end{equation}
\end{lemma}
\begin{proof}
We write
$$ Y_N = \frac{ Y_N + Y_N^\ast} {2} + \sqrt{-1} \frac{Y_N - Y_N^\ast }{2\sqrt{-1}} $$
and hence
\begin{equation} \label{eq:normsplit}
	\|Y_N\| \leq \left\| \frac{ Y_N + Y_N^\ast} {2} \right\| + \left\|  \frac{Y_N - Y_N^\ast }{2\sqrt{-1}} \right\|. 
\end{equation}
We observe that $\frac{ Y_N + Y_N^\ast} {2}$ and $\frac{Y_N- Y_N^\ast }{2\sqrt{-1}}$ are both Hermitian random matrices.  

Consider the matrix $\frac{ Y_N + Y_N^\ast} {2}$.  By assumption, the diagonal entries of the matrix have mean zero and finite variance.  The above-diagonal entries are iid copies of $\frac{\xi_1 + \xi_2}{2}$.  Thus the above-diagonal entries have mean zero and variance
$$ \frac{1}{4} \E|\xi_1 + \xi_2|^2 \leq \frac{1}{2} (\E|\xi_1|^2 + \E|\xi_2|^2) \leq 1. $$
Moreover, the above-diagonal entries have finite fourth moment:
$$ \E\left| \frac{\xi_1 + \xi_2}{2} \right|^4 \leq \E|\xi_1|^4 + \E|\xi_2|^4 \leq 2 M_4 < \infty. $$
By \cite[Theorem 5.2]{BSbook}, we obtain a.s.
$$ \limsup_{N \rightarrow \infty} \left\| \frac{Y_N + Y_N^\ast}{2 \sqrt{N}} \right\| \leq 2. $$
Similarly, we have a.s.
$$ \limsup_{N \rightarrow \infty} \left\| \frac{Y_N - Y_N^\ast}{2 \sqrt{N}} \right\| \leq 2. $$
The claim follows from the bounds above and \eqref{eq:normsplit}.  
\end{proof}

\subsection{Hermitization} \label{sec:hermitization}

In order to study the spectrum of a non-normal matrix it is often useful to instead consider the spectrum of a family of Hermitian matrices.

We define the \emph{Hermitization} of an $N \times N$ matrix $X$ to be an $N \times N$ matrix with entries that are  $2 \times 2$ block matrices. The $ij^{th}$ entry is the $2 \times 2$ block:
\[ \begin{pmatrix} 0 & X_{ij} \\ \overline{X}_{ji}& 0   \end{pmatrix} \]
We note the Hermitization of $X$ can be conjugated by a $2N \times 2N$ permutation matrix to
\[\begin{pmatrix} 0 & X  \\ X^*  &  0 \end{pmatrix}\]

Let $X_N := \frac{1}{\sqrt{N}} Y_N$ and define $H_N$ to be the Hermization of $X_N$. 
We will generally treat $H_N$ as an $N \times N$ matrix with entries that are $2\times 2$ blocks, but occasionally it will instead be useful to consider $H_N$ as a $2N \times 2N$ matrix.


Additionally we define the $2 \times 2$ matrix
\begin{align}
\label{qdef}
 q := \begin{pmatrix} \eta & z \\ \overline{z}& \eta   \end{pmatrix} 
\end{align}
with $ \eta=E+\sqrt{-1} t \in \C^{+} := \{w \in \mathbb{C} : \Im(w) > 0\}$ and $z\in \C$.  We define the Hermitized resolvent
\[ R_N(q) = R_N(\eta,z) := (H_N - I \otimes q)^{-1}. \]
Note that this is the usual resolvent of the Hermitization of $X_N - zI$, hence it inherits the usual properties of resolvents.  For example, its operator norm is bound from above by $t^{-1}$. We will use the Hermitized resolvent extensively in Section \ref{sec:least} to estimate the least singular value of $X_N -zI$ and in Section \ref{Off-diagonal entries} to estimate the expectation of bilinear forms involving $G_N(z)$.

\subsection{Truncation}

Let $\{Y_N\}_{N \geq 1}$ be a sequence of random matrices that satisfies condition {\bf C0} with atom variables $(\xi_1,\xi_2)$.  Instead of working with $Y_N$ directly, we will work with a truncated version of this matrix.  Specifically, we will work with a matrix $\hat{Y}_N$ where the entries are truncated versions of the original entries of $Y_N$.  

Recall that $Y_N = (y_{ij})_{i,j=1}^N$.  Let $L > 0$.  We define
$$ \tilde{\xi}_i := \xi_i \indicator{|\xi_i| \leq L} - \E \left[\xi_i \indicator{|\xi_i| \leq L} \right] $$
for $i \in \{1,2\}$, and 
$$ \tilde{\rho} := \E \left[\tilde{\xi}_1 \tilde{\xi}_2 \right]. $$
Here $\oindicator{E}$ denotes the indicator function of the event $E$.  We will also define the truncated entries
$$ \tilde{y}_{ij} := y_{ij} \indicator{|y_{ij}| \leq L} - \E \left[y_{ij} \indicator{|y_{ij}| \leq L} \right] $$
for $i \neq j$, and
$\tilde{y}_{ii} := 0$ for all $i \geq 1$.  We set $\tilde{Y}_N := (\tilde{y}_{ij})_{i,j=1}^N$.  

We also define
$$ \hat{\xi}_i := \frac{\tilde{\xi}_i}{\sqrt{\var (\tilde{\xi}_i) } } $$
for $i \in \{1,2\}$, and 
$$ \hat{\rho} := \E[\hat{\xi}_1 \hat{\xi}_2]. $$
We define the entries
$$ \hat{y}_{ij} := \frac{ \tilde{y}_{ij} } { \sqrt{ \var(\tilde{y}_{ij}) }} $$
for $i \neq j$, and $\hat{y}_{ii} := 0$ for all $i \geq 1$. We set $\hat{Y}_N := (\hat{y}_{ij})_{i,j=1}^N$.  We also introduce the notations
\begin{equation} \label{eq:def:hatG_N}
	\hat{G}_N = \hat{G}_N(z) := \left( \frac{1}{\sqrt{N}} \hat{Y}_N - z I \right)^{-1},
\end{equation}
\begin{equation} \label{eq:def:hatmN}
	\hat{m}_N(z) := \frac{1}{N} \tr \hat{G}_N(z). 
\end{equation}

We verify the following standard truncation lemma.  

\begin{lemma}[Truncation] \label{lemma:truncation}
Let $\{Y_N\}_{N \geq 1}$ be a sequence of random matrices that satisfies condition {\bf C0} with atom variables $(\xi_1,\xi_2)$.  Then there exists constants $C_0,L_0>0$ such that  the following holds for all $L > L_0$.  
\begin{itemize}
\item $\{\hat{Y}_N\}_{N \geq 1}$ is a sequence of random matrices that satisfies condition {\bf C0} with atom variables $(\hat{\xi}_1, \hat{\xi}_2)$.
\item a.s., one has the bounds
\begin{equation} \label{eq:aselbnd}
	\max_{1 \leq i,j } |\hat{y}_{ij}| \leq 4 L
\end{equation}
and
\begin{equation} \label{eq:rhohatbnd} 
	|\rho - \hat{\rho}| \leq \frac{C_0}{L}.
\end{equation}
\item a.s., one has
\begin{equation} \label{eq:ynhat}
	\limsup_{N \rightarrow \infty} \frac{1}{\sqrt{N}} \| Y_N - \hat{Y}_N \| \leq \frac{C_0}{L}
\end{equation}
and
\begin{equation} \label{eq:tilderesolventbnd}
	\limsup_{N \rightarrow \infty} \sup_{|z| \geq 5} \| G_N(z) - \hat{G}_N(z) \| \leq \frac{C_0}{L}. 
\end{equation}
\end{itemize}
\end{lemma}

The proof of Lemma \ref{lemma:truncation} follows somewhat standard arguments; we present the proof in Appendix \ref{app:truncation}.  

For the truncated matrices $\hat{Y}_N$, we have the following bound on the spectral norm.  

\begin{lemma}[Spectral norm bound for $\hat{Y}_N$] \label{lemma:truncatednorm}
Let $\{Y_N\}_{N \geq 1}$ be a sequence of random matrices that satisfies condition {\bf C0} with atom variables $(\xi_1, \xi_2)$.  Consider the truncated matrices $\{\hat{Y}_N\}_{N \geq 1}$ from Lemma \ref{lemma:truncation} for any fixed $L > 0$.  Let $\eps > 0$.  Then 
$$ \frac{1}{\sqrt{N}} \| \hat{Y}_N \| \leq 4 + \eps  $$
with overwhelming probability.  
\end{lemma}

The proof of Lemma \ref{lemma:truncatednorm} is almost identical to the proof of Lemma \ref{lemma:normbnd} except one applies \cite[Remark 5.7]{BSbook} instead of \cite[Theorem 5.2]{BSbook}.

%
%

\subsection{Martingale inequalities}

The following standard bounds were originally proven for real random variables; the extension to the complex case is straightforward.  

\begin{lemma}[Rosenthal's inequality, \cite{Bmart}] \label{lemma:rosenthal}
Let $\{x_k\}$ be a complex martingale difference sequence with respect to the filtration $\{ \mathcal{F}_k\}$.  Then, for $p \geq 2$, 
$$ \E\left| \sum x_k \right|^p \leq K_p \left( \E \left( \sum \E(|x_k|^2 | \mathcal{F}_{k-1}) \right)^{p/2} + \E \sum |x_k|^p \right). $$
\end{lemma}

\begin{lemma}[Burkholder's inequality, \cite{Bmart}] \label{lemma:burkholder} 
Let $\{x_k\}$ be a complex martingale difference sequence with respect to the filtration $\{ \mathcal{F}_k\}$.  Then, for $p \geq 1$, 
$$ \E \left|\sum x_k \right|^p \leq K_p \E \left(\sum \left|x_k \right|^2 \right)^{p/2} $$
\end{lemma}

\begin{lemma}[Dilworth, \cite{D}] \label{lemma:dilworth}
Let $\{\mathcal{F}_k\}$ be a filtration, $\{x_k\}$ a sequence of integrable random variables, and $1 \leq q \leq p < \infty$.  Then
$$ \E \left( \sum \left| \E(x_k | \mathcal{F}_k)\right|^q \right)^{p/q} \leq \left( \frac{p}{q} \right)^{p/q} \E \left( \sum |x_k|^q \right)^{p/q}. $$
\end{lemma}

\begin{lemma}[Lemma 6.11 of \cite{BSbook}]\label{Baimart}
 Let $\{F_n\}$ be an increasing sequence of $\sigma$-fields and $\{X_n\}$ a sequence of random variables. Write $\E_k = \E(\cdot | F_k)$, $\E_\infty = \E (\cdot | F_\infty)$, $F_\infty := \bigvee_j F_j$. If $X_n \to 0$ a.s. and $\sup_n |X_n|$ is integrable, then a.s.
\[ \lim_{n \to \infty} \max _{k\leq n}E_k [X_n] =0.\]
\end{lemma}
\subsection{Concentration of bilinear forms}

We establish the following large deviation estimate for bilinear forms, which is a consequence of Lemma \ref{lemma:lde} from Appendix \ref{section:lde}.  

\begin{lemma}[Concentration of bilinear forms] \label{lemma:quadp}
Let $(x,y)$ be a random vector in $\mathbb{C}^2$ where $x,y$ both have mean zero, unit variance, and satisfy 
\begin{itemize}
\item $\max\{|x|,|y|\} \leq L$ a.s.,
\item $\E[\bar{x} y] = \rho$.
\end{itemize}
Let $(x_1, y_1), (x_2, y_2), \ldots, (x_N, y_N)$ be iid copies of $(x,y)$, and set $X = (x_1, x_2, \ldots, x_N)^\mathrm{T}$ and $Y=(y_1, y_2, \ldots, y_N)^\mathrm{T}$.  Let $B$ be a $N \times N$ random matrix, independent of $X$ and $Y$.  Then for any integer $p \geq 2$, there exists a constant $K_p > 0$ such that, for any $t > 0$,
\begin{equation} \label{eq:quadp}
	\Prob \left( \left| \frac{1}{N} X^\ast B Y - \frac{\rho}{N} \tr B \right| \geq t \right) \leq K_p \frac{L^{2p} \E \left( \tr (B B^\ast) \right)^{p/2} }{N^{p} t^{p}}. 
\end{equation}
In particular, if $\|B\| \leq N^{1/4}$ a.s., then
\begin{equation} \label{eq:n18quadp}
	\Prob \left( \left| \frac{1}{N} X^\ast B Y - \frac{\rho}{N} \tr B \right| \geq N^{-1/8} \right) \leq K_p \frac{ L^{2p}} {N^{p/8}}
\end{equation}
for any integer $p \geq 2$.
\end{lemma}
\begin{proof}
We first note that \eqref{eq:n18quadp} follows from \eqref{eq:quadp} by taking $t = N^{-1/8}$ and applying the deterministic bound
$$ \left( \tr(B B^\ast) \right)^{p/2} \leq N^{p/2} \|B B^\ast\|^{p/2} \leq N^{p/2} \|B\|^{p}. $$

It remains to prove \eqref{eq:quadp}.  By Markov's inequality, it suffices to show 
\begin{equation} \label{eq:xqyp}
	\E \left| X^\ast B Y - \rho \tr B \right|^{p} \ll_p L^{2p} \E \left( \tr (BB^\ast) \right)^{p/2}
\end{equation}
for any integer $p \geq 2$.  We will use Lemma \ref{lemma:lde} from Appendix \ref{section:lde} to verify \eqref{eq:xqyp}.  

By conditioning on the matrix $B$ (which is independent of $X$ and $Y$), we apply Lemma \ref{lemma:lde} and obtain
\begin{align*}
	\E \left| X^\ast B Y - \rho \tr B \right|^p &\ll_p \E\left[ \left( L^4 \tr (B B^\ast) \right)^{p/2} + L^{2p} \tr (B B^\ast)^{p/2} \right] \\
		&\ll_p L^{2p} \left(  \E \left( \tr(B B^\ast) \right)^{p/2} + \E \tr(B B^\ast)^{p/2} \right) \\
		&\ll_p L^{2p} \E \left( \tr(B B^\ast) \right)^{p/2}
\end{align*}
since $\tr(B B^\ast)^{p/2} \leq \left( \tr(B B^\ast) \right)^{p/2}$. 
\end{proof}

\subsection{$\eps$-nets}

We introduce $\eps$-nets as a convenient way to discretize a compact set.  Let $\eps > 0$.  A set $X$ is an $\eps$-net of a set $Y$ if for any $y \in Y$, there exists $x \in X$ such that $\|x-y\| \leq \eps$.  We will need the following well-known estimate for the maximum size of an $\eps$-net.  

\begin{lemma} \label{lemma:epsnet}
Let $D$ be a compact subset of $\{z \in \mathbb{C} : |z| \leq M \}$.  Then $D$ admits an $\eps$-net of size at most
$$ \left( 1 + \frac{2M}{\eps} \right)^2. $$
\end{lemma}
\begin{proof}
Let $\mathcal{N}$ be maximal $\eps$ separated subset of $D$.  That is, $|z-w| \geq \eps$ for all distinct $z,w \in \mathcal{N}$ and no subset of $D$ containing $\mathcal{N}$ has this property.  Such a set can always be constructed by starting with an arbitrary point in $D$ and at each step selecting a point that is at least $\eps$ distance away from those already selected.  Since $D$ is compact, this procedure will terminate after a finite number of steps.  

We now claim that $\mathcal{N}$ is an $\eps$-net of $D$.  Suppose to the contrary.  Then there would exist $z \in D$ that is at least $\eps$ from all points in $\mathcal{N}$.  In other words, $\mathcal{N} \cup \{ z \}$ would still be an $\eps$-separated subset of $D$.  This contradicts the maximal assumption above.  

We now proceed by a volume argument.  At each point of $\mathcal{N}$ we place a ball of radius $\eps/2$.  By the triangle inequality, it is easy to verify that all such balls are disjoint and lie in the ball of radius $M + \eps/2$ centered at the origin.  Comparing the volumes give
$$ |\mathcal{N}| \leq \frac{ \left( M + \eps/2 \right)^2}{ (\eps/2)^2 } = \left( 1 + \frac{2M}{\eps} \right)^2. $$
\end{proof}

Similarly, if $I$ is an interval on the real line with length $|I|$, then $I$ admits an $\eps$-net of size at most $1 + |I|/\eps$.

\section{Stability of the fixed point equation} \label{sec:stability}

We will study the limit of the sequence of functions $\{m_N\}_{N \geq 1}$ (defined in \eqref{eq:def:mN}).  As is standard in random matrix theory, we will not compute the limit explicitly, but instead show that the limit satisfies a fixed point equation.  In particular, we will show that the limiting function satisfies 
\begin{equation} \label{eq:stab}
	\Delta(z) = -\frac{1}{z + \rho \Delta(z)}.
\end{equation}  
\begin{remark}
When $\rho > 0$, \eqref{eq:stab} also characterizes the Stieltjes transform of the semicircle distribution with variance $\rho$ (see for instance \cite[Chapter 2]{BSbook}).
\end{remark}

In this section, we study the stability of \eqref{eq:stab} for $-1 \leq \rho \leq 1$.  We begin with a few preliminary results.  

\begin{lemma} \label{lemma:ellipse}
For $-1 \leq \rho \leq 1$, $\pm 2 \sqrt{\rho} \in \mathcal{E}_\rho$.  
\end{lemma}
\begin{proof}
Let $z = \pm 2 \sqrt{\rho}$.  First consider the case when $0 \leq \rho \leq 1$.  Then $z^2 = \Re(z)^2 = 4 \rho$.   Since $0 \leq (1-\rho)^2 = (1+\rho)^2 - 4 \rho$, it follows that
$$ \frac{z^2}{(1+\rho)^2} = \frac{4 \rho}{(1+\rho)^2}  \leq 1, $$
and hence $z \in \mathcal{E}_\rho$.  A similar argument works for the case $-1 \leq \rho \leq 0$.  
\end{proof}

Since \eqref{eq:stab} can be written as a quadratic polynomial, the solution of \eqref{eq:stab} has two branches when $\rho \neq 0$.  We refer to the two branches as the solutions of \eqref{eq:stab}.  

\begin{lemma}[Solutions of \eqref{eq:stab}] \label{lemma:uniq}
Consider equation \eqref{eq:stab}.  Then one has the following.  
\begin{enumerate}[(i)]
\item If $\rho = 0$, there exists exactly one solution of \eqref{eq:stab}.    
\item If $-1 \leq \rho \leq 1$ and $\rho \neq 0$, there exists two solutions of \eqref{eq:stab}, which are distinct and analytic outside the ellipsoid $\mathcal{E}_\rho$.   \label{item:distinct}
\item For any $-1 \leq \rho \leq 1$, there exists a unique solution of \eqref{eq:stab}, which we denote by $m(z)$, which is analytic outside $\mathcal{E}_\rho$ and satisfies
\begin{equation} \label{eq:lim}
	\lim_{|z| \rightarrow \infty} m(z) = 0.
\end{equation}
Furthermore, 
\begin{equation} \label{eq:def:mz}
	m(z) := \left\{
     		\begin{array}{lr}
       		\frac{ -z + \sqrt{z^2 - 4 \rho}}{2 \rho} & \text{for} \quad \rho \neq 0\\
       		\frac{-1}{z} & \text{for}\quad \rho = 0
     		\end{array}
   	\right. ,
\end{equation}
where $\sqrt{z^2 - 4 \rho}$ is the branch of the square root with branch cut $[-2\sqrt{\rho}, 2\sqrt{\rho}]$ for $\rho > 0$ and $[-2\sqrt{|\rho|}, 2\sqrt{|\rho|}] \sqrt{-1}$ for $\rho < 0$, and which equals $z$ at infinity.  
\end{enumerate}
\end{lemma}
\begin{proof}
When $\rho = 0$, the results are trivial.  Assume $\rho \neq 0$.  By rewriting \eqref{eq:stab}, we find
$$ \rho \Delta(z)^2 + z \Delta(z) + 1 = 0. $$
Thus, by the quadratic equation, we have two solutions
\begin{equation} \label{eq:m1m2}
	m_1, m_2 = \frac{-z \pm \sqrt{z^2 - 4 \rho}}{2 \rho}, 
\end{equation}
where $\sqrt{z^2 - 4 \rho}$ is the branch of the square root with branch cut $[-2\sqrt{\rho}, 2\sqrt{\rho}]$ for $\rho > 0$ and $[-2\sqrt{|\rho|}, 2\sqrt{|\rho|}] \sqrt{-1}$ for $\rho < 0$, and which equals $z$ at infinity.    

Now suppose $m_1(z) = m_2(z)$ for some $z \in \C$.  Then we find
$$ z = \pm 2 \sqrt{\rho}. $$
Since $\pm 2 \sqrt{\rho} \in \mathcal{E}_\rho$ by Lemma \ref{lemma:ellipse}, the proof of \eqref{item:distinct} is complete.  

Finally, it is straightforward to check that 
$$ \frac{-z + \sqrt{z^2 - 4 \rho}}{2\rho} $$ 
is the only solution of \eqref{eq:stab} that satisfies \eqref{eq:lim}.  
\end{proof}

For the remainder of the paper, we let $m(z)$ be the unique solution of equation \eqref{eq:stab} given by \eqref{eq:def:mz}.  For $\rho \neq 0$, we let $m_2(z)$ denote the other solution of equation \eqref{eq:stab} described in Lemma \ref{lemma:uniq}.  Indeed, from the proof of Lemma \ref{lemma:uniq}, we have
\begin{equation} \label{eq:def:m2z}
	m_2(z) := \frac{-z - \sqrt{z^2 - 4 \rho}}{2\rho}. 
\end{equation}

\begin{lemma} \label{lemma:diffbound}
Let $-1 \leq \rho \leq 1$ with $\rho \neq 0$ and let $\delta > 0$.  Then
$$ |m(z) - m_2(z) | \geq \frac{ \delta }{|\rho|} $$
for all $z \in \C$ with $\dist(z, \mathcal{E}_\rho) \geq \delta$. 
\end{lemma}

\begin{proof}
From \eqref{eq:def:mz} and \eqref{eq:def:m2z}, we have
$$ |m(z) - m_2(z)|^2 = \frac{|z^2 - 4 \rho|}{|\rho|^2} = \frac{ |z - 2\sqrt{\rho}| |z+2\sqrt{\rho}|}{|\rho|^2}. $$
Since $\pm 2 \sqrt{\rho} \in \mathcal{E}_\rho$ by Lemma \ref{lemma:ellipse}, we conclude that 
$$ |m(z) - m_2(z)|^2 \geq \frac{\delta^2}{|\rho|^2} $$
for $\dist(z,\mathcal{E}_\rho) \geq \delta$.  
\end{proof}

\begin{lemma} 	\label{lemma:m'bound}
Let $D \subset \C$ such that $D \subset \{z \in \mathbb{C} : \dist(z,\mathcal{E}_\rho) \geq \delta, |z| \leq M\}$, for some $M,\delta >0$.  Then there exists $\eps, C,c > 0$ (depending only on $\delta, M, \rho$) such that the following holds.  Suppose $m'$ satisfies
\begin{equation} \label{eq:def:m'}
	m'(z) = \frac{-1}{z + \rho m'(z) + \eps_1(z)} + \eps_2(z), 
\end{equation}
for all $z \in D$.  If $|\eps_1(z)|, |\eps_2(z)| \leq \eps$ for all $z \in D$, then:
\begin{enumerate}
\item $|m'(z)| \leq C$ for all $z \in D$,
\item $|\rho m'(z) + z| \geq c$ for all $z \in D$.
\end{enumerate}
\end{lemma}
\begin{proof}
When $\rho = 0$, we note that
$$ |\rho m'(z) + z| = |z| \geq 1 + \delta $$
for all $z \in D$.  Moreover
$$ |z + \eps_1(z)| \geq |z| - \eps \geq 1 + \delta - \eps \geq \frac{1}{2} $$
for $\eps < 1/2$ and all $z \in D$.  Thus we obtain the bound $|m'(z)| \leq 5/2$.  

Assume $\rho \neq 0$.  Let $C$ be a large positive constant such that $C > 100 M$ and $C^2 > 2|\rho|$.  Assume $\eps > 0$ satisfies 
$$ \eps < \frac{49}{100} \sqrt{2 |\rho|}. $$
Then $\eps < \frac{49}{100} C$ by construction.  We will show that $|m'(z)| \leq C / |\rho|$ for all $z \in D$.  Suppose to the contrary that $|m'(z)| > C/|\rho|$ for some $z \in D$.  Then
\begin{align*}
	|z + \rho m'(z) + \eps(z)| \geq |\rho| |m'(z)| - |z| - \eps \geq C - \frac{C}{100} - \eps \geq \frac{C}{2}.
\end{align*}
Thus,
$$ \frac{C}{|\rho|} \leq |m'(z)| \leq \frac{2}{C}, $$
which contradicts the assumption that $C^2 > 2|\rho|$.  We conclude that $|m'(z)| \leq C/|\rho|$ for all $z \in D$.  

Using the bound above, we have
$$ \frac{|\rho|}{C} \leq |z + \rho m'(z) + \eps(z)| \leq |z + \rho m'(z)| + \eps $$
for all $z \in D$.  Thus, we have
$$ |z + \rho m'(z)| \geq \frac{|\rho|}{C} - \eps \geq \frac{|\rho|}{2C} $$
by taking $\eps$ sufficiently small.  
\end{proof}

\begin{lemma} \label{lemma:mbnd}
Let $\delta, M > 0$.  Then there exists $C, c>0$ (depending only on $\delta, M, \rho$) such that $c \leq |m(z)| \leq C$ for all $z \in \C$ satisfying $\dist(z,\mathcal{E}_\rho) \geq \delta$ and $|z| \leq M$.
\end{lemma}
\begin{proof}
Since $m(z)$ satisfies \eqref{eq:stab}, the claim follows from Lemma \ref{lemma:m'bound} by taking $\eps_1(z) = \eps_2(z) = 0$ (alternatively, one can derive the bounds directly from \eqref{eq:stab} and obtain an explicit expression for $C,c$ in terms of $\delta, \rho, M$).  
\end{proof}

\begin{lemma}[Stability] \label{lemma:dioch}
Let $D \subset \C$ be connected and satisfy $D \subset \{z \in \mathbb{C} : \dist(z,\mathcal{E}_\rho) \geq \delta, |z| \leq M\}$, for some $\delta, M > 0$.  Then there exists $\eps, C>0$ (depending only on $\delta, M, \rho$) such that the following holds.  Let $m'$ be a continuous function on $D$ that satisfies \eqref{eq:def:m'} for all $z \in D$.  If $|\eps_1(z)|, |\eps_2(z)| \leq \eps$ for all $z \in D$, then exactly one of the following holds:
\begin{enumerate}
	\item $|m'(z) - m(z)| \leq C( |\eps_1(z)| + |\eps_2(z)|)$ for all $z \in D$,
	\item $|m'(z) - m(z)| \geq \frac{\delta}{2|\rho|}$ for all $z \in D$.  
\end{enumerate}
\end{lemma}

\begin{proof}
First we consider the case $\rho = 0$.  For $\eps \leq 1/2$, we have that
$$ |z + \eps_1(z)| \geq |z| - \eps \geq 1/2 $$
for all $z \in D$.  Thus
$$ |m(z) - m'(z)| \leq \frac{|\eps_1(z)|}{|z||z + \eps_1(z)|} + |\eps_2(z)| \leq 2[ |\eps_1(z)| + |\eps_2(z)|]. $$

Assume $-1 \leq \rho \leq 1$ with $\rho \neq 0$.  By Lemma \ref{lemma:m'bound}, there exists $\eps, C'>0$ such that if $|\eps_1(z)|, |\eps_2(z)| \leq \eps/2$ for all $z \in D$, then $|m'(z)| \leq C'$ for all $z \in D$.  By rearranging \eqref{eq:def:m'}, we then obtain
$$ \left| m'(z)^2 + \frac{z}{\rho}m'(z) + \frac{1}{\rho} \right| \leq \frac{ C' }{|\rho|} |\eps_1(z)| + \frac{C' |\rho| + M + \eps }{|\rho|} |\eps_2(z)| \leq C (|\eps_1(z)| + |\eps_2(z)|), $$
where $C$ depends on $M, \rho, C'$.  Define $\tilde{\eps} := |\eps_1(z)| + |\eps_2(z)|$.  Factoring the left-hand side yields
\begin{equation} \label{eq:m'mm'm2}
	|m'(z) - m(z)| |m'(z) - m_2(z)| \leq C \tilde{\eps}
\end{equation}
for all $z \in D$.  From Lemma \ref{lemma:diffbound}, we obtain
\begin{equation} \label{eq:fdrmm2}
	\frac{\delta}{|\rho|} \leq |m(z) - m_2(z)| \leq |m(z) - m'(z)| + |m'(z) - m_2(z)| 
\end{equation}
for all $z \in D$.  Combining \eqref{eq:m'mm'm2} and \eqref{eq:fdrmm2} we obtain the quadratic inequality
$$ |m'(z) - m(z)|^2 - \frac{\delta}{|\rho|} |m'(z) - m(z)| + C \tilde{\eps} \geq 0. $$
For 
$$ \tilde{\eps} \leq \eps < \frac{\delta^2}{4C|\rho|^2}, $$
we obtain either 
$$ 2|m'(z) - m(z)| \leq \frac{\delta}{|\rho|} - \sqrt{ \frac{\delta^2}{|\rho|^2} - 4C \tilde{\eps}} \leq \frac{4 C |\rho|}{\delta} \tilde{\eps} $$
or 
$$ 2|m'(z) - m(z)| \geq \frac{\delta}{|\rho|} + \sqrt{ \frac{\delta^2}{|\rho|^2} - 4C \tilde{\eps}} \geq \frac{\delta}{|\rho|}. $$
For $\eps$ sufficiently small, the two possibilities above are distinct.  Because $m'-m$ is continuous and since $D$ is connected, a continuity argument implies that exactly one of the possibilities above holds for all $z \in D$.  
\end{proof}

We also verify that $m(z)$ is a continuous function of $\rho$.  

\begin{lemma} \label{lemma:rhostab}
Fix $z \in \mathbb{C}$ with $|z| > 2$.  Then $m(z)$ is a continuous function of $\rho \in [-1,1]$.  
\end{lemma}
\begin{proof}
In order to denote the dependence on $\rho$, we let $m_\rho(z)$ be the function defined by \eqref{eq:def:mz} for any $-1 \leq \rho \leq 1$.  Fix $z \in \C$ with $|z| > 2$.  Then $z \notin \cup_{-1 \leq \rho \leq 1} \mathcal{E}_\rho$.  By definition, 
\begin{equation} \label{eq:quad}
	\rho m^2_\rho(z) + z m_\rho(z) + 1 = 0
\end{equation}
for $-1 \leq \rho \leq 1$.  Since the roots of a (monic) polynomial are continuous functions of the coefficients (see \cite{CC,T}), we conclude that $m_\rho(z)$ is a continuous function of $\rho \in [-1,1]\setminus\{0\}$.  It remains to show $m_\rho(z)$ is continuous at $\rho =0$.  

Multiplying \eqref{eq:quad} by $\rho$, we see that $\rho m_{\rho}(z)$ is a continuous function of $\rho \in [-1,1]$.  Thus, we have
$$ \lim_{\rho\to 0} \rho m_\rho(z) = 0, $$
and hence there exists $\eps, c > 0$ such that
$$ \left| \rho m_{\rho}(z) + z \right| \geq c $$
for all $|\rho| \leq \eps$.  By \eqref{eq:stab}, it follows that
$$ |m_\rho(z)| \leq \frac{1}{c} $$
for all $|\rho| \leq \eps$.  Let $m_0(z) = -1/z$ (i.e. $m_0(z)$ is given by \eqref{eq:def:mz} when $\rho=0$).  Then
$$ zm_0(z) + 1 = 0. $$
Subtracting \eqref{eq:quad} from the equation above yields
$$ |z| |m_0(z) - m_\rho(z)| = |\rho| |m_\rho(z) |^2 \leq \frac{|\rho|}{c^2} $$
for $|\rho| \leq \eps$.  Since $|z| > 2$, we conclude that $m_\rho(z)$ is continuous at $\rho=0$.  
\end{proof}

\section{Truncation arguments and the isotropic limit law}  \label{sec:main}

In this section, we begin the proof of Theorem \ref{thm:main} and Theorem \ref{thm:least} by reducing to the case where we only need to consider the truncated matrices $\{\hat{Y}_N\}_{N \geq 1}$.  

\subsection{Isotropic limit law}
This subsection is devoted to Theorem \ref{thm:main}.  We will prove Theorem \ref{thm:main} using the following isotropic limit law, which is inspired by the isotropic semicircle law developed by Knowles and Yin \cite{KY,KY2}.  

\begin{theorem}[Isotropic limit law] \label{thm:mainuv}
Let $\{Y_N\}_{N \geq 1}$ be a sequence of random matrices that satisfies condition {\bf C0} with atom variables $(\xi_1, \xi_2)$, where $\rho = \E[\xi_1 \xi_2]$.  Let $\delta > 0$.  For each $N \geq 1$, let $u_N$ and $v_N$ be unit vectors in $\C^N$.  Then a.s. 
$$ \sup_{\dist(z, \mathcal{E}_{\rho}) \geq \delta} \left| u_N^\ast \left( \frac{1}{\sqrt{N}} Y_N - z I \right)^{-1} v_N - m(z) u_N^\ast v_N \right| \longrightarrow 0 $$
as $N \rightarrow \infty$.  
\end{theorem}

Assuming Theorem \ref{thm:least} and Theorem \ref{thm:mainuv}, we complete the proof of Theorem \ref{thm:main}.  By the singular value decomposition, we write $C_N = A_N B_N$, where $A_N$ is a $N \times k$ matrix and $B_N$ is a $k \times N$ matrix.  By assumption, both $A_N$ and $B_N$ have operator norm $O(1)$.  Based on \cite[Lemma 2.1]{Tout}, we have the following lemma.  

\begin{lemma}[Eigenvalue criterion] \label{lemma:eigenvalue}
Let $z$ be a complex number that is not an eigenvalue of $\frac{1}{\sqrt{N}} Y_N$.  Then $z$ is an eigenvalue of $\frac{1}{\sqrt{N}} Y_N + C_N$ if and only if 
$$ \det \left( I + B_N G_N(z) A_N \right) = 0. $$
\end{lemma}
\begin{proof}
Clearly $z$ is an eigenvalue of $\frac{1}{\sqrt{N}} Y_N + C_N$ if and only if 
$$ \det \left( \frac{1}{\sqrt{N}} Y_N + C_N - z I \right) = 0. $$
Since $\frac{1}{\sqrt{N}} Y_N - zI$ is invertible by assumption, we rewrite the above equation as
$$ \det \left( I + \left( \frac{1}{\sqrt{N}} Y_N - z I \right)^{-1} A_N B_N \right) = 0. $$
The claim now follows from \eqref{eq:sylvdet} and \eqref{eq:def:G_N}.  
\end{proof}

\begin{remark} \label{rem:arg}
The proof of Lemma \ref{lemma:eigenvalue} actually reveals that
\begin{equation} \label{eq:iddet}
	\det \left( I + B_N G_N(z) A_N \right) = \frac{\det \left( \frac{1}{\sqrt{N}} Y_N + C_N - zI \right) }{\det \left( \frac{1}{\sqrt{N}} Y_N - z I \right) } 
\end{equation}
provided the denominator does not vanish.  Versions of this identity have appeared in previous publications including \cite{AGG,BR,BGM,BGM2,CDF}.  
\end{remark}

Following Tao in \cite{Tout}, we define the functions
$$ f(z) := \det \left( I + B_N G_N(z) A_N \right) $$
and
$$ g(z) := \det \left( I + m(z) B_N A_N \right), $$
where $m(z)$ is defined in \eqref{eq:def:mz}.  Both $f$ and $g$ are meromorphic functions outside $\mathcal{E}_\rho$ that are asymptotically equal to $1$ at infinity.  By Lemma \ref{lemma:eigenvalue}, the zeroes of $f$ coincide with the eigenvalues of $\frac{1}{\sqrt{N}} Y_N + C_N$ outside the spectrum of $\frac{1}{\sqrt{N}} Y_N$.  Moreover, from \eqref{eq:iddet} we see that the multiplicity of any such eigenvalue is equal to the degree of the corresponding zero of $f$.  It follows from \eqref{eq:sylvdet} that
$$ g(z) = \prod_{i=1}^k \left( 1 + m(z) \lambda_i(C_N) \right), $$
where $\lambda_1(C_N), \ldots, \lambda_k(C_N)$ are the non-trivial eigenvalues of $C_N$ (some of which may be zero). 

In order to study the zeroes of $g$, we consider the values of $z\notin \mathcal{E}_\rho$ for which 
\begin{equation} \label{eq:mz1lam}
	m(z) = - \frac{1}{\lambda}.
\end{equation}
Indeed, for $0 < |\lambda| \leq 1$, there does not exist $z \notin \mathcal{E}_\rho$ which solves \eqref{eq:mz1lam}; for $|\lambda| > 1$, \eqref{eq:mz1lam} holds if and only if 
\begin{equation} \label{eq:inversem}
	z = \lambda + \frac{\rho}{\lambda}. 
\end{equation}
This follows from \eqref{eq:stab} and an analytic continuation argument\footnote{One technical issue that arises when $\rho \neq 0$ is that the solution of \eqref{eq:stab} has two distinct analytic branches $m,m_2$.  In order to overcome this obstacle, we make the following observations.
\begin{enumerate}[(i)]
\item From \eqref{eq:stab}, we see that if \eqref{eq:mz1lam} holds, it must be the case that $z = \lambda + \frac{\rho}{\lambda}$.  
\item The function $\lambda \mapsto \lambda + \frac{\rho}{\lambda}$ maps circles to ellipses and is one-to-one when restricted to the domain $\{\lambda \in \mathbb{C} : 0 < |\lambda| < |\rho|\}$ or  $\{\lambda \in \mathbb{C} : |\lambda| > 1\}$.  Moreover $\lambda + \frac{\rho}{\lambda} \in \mathcal{E}_\rho$ if and only if $|\rho| \leq |\lambda| \leq 1$.  
\end{enumerate}
It follows that \eqref{eq:mz1lam} has no solution outside the ellipse when $|\rho| \leq |\lambda| \leq 1$.  Furthermore, for $|\lambda|$ sufficiently large, one can deduce the solution \eqref{eq:inversem} for the branch $m$ and then extend to the region $|\lambda| > 1$ by analytic continuation.  Similarly, one can show that \eqref{eq:mz1lam} has no solution outside the ellipse when $|\lambda| < |\rho|$; in fact, for $|\lambda| < |\rho|$, \eqref{eq:inversem} is a solution of $m_2(z) = - \frac{1}{\lambda}$.}.

By Theorem \ref{thm:nooutliers} (which was proved in Section \ref{sec:mainresults} assuming Theorem \ref{thm:least} holds), it follows that a.s., for $N$ sufficiently large, all the eigenvalues of $\frac{1}{\sqrt{N}} Y_N$ are contained in $\mathcal{E}_{\rho, \delta}$.  By Rouch\'{e}'s theorem, in order to prove Theorem \ref{thm:main}, it suffices to show that a.s.
$$ \sup_{\dist(z, \mathcal{E}_{\rho}) \geq 2\delta} \left| f(z) - g(z) \right| \longrightarrow 0 $$
as $N \rightarrow \infty$.  Since $A_N, B_N, G_N$ a.s. have operator norm $O(1)$ (by Theorem \ref{thm:least}) and $k$ is fixed, independent of $N$, it suffices to show that a.s.
$$ \sup_{\dist(z, \mathcal{E}_{\rho}) \geq 2 \delta} \left \| B_N \left( G_N - m(z) I \right) A_N \right \| \longrightarrow 0 $$
as $N \rightarrow \infty$.  Since $B_N \left( G_N - m(z) I \right) A_N$ is a $k \times k$ matrix, the claim now follows from Theorem \ref{thm:mainuv}.  

\begin{remark} \label{rem:subord}
Equation \eqref{eq:inversem} is similar to the formulas obtained in \cite{CDFF} for the location of the outlier eigenvalues.  Indeed, \eqref{eq:inversem} can be obtained using techniques from free probability.  Let $\mu_{\mathrm{sc},\rho}$ be the semicircle distribution with variance $\rho$ and let $\mu_{\mathrm{circ},1-\rho}$ be the uniform distribution on the disk centered at the origin in the complex plane with radius $(1-\rho)^{1/2}$. Let $S_\rho$ have distribution $\mu_{\mathrm{sc},\rho}$, $C_{1-\rho}$ have distribution $\mu_{\mathrm{circ},1-\rho}$, and $E_\rho$ have elliptic distribution $\mu_\rho$. Then 
\[ E_\rho = S_\rho + C_{1-\rho} \]
with $S_\rho$ and $C_{1-\rho}$ free random variables.  Outside of the ellipse, the Stieltjes transform of $\mu_\rho$ can be expressed as the Stieltjes transform of the circular law evaluated at the subordination function $F(z) = z +\rho m(z)$.  This can be seen by adding the $R$-transforms together and inverting to obtain the Stieltjes transform.  The inverse function of $F$ is $H(z) = z + \rho/z$, which is precisely the function appearing in \eqref{eq:inversem}. The function $H$ plays the same role here as in \cite{CDFF}. Since we are only interested in solutions outside the ellipsoid, the domain of $H$ is restricted to $|z| > 1$.  
\end{remark}

We now reduce the proof of Theorem \ref{thm:mainuv} to the case where we only need to consider the truncated matrices $\{\hat{Y}_N\}_{N \geq 1}$.  We let $\hat{m}(z)$ be the function given by \eqref{eq:def:mz} with $\rho$ replaced by $\hat{\rho}$.

\begin{theorem}[Isotropic limit law for $\hat{Y}_N$] \label{thm:mainuv:trunc}
Let $\{Y_N\}_{N \geq 1}$ be a sequence of random matrices that satisfies condition {\bf C0} with atom variables $(\xi_1, \xi_2)$, where $\rho = \E[\xi_1 \xi_2]$.   Let $\eps > 0$.  Let $L >0$, and consider the truncated random matrices $\{\hat{Y}_N\}_{N \geq 1}$ from Lemma \ref{lemma:truncation}.  For each $N \geq 1$, let $u_N$ and $v_N$ be unit vectors in $\C^N$.  Fix $z \in \mathbb{C}$ with $5 \leq |z| \leq 6$.  Then a.s., for $N$ sufficiently large, 
$$ \left| u_N^\ast \left( \frac{1}{\sqrt{N}} \hat{Y}_N - z I \right)^{-1} v_N - \hat{m}(z) u_N^\ast v_N \right| \leq \eps. $$
\end{theorem}

We now prove Theorem \ref{thm:mainuv} assuming Theorem \ref{thm:least} and Theorem \ref{thm:mainuv:trunc}.

\begin{proof}[Proof of Theorem \ref{thm:mainuv}]
Let $\eps, \delta > 0$.  It suffices to show that a.s., for $N$ sufficiently large, 
\begin{equation} \label{eq:mainuvsuff}
	\sup_{\dist(z,\mathcal{E}_\rho) \geq \delta} \left| u_N^\ast G_N(z) v_N - m(z) u_N^\ast v_N \right| \leq \eps.
\end{equation}

Consider the compact set $D := \{ z \in \mathbb{C} : 5 \leq |z| \leq 6 \}$.  Since $m$ is analytic on $D$, it follows from the Heine--Cantor theorem (see for instance \cite[Theorem 4.19]{R}) that $m$ is uniformly continuous on $D$.  Thus, there exists $\eps' > 0$ such that if $z,w \in D$, then $|z-w| \leq \eps'$ implies $|m(z) - m(w)| \leq \eps/100$.  

Set $\eps'' := \min\{\eps/100, \eps'\}$, and let $\mathcal{N}$ be a $\eps''$-net of $D$.  By Lemma \ref{lemma:epsnet}, $|\mathcal{N}| = O(1)$.  By Theorem \ref{thm:mainuv:trunc}, we have a.s., for $N$ sufficiently large, 
$$ \sup_{z \in \mathcal{N}}  \left| u_N^\ast \hat{G}_N(z) v_N - \hat{m}(z) u_N^\ast v_N \right|  \leq \frac{\eps}{200}. $$
Furthermore, by Lemma \ref{lemma:truncation} and Lemma \ref{lemma:rhostab} (taking $L$ sufficiently large), we have a.s., for $N$ sufficiently large,
\begin{equation} \label{eq:supungnvn}
	\sup_{z \in \mathcal{N}}  \left| u_N^\ast \hat{G}_N(z) v_N - m(z) u_N^\ast v_N \right|  \leq \frac{\eps}{100}. 
\end{equation}
We now extend this bound to all $z \in D$.  By Lemma \ref{lemma:detresolventbnd}, Lemma \ref{lemma:normbnd}, and \eqref{eq:generalresolventid}, we have a.s., for $N$ sufficiently large, 
\begin{equation} \label{eq:gnzgnw}
	\|\hat{G}_N(z) - \hat{G}_N(w) \| \leq 4 |z - w| 
\end{equation}
for all $z, w \in D$.  Fix a realization in which \eqref{eq:supungnvn} and \eqref{eq:gnzgnw} hold.  Choose $w \in D$.  Then there exists $z \in \mathcal{N}$ with $|z-w| \leq \eps''$.  Thus, from \eqref{eq:gnzgnw}, we have 
$$ \left| u_N^\ast \hat{G}_N(z) v_N  - u_N^\ast \hat{G}_N(w) v_N \right| \leq \|\hat{G}_N(z) - \hat{G}_N(w) \| \leq \frac{\eps}{10}. $$
On the other hand, from the uniform continuity of $m$, we have
$$ |m(z)u_N^\ast v_N - m(w) u_N^\ast v_N | \leq |m(z) - m(w)| \leq \frac{\eps}{100}. $$
Combining the bounds above with \eqref{eq:supungnvn}, we conclude that, for $N$ sufficiently large,
\begin{equation} \label{eq:supzeps2}
	\sup_{z \in D}  \left| u_N^\ast \hat{G}_N(z) v_N - m(z) u_N^\ast v_N \right|  \leq \frac{\eps}{2}
\end{equation}
for any fixed realization in which \eqref{eq:supungnvn} and \eqref{eq:gnzgnw} hold.  In other words, we have a.s., for $N$ sufficiently large, \eqref{eq:supzeps2} holds.  

By Lemma \ref{lemma:detresolventbnd}, Lemma \ref{lemma:normbnd}, and the resolvent identity \eqref{eq:generalresolventid}, we have a.s., for $N$ sufficiently large, 
$$ \sup_{z \in D} \left| u_N^\ast \hat{G}_N(z) v_N - u_N G_N(z) v_N \right| \leq \frac{4}{\sqrt{N}} \| Y_N - \hat{Y}_N \|. $$
Thus, by Lemma \ref{lemma:truncation}, we obtain a.s., for $N$ sufficiently large, 
$$ \sup_{z \in D} \left| u_N^\ast \hat{G}_N(z) v_N - u_N G_N(z) v_N \right| \leq \frac{C}{L} \leq \frac{\eps}{2} $$
by taking $L$ sufficiently large.  Combining the bound above with \eqref{eq:supzeps2}, we conclude that a.s., for $N$ sufficiently large, 
$$ \sup_{z \in D} \left| u_N^\ast G_N(z) v_N - m(z) u_N^\ast v_N \right| \leq \eps. $$
Since $\eps$ is arbitrary, we in fact obtain that a.s.
$$ \sup_{z \in D} \left| u_N^\ast G_N(z) v_N - m(z) u_N^\ast v_N \right| \longrightarrow 0 $$
as $N \rightarrow \infty$.  

By definition \eqref{eq:def:mz}, there exists $M > 0$ such that $|m(z)| \leq \eps/100$ for all $|z| \geq M$.  Moreover, from Lemma \ref{lemma:detresolventbnd} and Lemma \ref{lemma:normbnd} (by increasing $M$ if necessary), we have a.s., for $N$ sufficiently large,
\begin{equation} \label{eq:zlargeM}
	\sup_{|z| \geq M} \| G_N(z) \| \leq \frac{\eps}{100}. 
\end{equation}
Consider the compact set $D' := \{ z \in \mathbb{C} : \dist(z,\mathcal{E}_\rho) \geq \delta, |z| \leq M \}$.  By Theorem \ref{thm:least} and Lemma \ref{lemma:mbnd}, we have a.s., for $N$ sufficiently large, $u_N^\ast (G_N(z) - m(z) I) v_N$ is analytic and uniformly bounded on $D'$.  We apply Vitali's convergence theorem to obtain a.s.
$$  \sup_{z \in D'} \left| u_N^\ast G_N(z) v_N - m(z) u_N^\ast v_N \right| \longrightarrow 0 $$
as $N \rightarrow \infty$.  This implies that a.s., for $N$ sufficiently large, 
\begin{equation} \label{eq:supcom1}
	\sup_{z \in D'} \left| u_N^\ast G_N(z) v_N - m(z) u_N^\ast v_N \right| \leq \eps/2. 
\end{equation}
On the other hand, by \eqref{eq:zlargeM}, we have a.s., for $N$ sufficiently large,
\begin{equation} \label{eq:supcom2}
	\sup_{|z| \geq M} \left| u_N^\ast G_N(z) v_N - m(z) u_N^\ast v_N \right| \leq \sup_{|z| \geq M} \left( \|G_N(z) \| + |m(z)| \right) \leq \eps/2. 
\end{equation}
Combining \eqref{eq:supcom1} and \eqref{eq:supcom2}, we obtain \eqref{eq:mainuvsuff}, and the proof is complete.
\end{proof}

We will prove Theorem \ref{thm:mainuv:trunc} in Section \ref{sec:mainuv}.

\subsection{Elliptic random matrices with nonzero mean}

In this subsection, we use Theorem \ref{thm:mainuv} to prove Theorem \ref{thm:mean}.  The proof is based on the arguments in \cite[Section 3]{Tout}.  

Let $\{Y_N\}_{N \geq 1}$, $\mu$, and $\varphi_N$ be as in Theorem \ref{thm:mean}; let $\delta > 0$.  It suffices to show that a.s. there exists one eigenvalue of $X_N + \mu \sqrt{N} \varphi_N \varphi_N^\ast$ outside $\mathcal{E}_{\rho,\delta}$, with this eigenvalue occurring within $o(1)$ of $\mu \sqrt{N}$.  We begin with the following lemma.

\begin{lemma} \label{lemma:firstmoment}
Let $\{Y_N\}_{N \geq 1}$ and $\varphi_N$ be as in Theorem \ref{thm:mean}.  Then a.s.
$$ \varphi_N^\ast X_N \varphi_N \longrightarrow 0 $$
as $N \rightarrow \infty$.
\end{lemma}
\begin{proof}
By Markov's inequality and the Borel--Cantelli lemma, it suffices to show that
\begin{equation} \label{eq:momentbnd}
	\E \left| \varphi_N^\ast X_N \varphi_N \right|^4 = O_{M_4} \left( \frac{1}{N^{2}} \right). 
\end{equation}
We write
$$ \E \left| \varphi_N^\ast X_N \varphi_N \right|^4 = \frac{1}{N^{6}} \sum_{i_1,i_2,i_3 i_4, j_1, j_2,j_3, j_4=1}^N \E[ y_{i_1 j_1} y_{i_2 j_2} y_{i_3 j_3} y_{i_4 j_4}]. $$
We now consider the pairs
$$ (i_1,j_1), (i_2,j_2), (i_3,j_3), (i_4,j_4) $$
for which $\E[ y_{i_1 j_1} y_{i_2 j_2} y_{i_3 j_3} y_{i_4 j_4}]$ is nonzero.  From Definition \ref{def:C1}, we see that each pair $(i_s,j_s)$ must correspond to some $(i_r,j_r)$ or ($j_r, i_r)$ for $s \neq r$.  Counting all such pairs yields \eqref{eq:momentbnd}.
\end{proof}

Define the functions
\begin{align*}
	f(z) &:= 1 + \mu \sqrt{N} \varphi_N^\ast G_N(z) \varphi_N, \\
	g(z) &:= 1 + \mu \sqrt{N} m(z), \\
	h(z) &:= 1 - \frac{\mu \sqrt{N}}{z}.
\end{align*}
By \eqref{eq:inversem}, it follows that $g$ has precisely one zero outside $\mathcal{E}_{\rho}$ located at $\mu \sqrt{N} + \frac{\rho}{\mu \sqrt{N}}$.  By Lemma \ref{lemma:eigenvalue} and Theorem \ref{thm:nooutliers}, a.s. the eigenvalues of $X_N + \mu \sqrt{N} \varphi_N \varphi_N^\ast$ outside $\mathcal{E}_{\rho,\delta}$ correspond to the zeroes of $f$.  From Theorem \ref{thm:mainuv}, we see that a.s.
$$ f(z) = g(z) + o(\sqrt{N}) $$
uniformly for $z \notin \mathcal{E}_{\rho,\delta}$.  We conclude that if $f$ has a zero outside $\mathcal{E}_{\rho,\delta}$, it must tend to infinity with $N$.  Thus, for the remainder of the proof, we restrict our attention to the region $|z| \geq 5$.  It remains to show that $f$ a.s. has exactly one zero outside $|z| \geq 5$ taking the value $z = \mu \sqrt{N} + o(1)$.  

By writing out the Neumann series and applying Lemma \ref{lemma:normbnd}, we obtain a.s.
\begin{align*}
	f(z) = h(z) - \frac{\mu \sqrt{N}}{z^2}\varphi_N^\ast X_N \varphi_N + O \left( \frac{\sqrt{N}}{|z|^3} \right)
\end{align*}
uniformly for $|z| \geq 5$.  Thus, by Lemma \ref{lemma:firstmoment}, we conclude that a.s.,
\begin{equation} \label{eq:fhbnd}
	f(z) = h(z) + o \left( \frac{\sqrt{N}}{|z|^2} \right) + O \left( \frac{\sqrt{N}}{|z|^3} \right)
\end{equation}
uniformly for $|z| \geq 5$.  Let $\eps > 0$; from Rouch\'{e}'s theorem, we conclude that a.s., for $N$ sufficiently large, $f$ has exactly one zero in the disk of radius $\eps$ centered at $\mu \sqrt{N}$.  

Let $z$ be any zero of $f$ outside $\mathcal{E}_{\rho,\delta}$.  Since $z$ tends to infinity with $N$, we apply \eqref{eq:fhbnd} and obtain a.s.
$$ h(z) = o \left( \frac{\sqrt{N}}{|z|^2} \right). $$
Thus, $|z| = \Omega(\sqrt{N})$, and hence a.s. 
$$ z = \mu \sqrt{N} + o\left( \frac{\sqrt{N}}{|z|} \right) = \mu \sqrt{N} + o(1). $$
Therefore, we conclude that a.s., for $N$ sufficiently large, $f$ has precisely one zero outside $\mathcal{E}_{\rho,\delta}$ taking the value $z = \mu \sqrt{N} + o(1)$, and the proof is complete.  

\subsection{Least singular value bound}

We now turn our attention to Theorem \ref{thm:least}.  Again, we will reduce to the case where we only need to consider the truncated matrices $\{ \hat{Y}_N \}_{N \geq 1}$. 

\begin{theorem} \label{thm:least:trunc}
Let $\{Y_N\}_{N \geq 1}$ be a sequence of random matrices that satisfies condition {\bf C0} with atom variables $(\xi_1, \xi_2)$, where $\rho = \E[\xi_1 \xi_2]$.  Let $\delta > 0$.  Then there exists $c>0$ such that the following holds.  Let $L > 0$ and consider the truncated random matrices $\{\hat{Y}_N\}_{N \geq 1}$ from Lemma \ref{lemma:truncation}.  Then a.s., for $N$ sufficiently large, 
$$ \inf_{\dist(z, \mathcal{E}_\rho) \geq \delta} \sigma_N \left( \frac{1}{\sqrt{N}}  \hat{Y}_N - z I \right) \geq c. $$
\end{theorem}

Assuming Theorem \ref{thm:least:trunc}, we now prove Theorem \ref{thm:least}.  

\begin{proof}[Proof of Theorem \ref{thm:least}]
In order to prove Theorem \ref{thm:least}, it suffices to show that a.s., for $N$ sufficiently large, 
$$ \sup_{\dist(z, \mathcal{E}_{\rho}) \geq \delta} \| G_N(z)\| \leq C' $$
for some constant $C' > 0$.  From \eqref{eq:generalresolventid}, we obtain 
$$ G_N(z) = \hat{G}_N(z) [ I - (\hat{X}_N - X_N) \hat{G}_N(z)]^{-1} $$
provided all the relevant matrices on the right-hand side are invertible.  From Theorem \ref{thm:least:trunc}, we have a.s., for $N$ sufficiently large, 
$$ \sup_{\dist(z, \mathcal{E}_{\rho}) \geq \delta} \| \hat{G}_N(z) \| \leq \frac{1}{c}. $$
It thus suffices to show that a.s., for $N$ sufficiently large, 
\begin{equation} \label{eq:leastsuff}
	\sup_{\dist(z, \mathcal{E}_{\rho}) \geq \delta } \|  [ I - (\hat{X}_N - X_N) \hat{G}_N(z)]^{-1} \| \leq C'' 
\end{equation}
for some constant $C'' > 0$.  

From Lemma \ref{lemma:truncation}, it follows that a.s., for $N$ sufficiently large, 
$$ \| \hat{X}_N - X_N \| \leq \frac{C}{L} $$
for some constant $C > 0$.  Thus, by taking $L$ sufficiently large, we conclude that 
$$ \sup_{\dist(z, \mathcal{E}_\rho) \geq \delta} \| (\hat{X}_N - X_N) \hat{G}_N(z)\| \leq \frac{C}{Lc} < 1. $$
Thus, by the Neumann series, we obtain a.s., for $N$ sufficiently large, 
\begin{align*}
	\sup_{\dist(z, \mathcal{E}_\rho) \geq \delta} \|  [ I - (\hat{X}_N - X_N) \hat{G}_N(z)]^{-1} \| &\leq \sup_{\dist(z, \mathcal{E}_\rho) \geq \delta}  \sum_{k=0}^\infty \| (\hat{X}_N - X_N) \hat{G}_N(z) \|^k \\
	&\leq \frac{1}{1 - \frac{C}{Lc}}, 
\end{align*}
and the proof is complete.  
\end{proof}

We now reduce to the case where $z$ is fixed (as opposed to taking the infimum over an uncountable number of complex numbers).  We proceed using an $\eps$-net argument and the following theorem.  

\begin{theorem} \label{thm:least:truncfix}
Let $\{Y_N\}_{N \geq 1}$ be a sequence of random matrices that satisfies condition {\bf C0} with atom variables $(\xi_1, \xi_2)$, where $\rho = \E[\xi_1 \xi_2]$.  Let $\delta > 0$.  Then there exists a constant $c> 0$ such that the following holds.  Let $L > 0$, and consider the truncated random matrices $\{\hat{Y}_N\}_{N \geq 1}$ from Lemma \ref{lemma:truncation}.  Then for any $z$ with $\dist(z, \mathcal{E}_\rho) \geq \delta$ and $|z| \leq 6$, a.s., for $N$ sufficiently large, 
$$ \sigma_N \left( \frac{1}{\sqrt{N}}  \hat{Y}_N - z I \right) \geq c. $$
\end{theorem}

We now verify Theorem \ref{thm:least:trunc} assuming Theorem \ref{thm:least:truncfix}.

\begin{proof}[Proof of Theorem \ref{thm:least:trunc}]
Let $\delta > 0$, and let $c > 0$ be the constant from Theorem \ref{thm:least:truncfix}.  We first note that a.s., for $N$ sufficiently large, 
$$ \sup_{|z| \geq 6} \| \hat{G}_N(z) \| \leq 1 $$
by Lemma \ref{lemma:detresolventbnd} and Lemma \ref{lemma:normbnd}.  Thus, it suffices to show that a.s., for $N$ sufficiently large, 
$$ \inf_{z \in D} \sigma_N \left( \frac{1}{\sqrt{N}} \hat{Y}_N - z I \right) \geq c' $$
for some constant $c' > 0$, where $D := \{z \in \mathbb{C} : \dist(z, \mathcal{E}_\rho) \geq \delta, |z| \leq 6\}$.  

Let $\mathcal{N}$ be a $c/10$-net of the compact region $D$.  By Lemma \ref{lemma:epsnet}, $|\mathcal{N}| = O(1)$.  Thus, by applying Theorem \ref{thm:least:truncfix} to each $z \in \mathcal{N}$, we obtain a.s., for $N$ sufficiently large, 
\begin{equation} \label{eq:zinmcN}
	\inf_{z \in \mathcal{N}} \sigma_N \left( \frac{1}{\sqrt{N}} \hat{Y}_N - z I \right) \geq c. 
\end{equation}
We now extend this bound to all $z \in D$.  Fix a realization in which \eqref{eq:zinmcN} holds.  Choose $z \in D$.  Then there exists $z' \in \mathcal{N}$ with $|z - z'| \leq c/10$.  By Weyl's perturbation theorem (see for instance \cite{B}), 
$$ \left| \sigma_N \left( \frac{1}{\sqrt{N}} \hat{Y}_N - z I \right) - \sigma_N \left( \frac{1}{\sqrt{N}} \hat{Y}_N - z' I \right) \right| \leq |z - z'| \leq \frac{c}{10}. $$
Thus, we conclude that 
$$ \inf_{z \in D} \sigma_N \left( \frac{1}{\sqrt{N}} \hat{Y}_N - z I \right) \geq \frac{c}{2} $$
for any realization in which \eqref{eq:zinmcN} holds.  The proof of the theorem is complete.  
\end{proof}

We will prove Theorem \ref{thm:least:truncfix} in Section \ref{sec:least}.

\subsection{Notation}

It remains to prove Theorem \ref{thm:mainuv:trunc} and Theorem \ref{thm:least:truncfix}.  As such, for the remainder of the paper we only consider the truncated matrices $\{\hat{Y}_N\}_{N \geq 1}$ from Lemma \ref{lemma:truncation} for some arbitrarily large fixed constant $L > 0$.  Thus, we drop the decorations from our notation and simply write $Y_N, X_N, G_N$ for the matrices $\hat{Y}_N, \hat{X}_N, \hat{G}_N$.  Similarly, we write $m_N(z)$ for the function $\hat{m}_N(z)$; we also write $m(z)$ for the function $\hat{m}(z)$.

\section{Least singular value bound} \label{sec:least}

This section is devoted to Theorem \ref{thm:least:truncfix}. For this entire section we  work with fixed $z$ satisfying the hypothesis of Theorem \ref{thm:least:truncfix}.

\subsection{Hermitization} 

Recall the Hermitization $H_N$ and its resolvent $R_N(q)$ defined in Section \ref{sec:hermitization}.

In this section, for any matrix $H$ with entries that are $2 \times 2$ blocks, we mean $\tr_N(H) = \frac{1}{N} \sum_{i} H_{ii}$ where $H_{ii}$ is the $i^{th}$ diagonal $2 \times 2$ block of $H$. When working with $N \times N$ matrices with entries that are $2 \times 2$ blocks, we use superscripts to refer to entries of the $2 \times 2$ blocks. Additionally, when forming an $N \times N$ matrix whose $ij^{th}$ entry is the $ab^{th}$ entry ($a,b \in \{1,2\}$) of the $ij^{th}$ $2 \times 2$ blocks we also use superscripts. For example, $R^{21}$ is the $N \times N$ matrix formed from taking each $R_{ij}$ block and replacing it by its (2,1)-entry.  

Let $\Gamma_N(q): = \tr_N(R_N )$. By the symmetry of the matrix $H_N$, $\sum_{l} R^{22}_{ll} = \sum_{l} R^{11}_{ll}$, i.e. $\Gamma^{11}_N = \Gamma^{22}_N$. Let $a_N(q) := \Gamma^{11}_N(q) $, $b_N(q) := \Gamma^{12}_N(q) $ and $c_N(a) := \Gamma^{21}_N(q) $.

From the calculations in \cite{NgO} (see also \cite{Nell}), it follows that $\Gamma_N(q)$ converges almost surely to a limit 
$$\Gamma(q):=\begin{pmatrix} a(q) &b(q) \\ c(q) & a(q) \end{pmatrix} $$ 
for each fixed $q$. This block matrix Stieltjes transform satisfies the fixed point equation
\begin{align}
\label{defGamma}
 \Gamma(q) = - ( q + \Sigma( \Gamma(q)) )^{-1}, 
 \end{align}
where $\Sigma$ is the operator on $2 \times 2$ matrices defined by
\[\Sigma \begin{pmatrix} a & b \\ c & d \end{pmatrix} :=\begin{pmatrix} d & \rho  c \\ \rho b & a \end{pmatrix}. \]
The fixed point equation should be viewed as a matrix version of \eqref{eq:stab}.
 For more information on the use of this block matrix resolvent, we refer the reader to \cite{BC,BCC} and the references within.

For a $N \times N$ matrix $A$, let $\nu_A$ denote the symmetric empirical measure built from the singular values of $A$.  That is, 
\begin{equation} 
\label{nu}
\nu_A := \frac{1}{2N} \sum_{i=1}^N ( \delta_{\sigma_i(A)} + \delta_{-\sigma_i(A)} ), 
\end{equation}
where $\sigma_1(A) \geq \cdots \geq \sigma_N(A)$ are the singular values of $A$. The measure $\nu_A$ is also the empirical spectral measure of the Hermitization of $A$. It was established in \cite{NgO} that $\nu_{X_N-zI}$ converges almost surely to a probability measure $\nu_z$ as $N \rightarrow \infty$.  In Appendix \ref{app:prop}, we study the properties of $\Gamma(q)$ and $\nu_z$.   In particular, we will establish the following bound on the support of $\nu_z$ when $z$ is outside the ellipsoid.

\begin{theorem} \label{thm:support}
Fix $-1 < \rho < 1$ and let $\delta > 0$.  Then there exists $c > 0$ such that
$$ \nu_z( [-c,c]) = 0 $$
for all $z \in \mathbb{C}$ with $\dist(z,\mathcal{E}_\rho) \geq \delta$.  
\end{theorem}

\begin{remark} \label{rem:rho1}
Theorem \ref{thm:support} also holds when $\rho = \pm 1$.  When $\rho = 1$, $Y_N$ is a real symmetric Wigner matrix, and the singular values of $\frac{1}{\sqrt{N}} Y_N - z I$ are given by
$$ \left| \frac{1}{\sqrt{N}} \lambda_1(Y_N) - z \right|, \ldots, \left| \frac{1}{\sqrt{N}} \lambda_N(Y_N) - z \right|, $$
where $\lambda_1(Y_N), \ldots, \lambda_N(Y_N)$ are the eigenvalues of $Y_N$.  In this case, a lower bound on the singular values follows from \cite[Chapter 5]{BSbook}.  The $\rho=-1$ case can be obtained by symmetry.  
\end{remark}

\begin{remark}
We give a complete proof of Theorem \ref{thm:support} in Appendix \ref{app:prop}.  We quickly describe an alternative proof using techniques from free probability.  From \cite{NgO} and the work of Voiculescu \cite{Voic}, one can study the limiting measure $\nu_z$ by considering the distribution of 
$$ A_z := (\sqrt{\rho} S + \sqrt{1-\rho}C - zI)(\sqrt{\rho} S + \sqrt{1-\rho}C - zI)^\ast $$
for $0 \leq \rho \leq 1$, where $S$ and $C$ are free non-commutative random variables, $S$ is a semi-circular variable, and $C$ is a circular variable.  Indeed, Biane and Lehner  \cite{BL} showed that the spectrum of $\sqrt{\rho} S + \sqrt{1-\rho}C$ is the ellipsoid $\mathcal{E}_\rho$.  Therefore, for any $z \in \mathbb{C}$ with $\dist(z,\mathcal{E}_\rho) \geq \delta$, it follows that $0$ is not in the spectrum of $A_z$, and hence $0$ is not in the support of the distribution of $A_z$.  A continuity argument then implies that for any $M>0$, there exists some $c>0$ such that for a.e. $z \in \mathbb{C}$ with $|z| \leq M$ and $\dist(z,\mathcal{E}_\rho) \geq \delta$, we have $\nu_z([-c,c]) = 0$.  
\end{remark}

Proving Theorem \ref{thm:least:truncfix} is equivalent to showing that a.s. $\nu_{ X_N-zI}([0,c]) =0$ for some $c>0$. By Theorem \ref{thm:support}, we choose $c$ such that $\nu_z([0,2 c]) =0$.  In order to show that $\nu_{ X_N-zI}([0,c]) =0$ we will show that $a_N(q)$ is close to $a(q)$ for $q$ as in \eqref{qdef} with $\eta$ = $E +\sqrt{-1}t$, $E\in [0,c]$ and $t$ sufficiently small. As $a_N(q)$ and $a(q)$ are the Stieltjes transform of $\nu_{X_N-zI}$ and $\nu_z$, respectively, at the point $\eta$ this will allow us to compare the two measures. The equations involving $a_N(q)$ depend crucially on $b_N(q)$ and $c_N(q)$ so it is actually more straightforward to show $\Gamma_N(q)$ is close to $\Gamma(q)$. We should note that the empirical spectral measure, $\mu_{X_N}$, can be recovered by the formula $-\pi \mu_{X_N} = \lim_{\eta=\sqrt{-1}t\to 0} \partial_z b_N$. This formula only uses purely imaginary $\eta$. We consider more general $\eta$ and a connection to the empirical spectral measure does not seem to be available.

In order to show that almost surely there are no singular values of $X_N-zI$ less than $c$, we follow the ideas of Bai and Silverstein \cite{BS}. First we prove an a priori bound on $\Gamma_N(q) - \Gamma(q)$, then use martingale inequalities to bound $\Gamma_N(q) - \E[\Gamma_N(q)]$, and finally bound $\E[\Gamma_N(q)] - \Gamma(q)$. Because of the correlations between $X_{ij}$ and $X_{ji}$ we don't directly study the Stieltjes transform of the empirical spectral measure of $(X_N-zI)^\ast(X_N-zI)$, but instead consider the linearized problem and study $\Gamma_N(q)$. Similar linearization tricks have been used to study eigenvalues of polynomials of Wigner matrices (e.g. see \cite{A,HT}).

Since the vector space of $2 \times 2$ matrices is finite dimensional, all norms on it are equivalent. Therefore the use of $\| \cdot \|$ in the this section can be any norm, but the reader might find it useful to think of it is the max of each entry of the matrix. In order to show a $2 \times 2$ matrix converges, it suffices to show that each entry of the matrix converges. We will often employ this strategy.

We conclude this section with some useful matrix identities and notation.

We write $H_i$ to be the $i^{th}$ column (of $2 \times 2$ blocks) of $H_N$ and $H_i^{(i)}$ to be the $i^{th}$ column of $H_N$ with the $i^{th}$ block removed. We let $R_N^{(i)}$ be the resolvent of $H_N$ where the $i^{th}$ row and $i^{th}$ column of $H_N$ (viewed as an $N \times N$ matrix of $2 \times 2$ blocks) have been removed. Finally $\Gamma_N^{(i)}(q) := \frac{1}{N} \sum_{j \not = i} R_{jj}^{(i)}$.



\subsection{A priori estimate}
\label{aprioriest}


Following the ideas of \cite{BS} we begin with an a priori bound on $\Gamma_N(q) - \Gamma(q)$ for $\eta = E+ \sqrt{-1} t_N$, with $t_N$ going to zero polynomially and $E \in [0,c]$. This gives an upper bound on the number of singular values of $X_N-zI$ less than $c$. In the second and third steps we use this bound to show that a.s. $\E[\Gamma_N(q)] - \Gamma_N(q) = o((t_N N)^{-1})$ and $\E[\Gamma_N(q)] -\Gamma(q) = O(N^{-1})$

By Schur's Complement, the diagonal entries of the resolvent are
\begin{align*}
R_{ii} &= -(q + H_{i }^{(i)} R_N^{(i)} H_{ i}^{(i)})^{-1} \\
&=  - ( q + \Sigma(\Gamma_N) - \Sigma(\Gamma_N) + \Sigma(\Gamma_N^{(i)})   - \Sigma(\Gamma_N^{(i)}) + H^{(i)*}_i R_N^{(i)} H^{(i)}_i )^{-1}
\end{align*}
Recall that the diagonal elements of $X_N$ and hence $H_N$ have been set to zero.
Let
\begin{equation}
\label{gammahat}
 \widehat \gamma_N^{(i)} := H_i^{(i)*} R_N^{(i)} H_i^{(i)} - \Sigma(\Gamma_N^{(i)}).
\end{equation}
 Summing over $i$ gives the trace:
\begin{align}
\label{gammaapexp}
\Gamma_N(q) 
&=\sum_{i} - (q + \Sigma(\Gamma_N(q)) -\Sigma(\Gamma_N(q)) + \Sigma(\Gamma_N^{(i)}(q))    + \widehat \gamma_N^{(i)} )^{-1}
\end{align}

Let $\mathcal S_N$ be a $N^{-1}$-net of the interval $[0,c]$.   Clearly $|\mathcal S_N| = O(N)$.

\begin{lemma} 
\label{epsbound}
There exist some $\alpha,\beta>0$ such that if $q$ is as in \eqref{qdef} with $t_N \geq N^{-\beta}$, then almost surely
\[ \sup_{1\leq i\leq N, E \in S_N } \| \Sigma(\Gamma_N(q)) - \Sigma(\Gamma_N^{(i)}(q))   -\widehat \gamma_N^{(i)}\| = O(N^{-\alpha}).\]
\end{lemma}
The proof will show that we can take $\alpha=1/3$ and $\beta = 1/16$; these values are not optimal, but are sufficient for our purposes. We will require that $\alpha + \beta < 1/2$ and $\beta < \alpha$.

\begin{proof}


We begin by showing
\begin{equation}
\label{rank2}
\|\Sigma( \Gamma_N - \Gamma_N^{(i)}) \| = O((N t_N)^{-1}).
\end{equation}
Since $\Sigma$ is a bounded operator it suffices to show 
$ \|\Gamma_N - \Gamma_N^{(i)} \| = O((N t_N)^{-1}).$

We define the modified resolvent $\check R_N^{(i)}(q) := (q - (H_N - e_i H_i^* - H_i e_i^{*}) )^{-1}$. Where $e_i$ is the $N \times 1$ vector whose $i^{th}$  $2 \times 2 $ block is the identity matrix and whose other entries are zero. The difference between the trace of $\check R_N^{(i)}(q)$ and the trace of $R_{N}^{(i)}(q)$ is $q^{-1}$. The difference between the trace of $R_N(q)$ and of $\check R_N^{(i)}(q)$ is bounded by the operator norm of $R_N(q) - \check R_N^{(i)}(q)$ times its rank.

The matrix $R_N(q) - \check R_N^{(i)}(q)$ has rank at most 4 (viewed as a $2N$ by $2N$ matrix).  Indeed, by the resolvent identity,
\[ R_N(q) - \check R_N^{(i)}(q) =  R_N(q) ( e_i H_i^* + H_i e_i^{*})  ) \check R_N^{(i)}(q). \]
The trivial bound on the resolvent, \eqref{symresest}, shows the operator norm of the difference is bounded by $2 t_N^{-1}$.

Thus, we obtain the estimate 
\[ \| \Gamma_N - \Gamma_N^{(i)} \| = \frac{1}{N} \| \tr(R_N(q)) - \tr(\check R_N^{(i)}(q)) + \tr(\check R_N^{(i)}(q)) - \tr(R_N^{(i)}(q)) \|= O((N t_N)^{-1}). \] 
Since we assume $ t_N \geq N^{-\beta}$, this term is deterministically bounded by $C N^{\beta-1}$ uniformly for $E \in [0,c]$ and $1 \leq i \leq N$.

We now bound $\|\widehat \gamma_N^{(i)}\| $ by applying the bound on quadratic forms (Lemma \ref{lemma:lde}) to each entry of this block.
\begin{align}
\label{gammoment}
&\E [|\widehat  \gamma_N^{(i)ab}|^p ] \leq \frac{ K_p }{N^p} \E\left[ (\tr(R^{(i)a'b'}(R^{(i)a'b'})^*))^{p/2}\right] \notag \\
&\leq \frac{ K_p t_N^{-p} }{N^{p/2} } 
\end{align}
with $a$ and $b$ either 1 or 2, and $a'=a+1 \pmod 2$, $b'=b+1 \pmod 2$.  The final estimate uses that $N$ times the operator norm of a self adjoint matrix bounds its trace. The trivial bound shows the operator norm is bounded by $t_N^{-2}$.

Then by Chebyshev's inequality and the union bound
\begin{align*}
\P(\max_{1\leq i\leq N, E_j \in S_N} \|\widehat \gamma_N^{(i)} \| \geq N^{-\alpha} ) &\leq \sum_{1\leq i\leq N, E_j \in S_N} N^{p \alpha} \E( \| \widehat \gamma_N^{(i)} \|^p ) \notag\\
&\leq  K_p   N^{2+ p (\alpha+\beta-1/2)} 
\end{align*}

In order for this term to converge to zero we require that $\alpha+\beta < 1/2$.  Then, $p$ can be chosen large enough to make the right-hand side summable. An application of the Borel-Cantelli lemma implies almost sure convergence. 

Since $\alpha+\beta < 1/2$ implies that $ \beta - 1 < - \alpha$, we conclude the proof of the lemma.

\end{proof}

Now we state and prove our a priori bound.

\begin{lemma}
\label{apriorilemma}
Let $q$ be as in \eqref{qdef} with $t_N \geq N^{-\beta}$ and $\beta$ as in Lemma \ref{epsbound}. Then almost surely
\[ \sup_{E \in [0,c]} \|\Gamma_N(q) - \Gamma(q)\| = o( N^{-\beta}) \]
\end{lemma}
\begin{proof}
First note that it is sufficient to prove the estimate on $\mathcal S_N$. If $|E - E'| \leq \delta$, then $\|\Gamma_N(q) - \Gamma_N(q')\|= \|\Gamma_N(q) ( q-q') \Gamma_N(q') \| \leq \delta t_N^{-2}$. Therefore showing that $\|\Gamma_N(q)\| = O(N^{-\beta})$ for $E \in \mathcal S_N$ with $t_N> N^{-\beta}$ implies the bound $ \|\Gamma_N(q')\| =  O(N^{-\beta})$  for all $q'$ with $E' \in [0,c]$.


We introduce the notation
\[\epsilon_N^i := -  \Sigma(\Gamma_N(q)) + \Sigma(\Gamma_N^{(i)}(q))  + \widehat \gamma_N^{(i)} \]
and
\[\epsilon_N := \frac{1}{N} \sum_{i=1}^N (q + \Sigma(\Gamma_N(q) ))^{-1} (\epsilon_N^i) (q + \Sigma(\Gamma_N(q) ) + \epsilon_N^i)^{-1}.\] 
Let $\Lambda_N$ be the event that $\max_{1\leq i\leq N} \| \epsilon_N^i \| \leq N^{-\alpha}$.  By Lemma \ref{epsbound}, $\oindicator{\Lambda_N}=1$ almost surely. 
With this notation we rewrite \eqref{gammaapexp} as
\begin{equation}
\label{approxsceq}
\Gamma_N(q) = -(q + \Sigma( \Gamma_N(q)) )^{-1} + \epsilon_N. 
\end{equation}

For sufficiently large $N$ we have the bound,
\[ \| 1/N \sum_i  (\epsilon_N^i) (q + \Sigma(\Gamma_N(q) ) + \epsilon_N^i)^{-1}    \| \oindicator{\Lambda_N}=   \|1/N \sum_{i=1}^{N} (\epsilon_N^i) R_{ii} \|\oindicator{\Lambda_N} =O( N^{-\alpha} t_N^{-1}) < 1/2. \]
Thus, we can solve for $-(q + \Sigma(\Gamma_N(q)))^{-1}\oindicator{\Lambda_N}$:
\[-(q + \Sigma(\Gamma_N(q)))^{-1} \oindicator{\Lambda_N} = \Gamma_N(q) (1 - 1/N \sum_{i=1}^{N} (\epsilon_N^i) R_{ii})^{-1}  \oindicator{\Lambda_N},  \]
and conclude
\[ \|(q + \Sigma(\Gamma_N(q)))^{-1} \| \oindicator{\Lambda_N} \leq C t_N^{-1} .\]
This leads to the bound
\[ \| \epsilon_N \| \oindicator{\Lambda_N} = \|\frac{1}{N} \sum_{i=1}^N (q + \Sigma(\Gamma_N(q) ))^{-1} (\epsilon_N^i) R_{ii} \| \oindicator{\Lambda_N} \leq C N^{-\alpha} t_N^{-2}.  \]

The following lemma will allow us to complete the proof.

\begin{lemma}
\label{trieste}
Let $\tilde q=\tilde q_N: = q + \Sigma(\epsilon_N)$ with $q$ as in \eqref{qdef}, $E \in [0,c ]$, and $t \geq N^{-\beta}$.  Then almost surely
\[ \Gamma_N(q) = \Gamma(\tilde q) + \epsilon_N. \]
\end{lemma}

We defer the proof of this lemma until the end of the current proof.

Now, assuming Lemma \ref{trieste}, almost surely, we can replace $ \Gamma(q)$ with $\Gamma_N(q-\Sigma(\epsilon_N)) - \epsilon_N$ and conclude
\begin{align*}
\|\Gamma_N(q) - \Gamma(q) \| &= \|\Gamma(\tilde q) + \epsilon_N - \Gamma(q)  \| \\
&\leq \|\Gamma(\tilde q)\left(\Sigma(\epsilon_N ) \right)\Gamma(q) \| + \| \epsilon_N \| \\
&\leq \|\Gamma(\tilde q)\| \| \epsilon_N\|  \| \Gamma(q) \| + \| \epsilon_N \| \\
&= O(t_N^{-4} N^{-\alpha}).
\end{align*}
Choosing $\alpha=1/3$ and $\beta=1/16$ gives the almost sure bound $\sup_{E  \in \mathcal{S_N} } \|\Gamma_N(q) - \Gamma(q) \| = o(N^{-\beta})$.
\end{proof}

\begin{proof}[Proof of Lemma \ref{trieste}]
Applying the resolvent identity to the difference between \eqref{defGamma} and \eqref{approxsceq} leads to
\begin{align*}
\Gamma_N(q)-\Gamma(\tilde q) - \epsilon_N &=   -(q + \Sigma(\Gamma_N(q) ))^{-1} + ( \tilde q + \Sigma(\Gamma( \tilde q)))^{-1}\\
&= -(q + \Sigma(\Gamma_N(q)) )^{-1} \left(\tilde q - q + \Sigma(\Gamma(\tilde q) -\Gamma_N(q)  )    \right) ( \tilde q + \Sigma(\Gamma(\tilde q)))^{-1} \\
&= (q + \Sigma(\Gamma_N(q)) )^{-1}\left( \Sigma(\Gamma_N(q) - \Gamma(\tilde q) - \epsilon_N )    \right) ( \tilde q + \Sigma(\Gamma( \tilde q)))^{-1} 
\end{align*}

From the \eqref{eq:CCsupbnd} and Remark \ref{bcbound} there exist a $K$ such that
\[ \|( q + \Sigma(\Gamma(  q)))^{-1}\| =\| \Gamma(q) \| \leq K\]
for all $q$ with $t > 0$.

Since  $ \tilde q  \oindicator{\Lambda_N} = (q + \Sigma( \epsilon_N) ) \oindicator{\Lambda_N} $ and $ \|\Sigma( \epsilon_N) \| \oindicator{\Lambda_N} $ converges to zero, we can choose $N$ sufficiently large such that the imaginary part of the diagonal entries of $\tilde q$ are almost surely greater than zero, yielding the almost sure bound $\| \Gamma(\tilde q) \| \leq K$.

 Then using that $\|\Gamma_N(q)\| \leq t_N^{-1}$, we obtain almost surely
\[\| \Gamma_N(q)-\Gamma(\tilde q) - \epsilon_N \| \leq K |t_N|^{-1} \| \Gamma_N(q)-\Gamma(\tilde q) - \epsilon_N \|. \]
We conclude for $\eta$ such that $ 1 >  K |t_N|^{-1}$, $E \in [0,c]$  and then for all $\eta$ with $t_N > N^{-\beta}$ and $E \in [0,c]$ by analytic continuation that almost surely
\[ \Gamma_N(q) = \Gamma(\tilde q ) + \epsilon_N. \]
\end{proof}

Now we define the Stieltjes transforms of the measure $\nu_{X_N-zI}$ and $ \nu_{z}$ restricted to $[-2c,2c]$ and its complement to be 
\[ a^{in}_N(q) = \int_{[-2c,2c]^c} \frac{ d \nu_{X_N-zI}(x) }{x-\eta}, \qquad a^{in}(q) = \int_{[-2c,2c]^c} \frac{ d \nu_{z}(x) }{x-\eta} \]

\[ a^{out}_N(q) = \int_{[-2c,2c]} \frac{ d \nu_{X_N-zI}(x) }{x-\eta},\qquad a^{out}(q) = \int_{[-2c,2c]} \frac{ d \nu_{z}(x) }{x-\eta} =0. \]

Note that $t_N^{-1} \Im( a^{in}_N(q) ) =  \int_{[-2c,2c]^c} \frac{ d \nu_{X_N-zI}(x) }{(x-E)^2+ t_N^2}$, and observe that  $\frac{ 1}{(x-E)^2+ t_N^2}$ forms a uniformly bounded, equicontinuous family as a function of $E \in [0,c]$ for $x \not \in [-2c,2c] $.
Furthermore, since $\nu_{X_N-zI}$ converges almost surely to $\nu_z$ by the calculations in \cite{NgO} (see also \cite{Nell}), we can conclude that a.s.
 \[ \sup_{E \in \mathcal S_N}  t_N^{-1} (\Im( a^{in}_N(q) )  - \Im( a^{in}(q) ) ) \longrightarrow 0.\]
Combining this estimate with Lemma \ref{apriorilemma} gives that a.s.
\[ \sup_{E \in \mathcal S_N}  t_N^{-1}  \Im(a^{out}_N(q) - a^{out}(q) ) \longrightarrow 0. \]

We conclude this section with a bound on the number of singular values less than $c$ and then turn this into a bound on the trace of the resolvent. Let $\mathcal{T}_N$ be an $t_N$-net of $[0,c]$. Using the inequality $\oindicator{[E_j - t_N, E_j + t_N]}(x) \leq -2 t_N \Im( 1/(x-E_j +\sqrt{-1}t_N)  )$, we obtain
\begin{align*}
&\nu_{X_N-zI}([0,c])  
\leq \sum_{E_j \in \mathcal{T}_N}  \int_{E_j-t_N}^{ E_j+t_N} d \nu_{X_N-zI} (x) \\
& \leq \sum_{E_j \in \mathcal{T}_N} \int_0^c \frac{2 t_N^2}{(E_j-x)^2 + t_N^2}   d \nu_{X_N-zI} (x)\\
& \leq C t_N^{-1} t_N  \sup_{E_j \in \mathcal{T}_N}  |\Im (a^{out}_N(q) - a^{out}(q))| = o(N^{-\beta}) \quad \text{a.s.}
\end{align*}


So we conclude on the almost sure event $\Lambda_N$ that there are $o(N^{1-\beta})$ eigenvalues in the interval $[0,c]$. We will require a similar a priori bound on the number of small eigenvalues for the $N-1 \times N-1$ submatrices $X_N^{(i)}$, defined by removing the $i^{th}$ row and column of $X_N$. Thus, we define the event $\Lambda_i$ that $\nu_{X_N^{(i)}-zI}([0,c])= o( N^{-\beta})$. By the interlacing theorem, $\Lambda_N \subset \Lambda_i$, so $\Lambda_i$ also occurs almost surely.

 For $k=1, \ldots, N$, let $\E_{k}$ be averaging with respect to the first $k$ rows and columns of $X_N$, and let $\E_0$ be the identity. Since $ \nu_{X_N-zI}([0,c])$ is bounded, almost surely $o(N^{-\beta})$ and $\E_k[ \nu_{X_N-zI}([0,c])]$ forms a martingale, we can apply Lemma \ref{Baimart} to obtain the almost sure estimate
\begin{equation} \label{monster}
 \max_{k \leq N} \E_{k} [\nu_{X_N-zI}([0,c]) ]= o(N^{-\beta}).
 \end{equation}
Repeating the argument shows this estimate also holds for $  \max_{k \leq N} \E_{k} [\nu_{X_N^{(i)}-zI}([0,c]) ] $.


Now we use the spectral theorem to turn this bound on the number of singular values into a bound on the trace of powers of the resolvent. In order for this bound to be useful we will increase the imaginary part of $\eta$.

\begin{lemma} Let $q$ be as in \eqref{qdef} with $E \in [0,c ]$, $t = N^{-\beta/4}$ and $\beta$ as in Lemma \ref{epsbound}.  Then
\begin{align}
\label{resbound}
&\max_{k\leq N} \E_k[\tr_N(|R_N^{ab}|^2)\oindicator{\Lambda_N}] = O(1),\qquad \max_{k\leq N} \E_k[ \tr_N(|(R_N R_N)^{ab}|^2)\oindicator{\Lambda_N}] = O(1)\notag \\
&\max_{i,k\leq N} \E_k[ \tr_N(|R_N^{(i)ab}|^2)\oindicator{\Lambda_i}] = O(1), \qquad \max_{i,k\leq N} \E_k[\tr_N(|(R_N^{(i)} R_N^{(i)})^{ab}|^2)\oindicator{\Lambda_i}] = O(1).
\end{align}
\end{lemma}

\begin{proof} 

Let $X_N -z$ have singular value decomposition $U_N D_N V_N$ with $D_{ii} = \sigma_i = \sigma_{i}(X_N - z)$, then the block matrix
\[ \begin{pmatrix} R_N^{11} & R_N^{12} \\ R_N^{21} & R_N^{22} \end{pmatrix}^p = \frac{1}{2} \begin{pmatrix} U_N & U_N \\ V_N^* & -V_N^{*} \end{pmatrix} \begin{pmatrix} (D_N -\eta)^{-1} & 0 \\0& (-D_N-\eta)^{-1} \end{pmatrix}^p  \begin{pmatrix} U_N^* & V_N \\ U_N^* & -V_N \end{pmatrix}  \]


\begin{align*}
\max_{k\leq N} \E_k[ \tr_N(|(R_N^p)^{11}|^2)]\oindicator{\Lambda_N} 
=&  \max_{k\leq N} \E_k[ \sum_{i=1}^N \sum_{j=1}^N U_{ij} \left| 
\left(\frac{ 1 }{ \sigma_j - \eta  }\right)^p + \left(\frac{ 1 }{-\sigma_j - \eta  }\right)^p  \right|^2 U_{ji}^*     ]\oindicator{\Lambda_N}\\
\leq &  \max_{k\leq N} \E_k[ \frac{C}{N}\sum_{ |\sigma_j| \leq c} \frac{1}{|\eta-\sigma_j|^{2p}}+\frac{C}{N}\sum_{|\sigma_j| \geq c} \frac{1}{|\eta-\sigma_j|^{2p}}]\oindicator{\Lambda_N} \\
\leq& C( N^{-1}  (N^{1-\beta}) t_N^{-2p} + N^{-1} N c^{-1} ).
\end{align*}

\begin{align*}
\max_{k\leq N} \E_k[ \tr_N(|R_N^{12}|^p)]\oindicator{\Lambda_N} 
=&  \max_{k\leq N} \E_k[ \sum_{i=1}^N \sum_{j=1}^N U_{ij} \left| 
\left(\frac{ 1 }{ \sigma_j - \eta  }\right)^p - \left(\frac{ 1 }{-\sigma_j - \eta  }\right)^p  \right|^2 V_{ji}     ]\oindicator{\Lambda_N}\\
\leq & \max_{k\leq N} \E_k[ \sum_{j=1}^N \left| \frac{\sigma_j }{\eta^2 - \sigma_j^2} \right|^{2p} \| U_{\cdot j}\|_2 \| V_{j \cdot }\|_2 ] \oindicator{\Lambda_N} \\  
\leq& C( N^{-1}  (N^{1-\beta}) t_N^{-2p} + N^{-1} N c^{-1} ).
\end{align*}

This term is $O(1)$ if $p=1,2$ because $t_N = N^{-\beta/4}$. The same argument bounds the $21$ and $22$ term.
The above computation also verifies the lemma when a row or column has been removed.

\end{proof}

\subsection{Estimating  $\Gamma_N - \E[\Gamma_N] $}

\begin{theorem} \label{gammaminusexp} Let $q$ be as in \eqref{qdef} with $t = N^{-\beta/4}$ and $\beta$ as in Lemma \ref{epsbound}.  Then almost surely
\[ \sup_{E \in [0,c]}  N t_N (\Gamma_N(q) - \E[\Gamma_N(q)]) \to 0. \]

\end{theorem}

Before proceeding with the proof, we define the relevant notation and give a lemma containing crude estimates.

Applying \eqref{trident} to $R_N$ (which we view as an $N \times N$ matrix) yields
\begin{align}
\label{r2comp}
\Gamma_N - \Gamma_N^{(i)} =& -\frac{1}{N} (q + H_i^{(i)*} R_N^{(i)} H_i^{(i)} )^{-1} ( I_2 + H_i^{(i)*} R_N^{(i)}R_N^{(i)} H_i^{(i)})
\end{align}

Note that by Schur's Complement \eqref{Schur1} the first term is an entry of the resolvent:
\begin{align*}
R_{ii} &= - ( q + H_i^{(i)*} R_N^{(i)} H_i^{(i)} )^{-1}.
\end{align*}
We define
\[ \zeta_N^{(i)} := ( I_2 + H_i^{(i)*} R_N^{(i)}R_N^{(i)} H_i^{(i)}). \]

In order to study $R_{ii}$ we introduce the non-random $2 \times 2$ matrix
\begin{align*}
\label{Rhat}
\widehat R_{ii}:= -(q + \E[\Sigma(\Gamma_N^{(i)})])^{-1}. 
\end{align*}
Note that this is not actually an entry of a resolvent. 
In order to control the fluctuations of $R_{ii}$, we use the resolvent identity to compare the $R_{ii}$ with $\widehat R_{ii}$:
\begin{equation}
\label{Rtohat}
R_{ii} - \widehat R_{ii} = \widehat R_{ii} 
(H_i^{(i)*} R_N^{(i)} H_i^{(i)} - \E[\Sigma(\Gamma_N^{(i)})])
 R_{ii} .
\end{equation}
This motivates the definition 
\begin{equation}
\label{gamma}
 \gamma_N^{(i)} := H_i^{(i)*} R_N^{(i)} H_i^{(i)} - \E[\Sigma(\Gamma_N^{(i)})] .
\end{equation}
We remind the reader
\begin{equation*}
 \widehat \gamma_N^{(i)} := H_i^{(i)*} R_N^{(i)} H_i^{(i)} - \Sigma(\Gamma_N^{(i)}).
\end{equation*}

Redefine $\mathcal S_N$ to be a $N^{-2}$-net of the interval $[0,c]$. Once again it suffices to prove the theorem for $E \in \mathcal S_N$.

\begin{lemma} For $a,b \in \{1,2\}$, and $p \geq 2$:

\begin{equation}
\label{crcont}
\E[|\Gamma_N^{(i)ab}(q) - \E[\Gamma_N^{(i)ab}(q)]|^{p}] \leq K_p N^{-p/2} t_N^{-p}
\end{equation}
\begin{equation}
\label{crquad}
\E [| \gamma_N^{(i)ab} |^p ] \leq \frac{ K_p  t_N^{-p} }{N^{p/2} } .
\end{equation}

There exist some $K$ such that for all large $N$,
\begin{equation}
\label{crhR}
\| \widehat{R}_{ii} \| \leq K .
\end{equation}

\end{lemma}

We note that part of the use of the first inequality is the equality $\gamma_N^{(i)} - \widehat \gamma_N^{(i)} = \Sigma(\Gamma_N^{(i)}(q) - \E[\Gamma_N^{(i)}(q)])$.

\begin{proof}
Using the martingale inequality, \eqref{lemma:burkholder}, and the bound on a trace of a matrix and a submatrix, \eqref{rank2}, we can bound any entry as 
\begin{align*}
\E[|\Gamma_N^{(i)ab}(q) - \E[\Gamma_N^{(i)ab}(q)]|^{p}] &= \E[|\sum_{j \not= i} (\E_{j-1} - \E_{j})(\Gamma_N^{(i)ab}-\Gamma_N^{(i,j)ab})|^{p}] \notag \\
&\leq K_p \E[ \sum_{j \not = i} |(\E_{i} - \E_{i-1})(\Gamma_N^{(i)ab}-\Gamma_N^{(i,j)ab}|^{2})^{p/2}] \notag \\
&\leq K_p N^{-p/2} t_N^{-p}.
\end{align*}

Combining this estimate with \eqref{gammoment} leads to
\begin{equation*}
\E [| \gamma_N^{(i)ab} |^p ] =  \E [|\widehat \gamma_N^{(i)ab} +\Sigma(\Gamma_N^{(i)ab}(q)) - \Sigma(\E[\Gamma_N^{(i)ab}(q)]) |^p ]  \leq \frac{ K_p  t_N^{-p} }{N^{p/2} } 
\end{equation*}


To bound $\widehat R_{ii}$, we begin with
\begin{align*}
\E[R_{ii} ]&= \E[\Gamma_N].
\end{align*}
Since $\| \E[\Gamma_N] - \Gamma \|=  o(t_N)$ and $\Gamma$ is bounded by some constant $K$, so is $\E[R_{ii} ]$ Combining this estimate with the trivial bound, $|R_{ii}| \leq |t_N|^{-1}$, leads to:
\begin{align*}
 \| &\widehat{R}_{ii} - \E[R_{ii} ] \|  \\
& = \left\| \E \left[ \widehat R_{ii} \gamma_{N}^{(i)} \E[R_{ii} ] \right] \right\| \\
&\leq  \Im(\eta)^{-1} \E\left[  \left\|  \gamma_{N}^{(i)}    \right\|     \right] K\\
&\leq  \Im(\eta)^{-3} N^{-1/2} K
\end{align*}
So
\[\| \widehat R_{ii}  \| \leq K +  \Im(\eta)^{-3} N^{-1/2} K\]
The last term is bounded for $\eta$ in our domain, and the proof is complete.  
\end{proof}

\begin{proof}[Proof of Theorem  \ref{gammaminusexp}]

To control $\Gamma_N(q) - \E[\Gamma_N(q)]$, we rewrite it as a sum of martingale differences. Using the equality $\E_{i}[ \Gamma^{(i)}_N(q) ] = \E_{i-1}[\Gamma^{(i)}_N(q)]$ and the formula for the differences of traces of submatrices \eqref{r2comp}, we have
\begin{equation}
\label{Gammafluct}
\Gamma_N - \E[ \Gamma_N] = \sum_i (\E_{i-1} - \E_{i} )\Gamma_N =  \sum_i (\E_{i-1} - \E_{i} )(\Gamma_N -\Gamma_N^{(i)}) = \frac{1}{N} \sum_i (\E_{i-1} - \E_{i} ) R_{ii} \zeta_N^{(i)}.
\end{equation}

To complete the proof it suffices to show that for arbitrary $\epsilon > 0$,
\[ \P( \max_{E\in\mathcal{S}_N} N t_N \| \Gamma_N - \E[ \Gamma_N] \| \geq \epsilon, ~~~\text{i.o.}) =0.\]
Recalling that 
\[ \P( \cup_{i=1}^{N}\{\oindicator{\Lambda_i} =0\},~~~\text{i.o.}) = 0,\]
leads to the estimate
\begin{align}
\label{almostsureR}
&\P( \max_{E\in\mathcal{S}_N}  N t_N \| \Gamma_N - \E[ \Gamma_N] \| > \epsilon, ~~~\text{i.o.}) \nonumber \\
=& \P( ( \max_{E\in\mathcal{S}_N}  N t_N \| \Gamma_N - \E[ \Gamma_N] \| > \epsilon) \cap_{i=1}^{N} [\oindicator{\Lambda_i} =1] , ~~~\text{i.o.})  \nonumber \\
\leq &\P(  \max_{E\in\mathcal{S}_N} t_N \| \sum_{i=1}^N (\E_{i-1} - \E_{i} ) R_{ii} \zeta_N^{(i)}\| \oindicator{\Lambda_i} > \epsilon , ~~~\text{i.o.}). 
\end{align}

To estimate \eqref{almostsureR} we iteratively apply \eqref{Rtohat} leading to:
\begin{align*}
R_{ii} \zeta_N^{(i)} 
 =& \left( R_{ii} - \widehat R_{ii} \right) \zeta_N^{(i)} + \widehat R_{ii} \zeta_N^{(i)}\\
 &= \widehat R_{ii}   \gamma_N^{(i)}   R_{ii} \zeta_N^{(i)}  + \widehat R_{ii} \zeta_N^{(i)}\\
 &= \widehat R_{ii}   \gamma_N^{(i)}  \widehat R_{ii}  (I_2 + \Sigma( \tr_N( R^{(i)}_N R^{(i)}_N )) \\
 &+ \widehat R_{ii}   \gamma_N^{(i)} 
\widehat R_{ii} \left( H_i^{(i)*} R_N^{(i)}R_N^{(i)} H_{i}^{(i)} - \Sigma( \tr_N( R^{(i)}_N R^{(i)}_N )) \right) \\
&+ \widehat R_{ii}
   \gamma_N^{(i)}
  \widehat R_{ii}
   \gamma_N^{(i)} R_{ii} \zeta_N^{(i)} + \widehat R_{ii} \zeta_N^{(i)}.
\end{align*}

After applying this expansion to \eqref{almostsureR}, it suffices to show that each entry of the $2 \times 2$ block converges to zero almost surely. By the triangle inequality, it suffices to bound an arbitrary product of entries of the blocks in the expansion. For the remainder of this section, we use lower case superscripts starting with the beginning of the alphabet to denote the values $1$ or $2$.

To bound $\widehat R_{ii} \gamma_N^{(i)} \widehat R_{ii} (I_2+ \Sigma( \tr_N( R^{(i)}_N R^{(i)}_N )) )$ we apply Rosenthal's inequality (Lemma \ref{lemma:rosenthal}), the bound on moments of quadratic forms (Lemma \ref{lemma:lde}), the bound on $\widehat{R}_{ii}$ \eqref{crhR}, and the a priori bound \eqref{resbound}.  We obtain
\begin{align*}
&\E[ |t_N \sum_{i=1}^N (\E_{i-1} - \E_{i} ) ( \widehat R_{ii}^{ab} \gamma_N^{(i)bc} \widehat R_{ii}^{cd} (I_2+ \Sigma( \tr_N( R^{(i)}_N R^{(i)}_N )) )^{de} ) \oindicator{\Lambda_i} |^p ]\\
&\leq K_p t_N^p \E[ (\sum_{i=1}^N  \E_{i} | \widehat \gamma_N^{(i) bc}\oindicator{\Lambda_i} |^2  )^{p/2} + \sum_{i=1}^N |\widehat \gamma_N^{(i) bc} \oindicator{\Lambda_i} |^p ]\\
&\leq K_p t_N^p \E[ (\sum_{i=1}^N  N^{-2} \E_i[\tr(R^{b'c'(i)}R^{b'c'(i)*})\oindicator{\Lambda_i}] )^{p/2}] + K_pN N^{-p/2} t_N^{-p}\\
&\leq K_p t_N^p ( C + N^{1-p/2}t_N^{-p} ),
\end{align*}
which is summable for large $p$. Recall that $b' = b + 1 \pmod 2$.
The same estimates are used to bound the $\widehat R_{ii} \zeta^{(i)}$ term. 

In order to bound $ \widehat R_{ii}\gamma_N^{(i)} \widehat R_{ii}\gamma_N^{(i)} R_{ii} \zeta_N^{(i)}$, we begin with Burkholder's inequality (Lemma \ref{lemma:burkholder}) and then apply $R_{ii} \zeta_N^{(i)} = N(\Gamma_N(q) - \Gamma_N^{(i)}(q) ) \leq K t_N^{-1} $ by \eqref{rank2} and the bound on $\widehat R_{ii}$ given in \eqref{crhR} along with the Cauchy-Schwarz inequality and the estimate on quadratic forms \eqref{crquad}.

\begin{align*}
&\E[|t_N \sum_{i=1}^{N}(\E_{i-1} - \E_{i}) \widehat R_{ii}^{ab}  \gamma_N^{(i)bc} \widehat R_{ii}^{cd} \gamma_N^{(i)de} (R_{ii} \zeta_N^{(i)})^{ef} \oindicator{\Lambda_i}|^p] \\
&\leq K_p t_N^p t_N^{-p}  \left( \sum_{i=1}^{N} \E[| \gamma_N^{(i)bc} \gamma_N^{(i)de} \oindicator{\Lambda_i}|^2] \right)^{p/2} \\
&\leq K_p   \left( \sum_{i=1}^{N} (\E[| \gamma_N^{(i)bc}|^4] \E[|\gamma_N^{(i)de} |^4])^{1/2} \right)^{p/2} \\
&\leq K_p   N^{p/2}  N^{-p} t_N^{-2p} = K_p   N^{-p/2}  t_N^{-2p}.
\end{align*}

The estimate of the $ \widehat R_{ii}   \gamma_N^{(i)} 
\widehat R_{ii} \left( H_i^{(i)*} R_N^{(i)}R_N^{(i)} H_{i}^{(i)} - \Sigma( \tr_N( R^{(i)}_N R^{(i)}_N )) \right) 
$ term is done the same way.

Choosing $p$ large enough in the above estimates to make the right-hand sides summable and an application of Borel-Cantelli shows that almost surely
\[ \max_{E \in \mathcal{S}_N } N t_N\| \Gamma_N - \E[ \Gamma_N] \| \to 0 \]

\end{proof}

\subsection{Estimating $\E[\Gamma_N(q)] - \Gamma(q)$}

We now show that for $q$ as in \eqref{qdef} with $t_N = N^{-\beta/4}$ and $\beta$ as in Lemma \ref{epsbound}
\begin{equation}  \label{expgammaminusgamma}  
\sup_{E \in [0,c]}  \|\E[\Gamma_N(q)] - \Gamma(q)\| = O(N^{-1}). 
\end{equation}

We begin in a similar fashion to the a priori estimates with Schur's Complement,
\begin{align*}
\E[\Gamma_N(q)] &= \E[R_{11}] = - \E[(  q + H^{(1)*}_1 R_N^{(1)} H^{(1)}_1)^{-1}],
\end{align*}
from which we will subtract $-(q + \Sigma(\E[\Gamma_N(q)]))^{-1}$. We will apply the resolvent identity, add and subtract $\Sigma(\E[\Gamma^{(1)}_N(q)])$, and repeatedly apply the identity
\begin{align*}
 \E[\Gamma_N(q)] + (q+\Sigma(\E[\Gamma^{(1)}_N(q)]))^{-1} &= \E[-(q + H^{(1)*}_1 R_N^{(1)} H^{(1)}_1)^{-1} + (q+\Sigma(\E[\Gamma^{(1)}_N(q)]))^{-1}]\\
&= \E[\Gamma_N(q) \gamma_N^{(1)} (q+\Sigma(\E[\Gamma^{(1)}_N(q)]))^{-1}]. 
\end{align*}
Leading to the expansion
\begin{align*}
& \E[\Gamma_N(q) + (q + \Sigma(\E[\Gamma_N(q)]))^{-1}] \\
=&\E[\Gamma_N(z) \left(  H^{(1)*}_1 R_N^{(1)} H^{(1)}_1 -\Sigma(\E[\Gamma_N(q)])  \right)(q + \Sigma(\E[\Gamma_N(q)]))^{-1}]\\
=&\E[\Gamma_N(z) \left( \Sigma(\E[\Gamma_N^{(1)}]) - \Sigma(\E[\Gamma_N(q)])  \right)(q + \Sigma(\E[\Gamma(q)]))^{-1}] 
+ \E[\Gamma_N(z) \left(  \gamma_N^{(1)}  \right)(q + \Sigma(\E[\Gamma_N(q)]))^{-1}]\\
=&\E[\Gamma_N(z) \left( \Sigma(\E[\Gamma_N^{(1)}]) - \Sigma(\E[\Gamma_N(q)])  \right)(q + \Sigma(\E[\Gamma(q)]))^{-1}] \\
&+ \E[\widehat R_{11} \left( \gamma_N^{(1)}  \right)(q + \Sigma(\E[\Gamma_N(q)]))^{-1}]\\
&+\E[\widehat R_{11}  \left( \gamma_N^{(1)}   \right) \widehat R_{11} \left(  \gamma_N^{(1)} \right)(q + \Sigma(\E[\Gamma_N(q)]))^{-1}]\\
&+\E[\Gamma_N(q)\left( \gamma_N^{(1)}     \right) \widehat R_{11}  \left(  \gamma_N^{(1)} \right)  \widehat R_{11}
\left(  \gamma_N^{(1)}   \right)(q + \Sigma(\E[\Gamma_N(q)]))^{-1}].
\end{align*}

Note that the third line is zero, because $\E[\gamma_N^{(1)}  ] = 0$ and the other terms are non-random.
To bound the rest of the terms we need the following lemma:
\begin{lemma} Let $q$ be as in \eqref{qdef} with $E \in [0,c ]$, $t_N = N^{-\beta/4}$, and $\beta$ as in Lemma \ref{epsbound}. Then
\begin{align}
\|\E[\Gamma_N(q) - \Gamma_N^{(1)}(q)]\| &= O( N^{-1}), \\
\label{2ndmombound}
 \E[ |\gamma^{(1)ab}|^2]  &\leq K (\E[|\widehat \gamma^{(1)ab}|^2] + \E[|\widehat \gamma^{(1)ab}-\gamma^{(1)ab}|^2])  = O( N^{-1}),\\
\| (q + \Sigma(\E[\Gamma_N(q)]) )^{-1}  \|&=O(1).
\end{align}
\end{lemma}

Before proving the lemma note that \eqref{expgammaminusgamma} will follow from a straight forward application of this lemma, the triangle inequality, H\"older's inequality and the estimates
$\|\Gamma_N\| \leq t_N^{-1} $,
$\|\widehat R_{ii}\| \leq K$, and the estimate on quadratic forms \eqref{crquad}.

\begin{proof}
Using the formula for the difference between traces, \eqref{r2comp}, the bound on the trace of the resolvent, \eqref{resbound}, and the bound on quadratic forms, \eqref{eq:quadp}, we obtain
\begin{align*}
|\E[\Gamma_N(q)^{ab} - \Gamma_N^{(1)ab}(q)] |&=\frac{1}{N} |\E[ (R_{11} \zeta^{(1)}_N)^{ab} ] |\\ 
&=\frac{1}{N} |\E[(\widehat R_{11} \zeta^{(1)})^{ab}] + \frac{1}{N}  \E[(\widehat R_{11}  \gamma_N^{(1)} R_{11} \zeta_N^{(1)})^{ab} ] |\\
&\leq\frac{ K }{N}  \E[1 +  \|\tr_N( (R_N^{(1)})^{2} ) \| ]\\
&\qquad + \frac{K }{N} \sum_{c,d=1}^{2} \E[|  \gamma_N^{(1) cd} | ] K t_N^{-1} \\
&\leq \frac{K}{N}( 1  +  \frac{t_N^{-2}}{N^{1/2}}).
\end{align*}


The first term of \eqref{2ndmombound} is bounded from a direct calculation and
the second term uses the martingale difference decomposition and the expansions of the previous estimates.
\begin{align*}
&\E[|\widehat \gamma^{(1)ab}-\gamma^{(1)ab}|^2] \\
&= \E[|\Gamma_N^{(1)ab}(q) - \E[\Gamma_N^{(1)ab}(q)]|^2] = \sum_{i=2}^N \E[| (\E_{i}-\E_{i-1})(\Gamma_N^{(1)}(q) - \Gamma_N^{(1,i)}(q) )^{ab}|^2 ] \\
&= \frac{1}{N^2} \sum_{i=2}^N  \E[  | (\E_{i}-\E_{i-1})(R_{ii}^{(1)} \zeta_N^{(1,i)}  )^{ab}|^2 ]\\
&\leq \frac{1}{N^2}   K \sum_{i=2}^N   \E[  | (\E_{i}-\E_{i-1})(\widehat R_{ii}^{(1)}  \zeta_N^{(1,i)}  )^{ab}|^2 ]\\
&+K \frac{1}{N^2} \sum_{i=2}^N  \E[  | (\E_{i}-\E_{i-1})(\widehat R_{ii}^{(1)}  \gamma_N^{(i,1)}  R_{ii}^{(1)} \zeta_N^{(1,i)} )^{ab}|^2 ]\\
&\leq \frac{K}{N^2} \sum_{i=2}^{N} \left( \sum_{c=1}^2 \E[ \tr_N (|R^{(1,i)c'b'}(q)|^4)  ] + \sum_{c,d=1}^2 2 K t_N^{-2} \E[  |   \gamma_N^{(i,1)cd}|^2] \right) \\
&\leq \frac{1}{N}.
\end{align*}

%

The final estimate follows from
\[ \E[\Gamma_N(q)] + (q+\Sigma(\E[\Gamma^{(1)}_N(q)]))^{-1} = \E[\Gamma_N(q) \Big(  H^{(1)*}_1 R_N^{(1)} H^{(1)}_1 - \Sigma(\E[\Gamma^{(1)}_N(q)]) \Big) (q+\Sigma(\E[\Gamma^{(1)}_N(q)]))^{-1}] \]
and the boundedness of  $\E[\Gamma_N(q)]$.  The proof of the lemma is complete.  
\end{proof}

To estimate $\|\E[\Gamma_N(q)] - \Gamma(q)\|$, we proceed as in the end of Section \ref{aprioriest}. Let
\[\epsilon_N:= \E[\Gamma_N(q)] + (q + \Sigma(\E[\Gamma_N(q)]))^{-1}] \]
and \[\tilde{q}:=\tilde{q}_N = q + \Sigma(\epsilon_N).\] 
Then the arguments of Lemma \ref{trieste} can be repeated to prove \eqref{expgammaminusgamma}.

%

\subsection{No small singular values}

Following Bai and Silverstein in \cite{BS}, we construct a uniformly bounded, equicontinuous family of functions, and then use weak convergence of $\nu_{X_N-zI}$ to show a.s. there are no singular values outside the support of the limiting distribution. 

The arguments of this section can be repeated to show for any fixed $k$, and $t_N = \sqrt{k} N^{-\beta/4}$, one has a.s.
\begin{equation}
\sup_{E \in [0,c]} | a_N(q) - a(q) | = o( (Nt_N)^{-1}).
\end{equation}
Taking the imaginary part of the above equation for any $k = 1,\ldots, 2/\beta $ leads to
\begin{equation}
\sup_{k = 1,\ldots, 2/\beta ; E \in [0,c]} \left |  \int \frac{ d( \nu_{X_N-zI} - \nu_z )(x) }{ (E-x)^2 + k t_N^2  }  \right| = o( N^{-1}t_N^{-2}).
\end{equation}
Taking differences for different $k_1$ and $k_2$ gives
\begin{equation}
\sup_{k_1 \not = k_2 ; E \in [0,c]} \left |  \int \frac{ t_N^2 d( \nu_{X_N-zI} - \nu_z )(x) }{ ((E-x)^2 + k_1 t_N^2 )((E-x)^2 + k_2 t_N^2 ) }  \right| =  o( N^{-1}t_N^{-2}).
\end{equation}

Repeating this for all values of $k$ and splitting the integral over two regions leads to:
\begin{align*}
\sup_{E\in [0,c]}  \left| \int \frac{I_{[0,c+\epsilon]^{c}} d(\nu_{X_N-z} - \nu_z)(\lambda) }{\prod_{k=1}^{2/ \beta} (E- \lambda)^2 + k t_N^2 }+\sum_{\lambda_j \in [0,c+\epsilon]} \frac{ t_N^{4 /\beta} }{\prod_{k=1}^{2/ \beta} (E- \lambda_j)^2 + k t_N^2 } \right| = o(1).
\end{align*}
The first integrand forms a uniformly bounded, equicontinuous family so the integral converges to zero by the weak convergence of $\nu_{X_N-z}$ to  $\nu_z$. The summand is uniformly bounded away from zero when evaluated at a singular value in the interval $[0,c]$. So we conclude that almost surely there are no singular values in the interval $[0,c]$.




\section{Isotropic limit law} \label{sec:mainuv}

This section is devoted to Theorem \ref{thm:mainuv:trunc}.  We divide the proof of Theorem \ref{thm:mainuv:trunc} into the following three steps.
\begin{enumerate}
\item Showing that the diagonal entries of ${G}_N(z)$ convergence uniformly to $m(z)$.
\item Establishing a rate of convergence of the off-diagonal entries of ${G}_N(z)$ to zero.  
\item Establishing a concentration inequality for $u_N^\ast {G}_N(z) v_N$.  
\end{enumerate}

\subsection{Diagonal entries}

Define the event
$$ \Omega_N := \left\{ \| X_N\| \leq 4.5 \right\}. $$
By Lemma \ref{lemma:truncatednorm}, it follows that $\Omega_N$ holds with overwhelming probability.  We establish the following convergence result for the diagonal entries of ${G}_N(z)$.

\begin{lemma}[Diagonal entries] \label{lemma:diagonal}
Let $\eps > 0$.  Then, for $N$ sufficiently large, 
$$ \sup_{5 \leq |z| \leq 6} \sup_{1 \leq i \leq N} \E\left| ({G}_N(z))_{ii} - m(z) \right| \oindicator{\Omega_N} \leq \eps. $$
\end{lemma}

\begin{proof}
We introduce the following notation.  For any $1 \leq i \leq N$, we let ${Y}_N^{(i)}$ be the $(N-1) \times (N-1)$ matrix formed from ${Y}_N$ by removing the $i$-th row and $i$-th column.  We let ${r}_i$ denote the $i$-th row of ${Y}_N$ with the $i$-th entry removed; let ${c}_i$ denote the $i$-th column of ${Y}_N$ with the $i$-th entry removed.  We define
$$ {G}_N^{(i)}(z) := \left( \frac{1}{\sqrt{N}} {Y}_N^{(i)} - z I \right)^{-1} $$
and ${m}_N^{(i)}(z) := \frac{1}{N} \tr {G}_N^{(i)}(z)$.  We let $\check{Y}^{(i)}_N$ denote the $N \times N$ matrix formed from ${Y}_N$ by setting the entries in the $i$-row and $i$-th column to zero.  We define 
$$ \check{G}_N^{(i)}(z) := \left( \frac{1}{\sqrt{N}} \check{Y}_N^{(i)} - z I \right)^{-1} $$
and $\check{m}_N^{(i)}(z) := \frac{1}{N} \tr \check{G}_N^{(i)}(z)$.  

Since ${Y}_N^{(i)}$ and $\check{Y}_N^{(i)}$ are formed from ${Y}_N$, it follows that
$$ \sup_{1 \leq i \leq N} \left\| \frac{1}{\sqrt{N}} {Y}_N^{(i)} \right\| \leq 4.5, \quad \sup_{1 \leq i \leq N} \left\| \frac{1}{\sqrt{N}} \check{Y}_N^{(i)} \right\| \leq 4.5 $$
on the event $\Omega_N$.  By Lemma \ref{lemma:detresolventbnd}, we have
\begin{equation} \label{eq:supGN}
	\sup_{|z| \geq 5} \| {G}_N(z) \| \leq 2, \quad \sup_{|z| \geq 5} \sup_{1 \leq i \leq N} \| {G}_N^{(i)}(z) \| \leq 2, \quad \sup_{|z| \geq 5} \sup_{1 \leq i \leq N} \| \check{G}_N^{(i)}(z) \| \leq 2 
\end{equation}
on the event $\Omega_N$.  It follows that
\begin{equation} \label{eq:supmN}
	\sup_{|z| \geq 5} | {m}_N(z) | \leq 2, \quad \sup_{|z| \geq 5} \sup_{1 \leq i \leq N} | {m}_N^{(i)}(z) | \leq 2, \quad \sup_{|z| \geq 5} \sup_{1 \leq i \leq N} | \check{m}_N^{(i)}(z) | \leq 2
\end{equation}
and 
\begin{equation} \label{eq:infmN}
	\inf_{|z| \geq 5} | z + \rho {m}_N(z) | \geq 3 
\end{equation}
on the event $\Omega_N$.  

Fix $1 \leq i \leq N$ and $z \in \C$ with $|z| \geq 5$.  Let $\eps > 0$.  By the Schur complement (since the diagonal entries of ${Y}_N$ are zero), we have that
\begin{equation} \label{eq:GNzii}
	({G}_N(z))_{ii} = - \frac{1}{z + \rho {m}_N(z) + \eps_N(z)}, 
\end{equation}
where 
$$ \eps_N(z) := {\rho} {m}_N^{(i)}(z) - \rho {m}_N(z) + \frac{1}{N} {r}_i {G}_N^{(i)}(z) {c}_i - {\rho} {m}_N^{(i)}(z). $$
We observe that $({r}_i^\mathrm{T}, {c}_i)$ and ${G}_N^{(i)}$ are independent.  Thus, by conditioning on the event $\{ \| {G}_N^{(i)}(z) \| \leq 2 \}$ (which holds with overwhelming probability), we apply Lemma \ref{lemma:quadp} and obtain
$$ \left| \frac{1}{N} {r}_i {G}_N^{(i)}(z) {c}_i -  {\rho} {m}_N^{(i)}(z) \right| \leq \eps $$
with overwhelming probability.  Since the eigenvalues of $\check{Y}_N^{(i)}$ are the eigenvalues of ${Y}_N^{(i)}$ with an additional eigenvalue of zero, it follows that
$$ | {m}_N^{(i)}(z) - \check{m}_N^{(i)}(z)| \leq \frac{1}{5N} $$
on $\Omega_N$ because $|z| \geq 5$.  Observe that $\check{Y}_N^{(i)} - {Y}_N$ is at most rank $2$.  By the resolvent identity, we have
\begin{align*}
	\left| \check{m}_N^{(i)}(z) - {m}_N(z) \right| &= \frac{1}{N} \left| \tr \left[ {G}_N(z) \frac{1}{\sqrt{N}} \left( {Y}_N - \check{Y}_N^{(i)} \right) \check{G}_N^{(i)}(z) \right] \right| \\
		&\leq \frac{2}{N} \| {G}_N(z) \| \frac{1}{\sqrt{N}} \left( \| {Y}_N \| + \| \check{Y}_N^{(i)} \| \right) \| \check{G}_N^{(i)}(z) \| \\
		& \leq \frac{C}{N}
\end{align*}
on the event $\Omega_N$.  We conclude that, for $N$ sufficiently large, 
$$ \eps_N(z) = O(\eps) $$
with overwhelming probability.  By \eqref{eq:infmN} and \eqref{eq:GNzii}, it follows that
\begin{equation} \label{eq:GNfixz}
	({G}_N(z))_{ii} = - \frac{1}{z + \rho {m}_N(z)} + O(\eps) 
\end{equation}
with overwhelming probability.  

Let $D$ be a compact, connected set that satisfies
$$ \{z \in \C : 5 \leq |z| \leq 6 \} \subset D \subset \{z \in \mathbb{C} : |z| \geq 5\}. $$ 
If $\rho \neq 0$, we additionally assume that there exists $z_0 \in D$ with 
\begin{equation} \label{eq:mz0100}
	|m(z_0)| \leq \frac{\delta}{100 |\rho|}, \quad |z_0| \geq 5 + \frac{100|\rho|}{\delta}.
\end{equation}
Such a choice of $z_0$ in \eqref{eq:mz0100} always exists by \eqref{eq:lim}.  

We now extend \eqref{eq:GNfixz} to all $z \in D$.  Let $\mathcal{N}$ be an $\eps$-net of $D$.  By Lemma \ref{lemma:epsnet}, $|\mathcal{N}| = O(1)$.  Thus, by the union bound, we have
\begin{equation} \label{eq:GNepsnet}
	\sup_{z \in \mathcal{N}} \left| ({G}_N(z))_{ii} + \frac{1}{z + \rho {m}_N(z)} \right| \leq C \eps, 
\end{equation}
with overwhelming probability.  Fix a realization in the event $\Omega_N$ such that \eqref{eq:GNepsnet} holds.  By \eqref{eq:generalresolventid} and \eqref{eq:supGN}, 
$$ \| {G}_N(z) - {G}_N(z') \| \leq 4|z-z'|, \quad \left| {m}_N(z) - {m}_N(z') \right| \leq 4 |z - z'| $$
for all $z,z' \in D$.  Let $z' \in D$.  Then there exists $z \in \mathcal{N}$ with $|z-z'| \leq \eps$.  Thus, we have
\begin{align*}
	\left| (G_N(z'))_{ii} + \frac{1}{z' + \rho {m}_N(z')} \right| &\leq (4+C) \eps + \left| \frac{1}{z' + \rho {m}_N(z')} - \frac{1}{z + \rho {m}_N(z)} \right| \\
		&\leq (5 + C) \eps
\end{align*}
by \eqref{eq:infmN}.  Therefore, we conclude that
$$ \sup_{z \in D} \left| (G_N(z))_{ii} + \frac{1}{z + \rho {m}_N(z)} \right| \leq C \eps $$
with overwhelming probability.  

By the union bound, we have
\begin{equation} \label{eq:supNsupzGN}
	\sup_{1 \leq i \leq N} \sup_{z \in D} \left| (G_N(z))_{ii} + \frac{1}{z + \rho {m}_N(z)} \right| \leq C \eps 
\end{equation}
with overwhelming probability.  Thus, with overwhelming probability, 
\begin{equation} \label{eq:supNsupzmN}
	\sup_{z \in D} \left| {m}_N(z) + \frac{1}{z + \rho {m}_N(z)} \right| \leq C \eps. 
\end{equation}

If $\rho = 0$, we conclude that
\begin{equation} \label{eq:rho0eps}
	\sup_{z \in D} \left| {m}_N(z) - m(z) \right| \leq C \eps 
\end{equation}
with overwhelming probability.  We now obtain this bound in the case that $\rho \neq 0$ by applying Lemma \ref{lemma:dioch}.  In view of \eqref{eq:mz0100} and Lemma \ref{lemma:detresolventbnd}, we have
$$ \left| {m}_N(z_0) - m(z_0) \right| \leq \left| {m}_N(z_0) \right| + |m(z_0)| \leq \frac{\delta}{50|\rho|} $$
on the event $\Omega_N$.  Thus, by Lemma \ref{lemma:dioch}, we conclude that \eqref{eq:rho0eps} holds with overwhelming probability for any $-1 \leq \rho \leq 1$.  

By \eqref{eq:supNsupzGN}, \eqref{eq:supNsupzmN}, and \eqref{eq:rho0eps}, we obtain
$$ \sup_{1 \leq i \leq N} \sup_{z \in D} \left| (G_N(z))_{ii} - m(z) \right| \leq C \eps $$
with overwhelming probability.  Since $\Omega_N$ holds with overwhelming probability, we have
$$ \sup_{1 \leq i \leq N} \sup_{z \in D} \left| (G_N(z))_{ii} - m(z) \right| \oindicator{\Omega_N} \leq C \eps $$
with overwhelming probability.  By Lemma \ref{lemma:mbnd} and \eqref{eq:supGN}, we conclude that, for $N$ sufficiently large,
$$ \sup_{1 \leq i \leq N} \sup_{z \in D} \E \left| (G_N(z))_{ii} - m(z) \right| \oindicator{\Omega_N} \leq C \eps. $$
The proof of the lemma is complete.  
\end{proof}

\subsection{Off-diagonal entries} \label{Off-diagonal entries}


Let $H_N$ be the Hermitization of $X_N$ as in Section \ref{sec:least}. Once again we will view $H_N$ as a $N \times N$ matrix of $2 \times 2$ blocks. We reuse the notation from Section \ref{sec:least} with $q = q(z,\eta)$. Let $R_N^{21}(\eta,z)$ be the $N \times N$ matrix with $( R_N^{21}(\eta,z))_{ij} = R_{ij}^{21}(\eta,z)$. We begin by noting that when defined, $u_N^* G_N(z) v_N = u_N^*(R_N^{21}(0,z))v_N $. Just as in Section \ref{sec:least} when we only needed to control $R_{ii}^{11}$ but found it easier to instead control the block $R_{ii}$, here we will estimate the $2 \times 2$ block $R_{ij}$ for $i \not = j$ in order to control $R_{ij}^{21}$. We should note that many of our estimates will involve the norm $\|R_N(z,\eta)\|$, but on the event that there are no eigenvalues outside the ellipse this norm is $O(1)$.

\begin{lemma}[Off-diagonal entries] \label{lemma:offdiag}
Fix $z,\eta \in \mathbb{C}$ with $5 \leq |z| \leq 6$ and $\Im(\eta) > 0$.  Then
$$ \sup_{i \neq j }  \|  \E[ (R_N)_{ij} ] \| = O(N^{-3/2}). $$
\end{lemma}
\begin{proof}


We begin with Schur's complement, \eqref{Schur1}, with $A$ being the upper $1$ by $1$ block, $D$ being the lower $N-1$ by $N-1$ block, and $B$ and $C$ being the corresponding off-diagonal blocks. Then for $i \not =1$
\begin{equation}
\label{offdiagexp1}
 R_{1 i} = -R_{11}  (H_{1 }^{(1)*} R^{(1)})_{i} 
\end{equation}
and
\begin{equation}
\label{offdiagexp2}
 R_{ i 1} = -  ( R^{(1)} H_{ 1}^{(1)} )_i R_{11}.
 \end{equation}
Other elements of $R_N$ can be computed by permuting the rows and columns of $H_N$ before applying Schur's complement.




Combining the identities (generalized to an arbitrary element) for $i \not = j$ leads to 
\begin{equation}
\label{offdiagexp}
 R_{ij} = R_{ii}^{(j)} ( H_{ij} + H_{i}^{(i,j)*} R^{(i,j)} H_{j}^{(i,j)}) R_{jj} .
\end{equation}

Let 
\[ \xi_{N}^{(i,j)} := H_{ij} + H_{i}^{(i,j)*} R^{(i,j)} H_{j}^{(i,j)}. \]
Additionally, recall the diagonal entries of the resolvent are
\begin{align*}
R_{ii} = (- q - H_{i}^{(i)*} R^{(i)} H_{i}^{(i)} )^{-1},
\end{align*}
the definition 
\begin{align*}
\widehat R_{ii}= (q - \E[\Sigma(\Gamma^{(i)})])^{-1}
\end{align*}
and  
\[ R_{ii} - \widehat{R}_{ii} = \widehat R_{ii} ( \gamma^{(i)}_N )  R_{ii}. \]

We begin with \eqref{offdiagexp1} and then apply \eqref{Rtohat} two times and finally \eqref{offdiagexp2} to obtain
\begin{align*}
\E[ R_{ij} ] =& \sum_{l \not = i} \E[ - R_{ii} ( H_{il} R_{lj}^{(i)} ) ]\\
=&  \sum_{l \not = i} \E[ - \widehat R_{ii} ( H_{il} R_{lj}^{(i)} ) - \widehat R_{ii} ( \gamma^{(i)}_N ) \widehat R_{ii}( H_{il} R_{lj}^{(i)} )-   \widehat R_{ii} ( \gamma^{(i)}_N ) \widehat R_{ii} ( \gamma^{(i)}_N ) R_{ii}( H_{il} R_{lj}^{(i)} )  ] \\
=&  \E[  \sum_{l \not = i} \left( - \widehat R_{ii} ( H_{il} R_{lj}^{(i)} ) - \widehat R_{ii} ( \gamma^{(i)}_N ) \widehat R_{ii}( H_{il} R_{lj}^{(i)} ) \right) \\
&-   \widehat R_{ii} ( \gamma^{(i)}_N ) \widehat R_{ii} ( \gamma^{(i)}_N )  R_{ii} ( H_{ij} - \sum_{l,k\not = i,j}  H_{il} R^{(i,j)}_{lk} H_{kj}) R_{jj}^{(i)}  ].
\end{align*}
The first term is zero because $\E[ H_{il}]= 0$. We estimate the other terms as in Section \ref{sec:least} and bound each entry of the $2 \times 2$ blocks. Each entry is a sum of products of entries from the blocks.  Thus, by the triangle inequality, it suffices to bound arbitrary products of each block's entries. As before lower case superscripts from the beginning of the alphabet are all either $1$ or $2$.

To bound the third term we apply H\"older's inequality and directly compute the moments:
\begin{align}
\label{3rdoffdiag}
&\E[ \widehat R_{ii}^{ab} ( \gamma^{(i)}_N )^{bc} \widehat R_{ii}^{cd} ( \gamma^{(i)}_N )^{de}  R_{ii}^{ef} ( \xi_N^{(i,j)})^{fg} R_{jj}^{(i) gh} ] \notag \\
& \leq K  \E[ |(\gamma^{(i)}_N )^{bc}|^4]^{1/4} \E[|(\gamma^{(i)}_N )^{de}|^4]^{1/4} \E[|( \xi_N^{(i,j)})^{fg}|^2]^{1/2} \notag \\
& =O(N^{-3/2}).
\end{align}

We begin estimating the second term by averaging over the $i^{th}$ row and column of $H$. Let $\mu_3 = \max\{\E[|H_{12}|^3], \E[|H_{21}|^{3}]\}$.  Then
\begin{align}
\label{LA}
&|\sum_{l,m,n \not = i}\E[  \widehat R_{ii}^{ab} (H_{im} R^{(i)}_{mn} H_{ni} - \E[ \tr_N( R_N^{(i)} )]  )^{bc} \widehat R_{ii}^{cd} ( H_{il} R_{lj}^{(i)} )^{de} ]|  \notag \\
&\leq \sum_{l \not = i}\frac{ \mu_3}{N^{3/2}} |\E[  \widehat R_{ii}^{ab} (R^{(i)}_{ll})^{bc} \widehat R_{ii}^{cd} ( R_{lj}^{(i)} )^{de} ]| 
\end{align}

We now apply the Cauchy-Schwarz inequality with \eqref{crhR} to get a weaker bound than desired. Once the weaker bound is proven, we will return to \eqref{LA} and prove the desired bound.
\begin{align*}
& \sum_{l \not = i}\frac{ \mu_3}{N^{3/2}} |\E[  \widehat R_{ii}^{ab} (R^{(i)}_{ll})^{bc} \widehat R_{ii}^{cd} ( R_{lj}^{(i)} )^{de} ]| \\
&\leq \frac{ K}{N^{3/2}} (\E[ \sum_{l \not = i} |R_{ll}^{(i)}|^2 ])^{1/2} (\E[ \sum_{l \not = i} |R_{lj}^{(i)}|^2 ])^{1/2}\\
&\leq \frac{ K}{N^{3/2}} (\E[ \sum_{l \not = i} |R_{ll}^{(i)}|^2 ])^{1/2} \| R_N^{(i)} \|\\
&=O(N^{-1}),
\end{align*}
which combed with \eqref{3rdoffdiag} implies $\| \E[R_{ij}]\| = O(N^{-1})$. Returning to \eqref{LA}, applying \eqref{Rtohat} and \eqref{offdiagexp} leads to:
\begin{align*}
& \sum_{l \not = i}\frac{ \mu_3}{N^{3/2}} |\E[  \widehat R_{ii}^{ab} (\widehat R^{(i)}_{ll})^{bc} \widehat R_{ii}^{cd} ( R_{lj}^{(i)} )^{de} ]| \\
&+ \sum_{l \not = i}\frac{ \mu_3}{N^{3/2}} |\E[  \widehat R_{ii}^{ab} (\widehat R^{(i)}_{ll} \gamma_N^{(l;i)} R^{(i)}_{ll} )^{bc} \widehat R_{ii}^{cd} (R_{ll}^{(i,j)} \xi_N^{(l,j;i)} R_{jj}^{(i)})^{de} ]| \\
&= O(N^{-3/2}),
\end{align*}
where $\xi_{N}^{(l,j;i)} := H_{lj} + H_{l}^{(l,j,i)*} R^{(l,j,i)} H_{j}^{(l,j,i)} $ and $\gamma_N^{(l;i)} = H_l^{(l,i)*} R_N^{(l,i)} H_i^{(l,i)} - \E[\Sigma(\Gamma_N^{(l,i)})]$.
The first term uses the just verified $O(N^{-1})$ bound and the second uses the Cauchy-Schwarz inequality and a direct computation.

\end{proof}

\subsection{Concentration of $u_N^\ast G_N v_N$}

We now establish the following concentration result.  

\begin{lemma}[Concentration of bilinear forms] \label{lemma:concentration:GN}
Let $\eps > 0$.  Fix $z \in \mathbb{C}$ with $5 \leq |z| \leq 6$.  Then a.s., for $N$ sufficiently large, 
$$ \left| u_N^\ast {G}_N(z) \oindicator{\Omega_N} v_N - \E u_N^\ast {G}_N(z)\oindicator{\Omega_N} v_N \right| \leq \eps. $$
\end{lemma}
\begin{proof}
The proof below is based on the arguments of Bai and Pan \cite{BP}.  Let $\eps > 0$ and fix $5 \leq |z| \leq 6$.  We will drop the dependence on $z$ and simply write ${G}_N$ to denote the matrix ${G}_N(z)$.  We introduce the following notation.  Let ${X}_N^{(k)}$ be the matrix obtained from ${X}_N$ by replacing all elements in the $k$-th column and $k$-th row with zero.  Define ${G}_N^{(k)} := ({X}_N^{(k)} - z I)^{-1}$.  Let ${r}_k$ be the $k$-th row of ${X}_N$; let ${c}_k$ be the $k$-th column of ${X}_N$.  Let $\E_k$ denote the conditional expectation given ${r}_{k+1}, \ldots,{r}_N, {c}_{k+1},\ldots,{c}_N$.  Let $e_1, \ldots, e_N$ denote the standard basis of $\C^N$.  Let $u_N = (u_{N,i})_{i=1}^N$ and $v_N = (v_{N,i})_{i=1}^N$.  

We will take advantage of the fact that all the elements of the $k$-th column and $k$-th row of ${G}_N^{(k)}$ are zero except that the $(k,k)$-th element is $-1/z$.  Thus, 
\begin{align}
	e_k^\ast {G}_N^{(k)} e_k = - \frac{1}{z}, \quad e_k^\ast {G}_N^{(k)} v_N &= -\frac{v_{N,k}}{z}, \quad u_N^\ast {G}_N^{(k)} e_k = -\frac{\bar{u}_{N,k}}{z},  \label{eq:GNkdiag} \\
	e_k^\ast {G}_N^{(k)} {c}_k = 0, &\quad  {r}_k {G}_N^{(k)} e_k = 0. \label{eq:GNKzero}
\end{align}

It follows from the definitions above that
$$ {X}_N = {X}_N^{(k)} + {c}_k e_k^\ast + e_k {r}_k. $$
We define
$$ {X}_N^{(k1)} := {X}_N^{(k)} + e_k {r}_k, \quad {X}_N^{(k2)} := {X}_N^{(k)} + {c}_k e_k^\ast, $$
and
$$ {G}_N^{(kj)} := \left( {X}_N^{(kj)} - z I \right)^{-1}, \quad j=1,2. $$

Define the events
$$ \Omega_N^{(k)} := \left\{ \| {X}_N^{(k)} \| \leq 4.5 \right\}. $$
We let $\oindicator{\Omega_N^{(k)}}$ denote the indicator function of the event $\Omega_N^{(k)}$.  Since $\Omega_N \subset \Omega_N^{(k)}$, it follows that $\oindicator{\Omega_N} \leq \oindicator{\Omega_N^{(k)}}$.  By Lemma \ref{lemma:detresolventbnd}, we have
$$ \|{G}_N \| \leq 2, \quad \sup_{1 \leq k \leq N} \| G_N^{(k)} \| \leq 2, \quad \sup_{1 \leq k \leq N} \| G_N^{(kj)} \| \leq 2, \quad j=1,2 $$
on the event $\Omega_N$.  Moreover, $\|G_N^{(k)} \| \leq 2$ on the event $\Omega_N^{(k)}$.  

Set
\begin{align*}
	\alpha_{N}^{(k)} &:= \frac{1}{1 + z^{-1} {r}_k {G}_N^{(k)} \oindicator{\Omega_N} {c}_k }, \\
	\gamma_N^{(k)} &:= \frac{1}{ 1 + z^{-1} \frac{{\rho}}{N} \left( \tr G_N^{(k)} + 1/z \right) \oindicator{\Omega_N}}, \\
	\xi_N^{(k)} &:= {r}_k {G}_N^{(k)} {c}_k - \frac{{\rho}}{N} \left( \tr {G}_N^{(k)} + \frac{1}{z} \right), 
\end{align*}
and
$$ \eta_N^{(k)} := {r}_k {G}_N^{(k)} v_N u_N^\ast {G}_N^{(k)} {c}_k - \frac{{\rho}}{N} \left( u_N^\ast {G_N^{(k)}}^2 v_N - z^{-2} u_{N,k} v_{N,k} \right). $$

We now collect a variety of preliminary calculations and bounds we will need to complete the proof.  

\begin{enumerate}[(i)]
\item \label{item:detbnd}
By \eqref{eq:rank1} and \eqref{eq:GNKzero}, we have
\begin{align*}
	e_k^\ast {G}_N^{(k1)} {c}_k = e_k^\ast {G}_N^{(k)} {c}_k - \frac{ e_k^\ast {G}_N^{(k)} e_k {r}_k {G}_N^{(k)} {c}_k }{1 + {r}_k {G}_N^{(k)} e_k } = z^{-1} {r}_k G_N^{(k)} {c}_k,
\end{align*}
and hence
$$ \frac{1}{1 + e_k^\ast {G}_N^{(k1)}\oindicator{\Omega_N} {c}_k } = \alpha_N^{(k)}. $$
Similarly, we obtain
$$\frac{1}{1 + {r}_k {G}_N^{(k2)} \oindicator{\Omega_N} e_k } = \alpha_N^{(k)}. $$
By the Schur complement, we have that
$$ ({G}_N)_{kk} = - \frac{1}{z + {r}_k {G}_N^{(k)} {c}_k}. $$
Thus, on the event $\Omega_N$, we have
$$ |\alpha_N^{(k)}| = \left| \frac{z}{z + {r}_k {G}_N^{(k)} {c}_k} \right| \leq 6 \| G_N^{(k)} \| \leq 12. $$
On the event $\Omega_N^\stcomp$, $\alpha_N^{(k)} = 1$.  Therefore, we conclude that a.s., 
$$ |\alpha_N^{(k)}| \leq 12. $$
Similarly, we have a.s., 
$$ |\gamma_N^{(k)}| \leq \frac{|z|}{|z| - \| G_N^{(k)}\| \oindicator{\Omega_N} - \frac{1}{|z| N}} \leq 3. $$

\item \label{item:diffungnvn}
By the Burkholder inequality (Lemma \ref{lemma:burkholder}), for any $p > 2$, we have 
\begin{align*}
	\E& \left| \sum_{k=1}^N ( \E_{k-1} - \E_k) u_N^\ast {G}_N^{(k)} v_N \left( \oindicator{\Omega_N^{(k)}} - \oindicator{\Omega_N} \right) \right|^{p} \\
		&\qquad \leq C_p \E \left( \sum_{k=1}^N |u_N^\ast {G}_N^{(k)} v_N|^2 \oindicator{\Omega_N^{(k)} \cap \Omega_N^\stcomp} \right)^{p/2} \\
		&\qquad \leq C_p N^{p/2} \Prob( \Omega_N^\stcomp).
\end{align*}

\item \label{item:k1k2calc}
By \eqref{eq:rank1vec} and \eqref{item:detbnd}, we have
\begin{align*}
	u_N^\ast {G}_N {c}_k \oindicator{\Omega_N} &= \frac{u_N^\ast {G}_N^{(k1)} {c}_k \oindicator{\Omega_N}}{1 + e_k^\ast {G}_N^{(k1)}{c}_k } =  \frac{u_N^\ast {G}_N^{(k1)} {c}_k \oindicator{\Omega_N}}{1 + e_k^\ast {G}_N^{(k1)}{c}_k \oindicator{\Omega_N} } = u_N^\ast {G}_N^{(k1)} {c}_k \oindicator{\Omega_N} \alpha_N^{(k)}.  
\end{align*}
Similarly, 
$$ u_N^\ast {G}_N e_k \oindicator{\Omega_N} = u_N^\ast {G}_N^{(k2)} e_k \oindicator{\Omega_N} \alpha_N^{(k)}. $$

\item \label{item:diffk1k2}
By \eqref{eq:rank1} and \eqref{eq:GNKzero}, we have
$$ {G}_N^{(k1)} = {G}_N^{(k)} - {G}_N^{(k)} e_k {r}_k {G}_N^{(k)} $$
and
$$ {G}_N^{(k2)} = {G}_N^{(k)} - {G}_N^{(k)} {c}_k e_k^\ast {G}_N^{(k)}. $$

\item \label{item:alphaeq}
By definition of $\alpha_N^{(k)}$, we have
$$ z^{-1} (\E_{k-1} - \E_{k}) {r}_k {G}_N^{(k)} {c}_k \alpha_{N}^{(k)} = - (\E_{k-1} - \E_{k}) (\alpha_N^{(k)}). $$

\item \label{item:alpharel}
By definition of $\alpha_N^{(k)}, \gamma_N^{(k)}, \xi_N^{(k)}$, we have
\begin{align*}
	\alpha_N^{(k)} - \gamma_N^{(k)} = - z^{-1} \alpha_N^{(k)} \gamma_N^{(k)} \xi_N^{(k)} \oindicator{\Omega_N}.
\end{align*}

\item \label{item:Np2}
We note that the entries of ${r}_k$ and ${c}_k$ have mean zero, variance $1/N$, and are a.s. bounded by $4L/\sqrt{N}$.  Moreover, $({r}_k^\mathrm{T}, {c}_k)$ and ${G}_N^{(k)} \oindicator{\Omega_N^{(k)}}$ are independent.  Thus, by Lemma \ref{lemma:lde} in Appendix \ref{section:lde}, for any $p \geq 2$, we have
$$ \sup_{1 \leq k \leq N} \E_k\left| \xi_N^{(k)} \right|^p \oindicator{\Omega_N} \leq \sup_{1 \leq k \leq N} \E_k \left| \xi_N^{(k)} \right|^{p}\oindicator{\Omega_N^{(k)}} = O_{L,p}(N^{-p/2}). $$
Similarly, 
$$ \sup_{1 \leq k \leq N} \E_k \left| \eta_N^{(k)}\right|^p \oindicator{\Omega_N} \leq \sup_{1 \leq k \leq N} \E_k \left| \eta_N^{(k)} \right|^{p}\oindicator{\Omega_N^{(k)}} = O_{L,p}(N^{-p/2}). $$

\item \label{item:comparek}
By the bounds in \eqref{item:detbnd}, we have
$$ \sup_{1 \leq k \leq N} \left| \gamma_N^{(k)} \oindicator{\Omega_N} - \frac{1}{1 + z^{-1} \frac{{\rho}}{N} \left( \tr {G}_N^{(k)} + \frac{1}{z} \right) \oindicator{\Omega_N^{(k)}}} \oindicator{\Omega_N^{(k)}} \right| \leq C \oindicator{\Omega_N^\stcomp}. $$
Thus, by the Burkholder inequality, for any $p \geq 2$, we have
\begin{align*}
	\E &\left| \sum_{k=1}^N (\E_{k-1} - \E_k ) \overline{u}_{N,k} v_{N,k} \left(  \gamma_N^{(k)} \oindicator{\Omega_N} - \frac{1}{1 + z^{-1} \frac{{\rho}}{N} \left( \tr {G}_N^{(k)} + \frac{1}{z} \right) \oindicator{\Omega_N^{(k)}}} \oindicator{\Omega_N^{(k)}} \right) \right|^p \\
	&\qquad \leq C_p \Prob( \Omega_N^\stcomp) \left( \sum_{k=1}^N |u_{N,k}|^2 |v_{N,k}|^2 \right)^{p/2} \\
	&\qquad \leq C_p \Prob(\Omega_N^\stcomp). 
\end{align*}

\end{enumerate}

We now complete the proof of the lemma.  Indeed, it suffices to show that, for any $p > 2$, 
$$ \E \left| u_N^\ast {G}_N v_N \oindicator{\Omega_N} - \E u_N^\ast {G}_N v_N \oindicator{\Omega_N} \right|^p = O_{L,p}(N^{-p/2}). $$
We begin by decomposing $u_N^\ast {G}_N v_N \oindicator{\Omega_N} - \E u_N^\ast {G}_N v_N$ as a martingale difference sequence.  Since 
$$ \E_{k} u_N^\ast {G}_N^{(k)} v_N \oindicator{\Omega_N^{(k)}} = \E_{k-1} u_N^\ast {G}_N^{(k)} v_N \oindicator{\Omega_N^{(k)}}, $$
we have
\begin{align*}
	u_N^\ast {G}_N v_N \oindicator{\Omega_N} - \E u_N^\ast {G}_N v_N \oindicator{\Omega_N} &= \sum_{k=1}^N ( \E_{k-1} - \E_{k} ) u_N^\ast {G}_N v_N \oindicator{\Omega_N} \\
		&= \sum_{k=1}^N (\E_{k-1} - \E_k) u_N^\ast \left( {G}_N \oindicator{\Omega_N} - {G}_N^{(k)} \oindicator{\Omega_N^{(k)}} \right) v_N.
\end{align*}
In view of \eqref{item:diffungnvn} and the fact that $\Omega_N$ holds with overwhelming probability, it suffices to show that, for any $p > 2$, 
$$ \E \left| \phi_N \right|^{p} = O_{L,p}(N^{-p/2}), $$
where
$$ \phi_N := \sum_{k=1}^N (\E_{k-1} - \E_k) u_N^\ast \left( {G}_N - {G}_N^{(k)} \right) v_N \oindicator{\Omega_N}. $$

By the resolvent identity, we have
\begin{align*}
	\phi_N &= - \sum_{k=1}^N ( \E_{k-1} - \E_k) u_N^\ast {G}_N ( {c}_k e_k^\ast + e_k {r}_k ){G}_N^{(k)} v_N \oindicator{\Omega_N} \\
		&=: - (\phi_{N1} + \phi_{N2}).
\end{align*}
By \eqref{item:k1k2calc}, \eqref{item:diffk1k2}, \eqref{item:alphaeq}, and \eqref{eq:GNkdiag}, we decompose
\begin{align*}
	\phi_{N1} &:= \sum_{k=1}^N (\E_{k-1} - \E_k) u_N^\ast {G}_N {c}_k e_k^\ast {G}_N^{(k)} v_N \oindicator{\Omega_N} \\
		&= -\sum_{k=1}^N (\E_{k-1} - \E_k) z^{-1} u_N^\ast {G}_N^{(k1)} {c}_k v_{N,k} \alpha_N^{(k)} \oindicator{\Omega_N} \\
		&= - \sum_{k=1}^N(\E_{k-1} - \E_k) z^{-1} v_{N,k} \alpha_{N}^{(k)} \left( u_N^\ast {G}_N^{(k)} {c}_k + z^{-1} \overline{u}_{N,k} {r}_k {G}_N^{(k)} {c}_k \right) \oindicator{\Omega_N} \\
		&= - \sum_{k=1}^N (\E_{k-1} - \E_k) z^{-1} \left( u_N^\ast {G}_N^{(k)} {c}_k v_{N,k} - \overline{u}_{N,k} v_{N,k} \right) \alpha_{N,k} \oindicator{\Omega_N} \\
		&=: \phi_{N11} + \phi_{N12}. 
\end{align*}
Similarly, by \eqref{item:k1k2calc}, \eqref{item:diffk1k2}, and \eqref{eq:GNkdiag}, we have
\begin{align*}
	\phi_{N2} &:= \sum_{k=1}^N (\E_{k-1} - \E_k) u_N^\ast {G}_N e_k {r}_k {G}_N^{(k)} v_N \oindicator{\Omega_N} \\
		&= \sum_{k=1}^N (\E_{k-1} - \E_k) u_N^\ast {G}_N^{(k2)} e_k {r}_k {G}_N^{(k)} v_N \alpha_N^{(k)} \oindicator{\Omega_N} \\
		&= -\sum_{k=1}^N (\E_{k-1} - \E_k) z^{-1} \left( \bar{u}_{N,k} - u_N^\ast {G}_N^{(k)} {c}_k \right) {r}_k {G}_N^{(k)} v_N \alpha_N^{(k)} \oindicator{\Omega_N} \\
		&=: \phi_{N21} + \phi_{N22}.
\end{align*}

Therefore, in order to complete the proof, it suffices to show that, for any $p > 2$, 
$$ \E| \phi_{N11}|^p + \E | \phi_{N12} |^p + \E| \phi_{N21}|^p + \E | \phi_{N22} |^p = O_{L,p}(N^{-p/2}). $$
We bound each term individually.  

By Rosenthal's inequality (Lemma \ref{lemma:rosenthal}), we have, for any $p > 2$,
\begin{align*}
	\E |\phi_{N11} |^p &= \E \left| \sum_{k=1}^N (\E_{k-1} - \E_k) z^{-1} u_N^\ast {G}_N^{(k)} {c}_k v_{N,k} \alpha_N^{(k)} \oindicator{\Omega_N} \right|^p \\
		&\leq \frac{C_p}{|z|^p} \Bigg[ \E \left( \sum_{k=1}^N \E_k \left| (\E_{k-1} - \E_k)  u_N^\ast {G}_N^{(k)} {c}_k v_{N,k} \alpha_N^{(k)} \oindicator{\Omega_N} \right|^2 \right)^{p/2} \\
		&\qquad\qquad + \sum_{k=1}^N \E \left| (\E_{k-1} - \E_k) u_N^\ast {G}_N^{(k)} {c}_k v_{N,k} \alpha_N^{(k)} \oindicator{\Omega_N} \right|^p \Bigg] \\
		&\leq C_p \Bigg[ \E \left( \sum_{k=1}^N |v_{N,k}|^2 \E_k \left|u_N^\ast {G}_N^{(k)} {c}_k \right|^2 \oindicator{\Omega_N^{(k)}} \right)^{p/2} \\
		&\qquad\qquad + \sum_{k=1}^N |v_{N,k}|^p \E \left| u_N^\ast {G}_N^{(k)} {c}_k \right|^p \oindicator{\Omega_N^{(k)}} \Bigg] \\
		&= O_{L,p}(N^{-p/2}).  
\end{align*}
Here we used Lemma \ref{lemma:nonneg} from Appendix \ref{section:lde} to verify that, for any $p \geq 2$, 
$$ \E_k \left| u_N^\ast {G}_N^{(k)} \oindicator{\Omega_N^{(k)}} {c}_k \right|^{p} = \E_k \left| {c}_k^\ast \left({G}_N^{(k)} \right)^\ast u_N  u_N^\ast {G}_N^{(k)} {c}_k \right|^{p/2} \oindicator{\Omega_N^{(k)}} = O_{L,p}(N^{-p/2}) $$
uniformly for $1 \leq k \leq N$.  

Similarly, by another application of Rosenthal's inequality, one obtains, for any $p > 2$, 
$$ \E | \phi_{N21}|^p = O_{L,p}(N^{-p/2}). $$

By \eqref{item:alpharel}, we have
\begin{align*}
	\phi_{N12} &:= \sum_{k=1}^N (\E_{k-1} - \E_k) z^{-1} \overline{u}_{N,k} v_{N,k} \alpha_N^{(k)} \oindicator{\Omega_N} \\
		&= \sum_{k=1}^N (\E_{k-1} - \E_k) z^{-1} \overline{u}_{N,k} v_{N,k} \left( \gamma_N^{(k)} - z^{-1} \alpha_N^{(k)} \gamma_N^{(k)} \xi_N^{(k)} \right)  \oindicator{\Omega_N} \\
		&=: \phi_{N121} + \phi_{N122}.
\end{align*}

Since 
$$ (\E_{k-1} - \E_k) \frac{1}{1 + z^{-1} \frac{{\rho}}{N} \left( \tr {G}_N^{(k)} + \frac{1}{z} \right) \oindicator{\Omega_N^{(k)}}} \oindicator{\Omega_N^{(k)}} = 0, $$
we apply \eqref{item:comparek} to obtain $\E |\phi_{N121}|^p = O_{L,p}(N^{-p/2})$ for any $p \geq 2$.  By \eqref{item:detbnd}, \eqref{item:Np2}, and Rosenthal's inequality, we have, for any $p > 2$, 
\begin{align*}
	\E |\phi_{N122}|^p &\leq C_p \left( \E \left( \sum_{k=1}^N |u_{N,k}|^2 |v_{N,k}|^2 \E_k \left|\xi_N^{(k)} \right|^2 \right)^{p/2} + \sum_{k=1}^N |u_{N,k}|^p |v_{N,k}|^p \E \left|\xi_{N}^{(k)} \right|^p \right) \\
		&= O_{L,p}(N^{-p/2}).
\end{align*}

By definition of $\eta_N^{(k)}$, we have
\begin{align*}
	\phi_{N22} &:= \sum_{k=1}^N (\E_{k-1} - \E_k) z^{-1} u_N^\ast {G}_N^{(k)} {c}_k {r}_k {G}_N^{(k)} v_N \oindicator{\Omega_N} \alpha_N^{(k)} \\
		&= \sum_{k=1}^N (\E_{k-1} - \E_k) z^{-1} \alpha_N^{(k)} \oindicator{\Omega_N} \left( \eta_N^{(k)} + \frac{{\rho}}{N} \left( u_N^\ast {G}_N^{(k)}v_N - z^{-2} \overline{u}_{N,k} v_{N,k} \right) \right).
\end{align*}
From \eqref{item:detbnd}, we have
$$ \left| z^{-1} \alpha_N^{(k)} \oindicator{\Omega_N} \frac{{\rho}}{N} \left( u_N^\ast {G}_N^{(k)}v_N - z^{-2} \overline{u}_{N,k} v_{N,k}\right) \right|^2 = O(N^{-2}), $$
and thus, by the Burkholder inequality, we have, for any $p \geq 2$, 
$$ \E \left| \sum_{k=1}^N (\E_{k-1} - \E_k)z^{-1} \alpha_N^{(k)} \oindicator{\Omega_N} \frac{{\rho}}{N} \left( u_N^\ast {G}_N^{(k)}v_N - z^{-2} \overline{u}_{N,k} v_{N,k} \right) \right|^{p} = O_{p} (N^{-p/2}). $$
On the other hand, by \eqref{item:Np2} and Rosenthal's inequality, for any $p > 2$, we conclude that
$$ \E \left| \sum_{k=1}^N (\E_{k-1} - \E_k) z^{-1} \alpha_N^{(k)} \oindicator{\Omega_N} \eta_N^{(k)} \right|^p = O_{L,p}(N^{-p/2}). $$
The proof of the lemma is complete.  
\end{proof}

\subsection{Proof of Theorem \ref{thm:mainuv:trunc}}

We are now ready to prove Theorem \ref{thm:mainuv:trunc} using the results of the previous subsections.  

\begin{proof}[Proof of Theorem \ref{thm:mainuv:trunc}]
Let $0 < \eps < 1/2$ and fix $5 \leq |z| \leq 6$.  Let $\eta := \sqrt{-1} t$, where $t := \eps/100$. 
By Lemma \ref{lemma:detresolventbnd} and Lemma \ref{lemma:normbnd}, it follows that 
\begin{equation} \label{eq:gnbnd2}
	\| {G}_N(z) \| \leq 2
\end{equation}
on the event $\Omega_N$.  
Moreover, since the eigenvalues of $R_N(\eta,z)$ are given by 
$$ \frac{1}{\pm \sigma_i({X}_N - zI ) - \eta }, \quad i=1,\ldots,N, $$ 
it follows that
$$ \| R_N(\eta, z) \| \leq 4 $$
on the event $\Omega_N$.  

Since $\Omega_N$ holds with overwhelming probability, it suffices to show that a.s., for $N$ sufficiently large, 
$$ \left| u_N^\ast G_N(z) v_N \oindicator{\Omega_N} - m(z) u_N^\ast v_N \oindicator{\Omega_N} \right| \leq \eps. $$
By the triangle inequality, we have
\begin{align*}
	\left| u_N^\ast G_N(z) v_N \oindicator{\Omega_N} - m(z) u_N^\ast v_N \oindicator{\Omega_N} \right|  &\leq \left | u_N^\ast G_N(z) v_N \oindicator{\Omega_N} - \E u_N^\ast G_N(z) v_N \oindicator{\Omega_N} \right| \\
	&\quad + \left| \E u_N^\ast G_N(z) v_N \oindicator{\Omega_N} - \E u_N^\ast R_N^{21}(\eta,z) v_N \oindicator{\Omega_N} \right| \\
	&\quad + \left| \E u_N^\ast R_N^{21}(\eta,z) v_N \oindicator{\Omega_N} - m(z) u_N^\ast v_N \oindicator{\Omega_N} \right|.
\end{align*}
The first term is a.s. less than $\eps/8$ by Lemma \ref{lemma:concentration:GN}.  The second term is bounded by noting that $R_{N}^{21}(0,z) = G_{N}(z)$ and using \eqref{eq:generalresolventid} to conclude that
\begin{equation} \label{eq:RNeta0}
	\| R_N(\eta,z) - R_N(0,z) \|\oindicator{\Omega_N} \leq 8 |\eta| \leq \frac{\eps}{8}. 
\end{equation}
Thus, it suffices to show that 
\begin{equation} \label{eq:suffRNeta}
	\left| \E u_N^\ast R_N^{21}(\eta,z) v_N \oindicator{\Omega_N} - m(z) u_N^\ast v_N \oindicator{\Omega_N} \right| \leq \eps/2. 
\end{equation}

We will verify \eqref{eq:suffRNeta} by considering the diagonal entries and off-diagonal entries of $R_N^{21}(\eta,z)$ separately.  For the diagonal terms we write
\begin{align*}
	&\left| \E \sum_{i=1}^N \bar{u}_i R_{ii}^{21}(\eta,z) v_{i} \oindicator{\Omega_N} - m(z) u_N^\ast v_N \right| \\
	&\qquad\qquad \leq \sum_{i=1}^N |u_i| |v_i| \max_{1 \leq i \leq N} \E \left| R_{ii}^{21}(\eta,z) - m(z) \right| \oindicator{\Omega_N} \\
	&\qquad\qquad \leq \max_{1 \leq i \leq N} \E \left| R_{ii}^{21}(\eta,z) - m(z) \right| \oindicator{\Omega_N}
\end{align*}
by the Cauchy-Schwarz inequality.  By \eqref{eq:RNeta0} and Lemma \ref{lemma:diagonal}, we have 
\begin{align*}
	\max_{1 \leq i \leq N} \E \left| R_{ii}^{21}(\eta,z) - m(z) \right| \oindicator{\Omega_N} \leq \frac{\eps}{8} + \max_{1 \leq i \leq N} \E \left| (G_N(z))_{ii} - m(z) \right|\oindicator{\Omega_N} \leq \frac{\eps}{4}. 
\end{align*}
Thus, it suffices to show that
$$ \E \sum_{i\neq j} \bar{u}_i R_{ij}^{21}(\eta,z) v_{j} \oindicator{\Omega_N} = o(1). $$

Since $\Omega_N$ holds with overwhelming probability, we have (say)
$$ \E \sum_{i\neq j} \bar{u}_i R_{ij}^{21}(\eta,z) v_{j} \oindicator{\Omega_N}  = \E \sum_{i \neq j} \bar{u}_i R_{ij}^{21}(\eta,z) v_{j} + O(N^{-100}) $$
by the deterministic bound $\|R_N(\eta,z)\| \leq \Im(\eta)^{-1}$.  Thus, it suffices to show that
$$ \E \sum_{i\neq j} \bar{u}_i R_{ij}^{21}(\eta,z) v_{j} = o(1). $$
From Lemma \ref{lemma:offdiag} and the Cauchy-Schwarz inequality, we see that
\begin{align*}
	\left| \E \sum_{i\neq j} \bar{u}_i R_{ij}^{21}(\eta,z) v_{j} \right| &\leq \sum_{i \neq j} |u_{i}| |v_{j}| \max_{i \neq j} | \E R_{ij}^{21}(\eta,z) | \\
	&\leq N \max_{i \neq j}  | \E R_{ij}^{21}(\eta,z) | = o(1),
\end{align*}
and the proof is complete.  
\end{proof}

\appendix

\section{Truncation of elliptic random matrices} \label{app:truncation}

In this appendix, we establish Lemma \ref{lemma:truncation}.  

\begin{proof}[Proof of Lemma \ref{lemma:truncation}]
We begin by noting that, for $i \in \{1,2\}$, 
\begin{equation} \label{eq:vartildexibnd}
	\var(\tilde{\xi}_i) \leq \E|\xi_i|^2 \indicator{|\xi_i| \leq L } \leq 1 
\end{equation}
and, by the dominated convergence theorem, 
$$ \lim_{L \rightarrow \infty} \var(\tilde{\xi_i}) = 1. $$
We take $L_0  > 1$ sufficiently large such that, for each $i \in \{1,2\}$, $\var(\tilde{\xi}_i) \geq 1/2$ for all $L > L_0$.  Assume $L > L_0$.  Then \eqref{eq:aselbnd} follows by an application of the triangle inequality.  Moreover, $\hat{\xi}_1, \hat{\xi}_2$ have mean zero and unit variance by construction.  Thus, $\{\hat{Y}_N\}_{N \geq 1}$ is a sequence of random matrices that satisfies condition {\bf C0} with atom variables $(\hat{\xi}_1,\hat{\xi}_2)$.

We now make use of the following bounds: if $\psi$ is a random variable with finite fourth moment, then
\begin{equation} \label{eq:genpsibnd}
	| \E \psi \indicator{|\psi| > L}| \leq \frac{ \E|\psi|^4 }{ L^3 } \quad \text{and } \quad \E|\psi|^2 \indicator{|\psi| > L} \leq \frac{ \E|\psi|^4}{ L^2 }. 
\end{equation}
We note that
\begin{align*}
	\rho &=  \E[\xi_1 \xi_2 \indicator{|\xi_1| \leq L} \indicator{|\xi_2| \leq L}] + \E[\xi_1 \xi_2 \indicator{|\xi_1| >L} \indicator{\|\xi_2 \leq L}] \\
		&\qquad + \E[\xi_1 \xi_2 \indicator{|\xi_1| \leq L} \indicator{|\xi_2| > L}] + \E[\xi_1 \xi_2 \indicator{|\xi_1| > L} \indicator{|\xi_2|>L}]. 
\end{align*}
Thus, by the Cauchy-Schwarz inequality and \eqref{eq:genpsibnd}, we obtain
\begin{equation} \label{eq:rhohat}
	|\rho - \tilde{\rho}| \leq \frac{C}{L} 
\end{equation}
for some constant $C > 0$ depending on $M_4$.  Similarly, we have
\begin{align} \label{eq:varbnd}
	|1 - \var(\hat{\xi}_i)| \leq \E|\xi_i|^2 \indicator{|\xi_i| > L} + | \E \xi_i \indicator{|\xi_i| > L}|^2 \leq \frac{C}{L^2}
\end{align}
for $i \in \{1,2\}$.  

By \eqref{eq:vartildexibnd}, we have
\begin{align*}
	|\tilde{\rho} - \hat{\rho}| &\leq |\tilde{\rho}| \left| \frac{1}{\sqrt{ \var(\tilde{\xi}_1) \var(\tilde{\xi}_2) }} - 1\right| \\
		&\leq 2 |\tilde{\rho}| \left| \sqrt{ \var(\tilde{\xi}_1) \var(\tilde{\xi}_2)} - 1\right| \\
		& \leq 2 |\tilde{\rho}| \left| \var(\tilde{\xi}_1) \var(\tilde{\xi}_2) - 1 \right| \\
		& \leq 2  |\tilde{\rho}|  \left( \left| 1- \var(\tilde{\xi}_1) \right| + \left| 1 - \var(\tilde{\xi}_2) \right| \right).
\end{align*}
From \eqref{eq:varbnd}, we conclude that $|\tilde{\rho} - \hat{\rho}| \leq \frac{C}{L^2}$ for some constant $C>0$ depending on $M_4$.  Combining this bound with \eqref{eq:rhohat} yields \eqref{eq:rhohatbnd}.

It remains to prove \eqref{eq:ynhat} and \eqref{eq:tilderesolventbnd}.  By Lemma \ref{lemma:normbnd}, it follows that a.s.
$$ \limsup_{N \rightarrow \infty} \frac{1}{\sqrt{N}} \| Y_N \| \leq 4 \quad \text{and} \quad \limsup_{N \rightarrow \infty} \frac{1}{\sqrt{N}} \| \hat{Y}_N \| \leq 4. $$
By Lemma \ref{lemma:detresolventbnd}, we have (say) a.s. 
$$ \limsup_{N \rightarrow \infty} \sup_{|z| \geq 5} \|G_N(z) \| \leq 2 \quad \text{and} \quad \limsup_{N \rightarrow \infty} \sup_{|z| \geq 5} \| \hat{G}_N(z) \| \leq 2. $$
Thus, by the resolvent identity \eqref{eq:generalresolventid}, we have a.s. 
\begin{align*}
	\limsup_{N \rightarrow \infty} \sup_{|z| \geq 5} \| G_N(z) - \hat{G}_N(z) \| &\leq 4 \limsup_{N \rightarrow \infty}  \frac{1}{\sqrt{N}} \| Y_N - \hat{Y}_N \| \\
		&\leq 4 \limsup_{N \rightarrow \infty} \frac{1}{\sqrt{N}} \left( \|Y_N - \tilde{Y}_N \| + \|\tilde{Y}_N - \hat{Y}_N \| \right). 
\end{align*}

Therefore, in order to prove \eqref{eq:ynhat} and \eqref{eq:tilderesolventbnd}, it suffices to show that a.s. 
$$ \limsup_{N \rightarrow \infty} \frac{1}{\sqrt{N}}  \left( \|Y_N - \tilde{Y}_N \| + \|\tilde{Y}_N - \hat{Y}_N \| \right) \leq \frac{C}{L} $$
for some constant $C>0$ depending on $M_4$.  Consider the second term on the left-hand side.  We write
$$ \| \tilde{Y}_N - \hat{Y}_N\| \leq\left\|  \frac{ \left( \tilde{Y}_N - \hat{Y}_N\right) + \left(\tilde{Y}_N - \hat{Y}_N\right)^\ast }{2} \right\| + \left\| \frac{ \left(\tilde{Y}_N - \hat{Y}_N\right) - \left(\tilde{Y}_N - \hat{Y}_N\right)^\ast }{2 \sqrt{-1}} \right\|. $$
We now apply \cite[Theorem 5.2]{BSbook} as in the proof of Lemma \ref{lemma:normbnd}.  By \eqref{eq:varbnd}, we obtain a.s.
$$ \limsup_{N \rightarrow \infty} \frac{1}{\sqrt{N}} \| \tilde{Y}_N - \hat{Y}_N\| \leq \frac{C}{L} $$
for some constant $C>0$ depending on $M_4$.  Similarly, by another application of \cite[Theorem 5.2]{BSbook} and \eqref{eq:genpsibnd}, we have a.s.
$$ \limsup_{N \rightarrow \infty} \frac{1}{\sqrt{N}} \| Y_N - \tilde{Y}_N \| \leq \frac{C}{L}. $$
The proof of the lemma is complete.  
\end{proof}

\section{Large Deviation Estimates} \label{section:lde}

This section is devoted to proving a large deviation estimate for bilinear forms.  Throughout this section, we let $K_p$ denote a constant that depends only on $p$.  These constants are non-random and may take on different values from one appearance to the next.  

\begin{lemma}[Concentration of bilinear forms] \label{lemma:lde}
Let $(x_1,y_1), (x_2,y_2), \ldots, (x_N,y_N)$ be iid random vectors in $\mathbb{C}^2$ such that
$$ \E[x_1] = \E[y_1] = 0, \quad \E|x_1|^2 = \E|y_1|^2 = 1, \quad \E[\bar{x}_1 y_1] = \rho. $$
Let $\mu_p = \max\{ \E|x_1|^p, \E|y_1|^p \}$ for $p \geq 4$.  Let $B=(b_{ij})$ be a deterministic complex $N \times N$ matrix and write $X = (x_1, x_2, \ldots, x_N)^\mathrm{T}$ and $Y = (y_1, y_2, \ldots, y_N)^\mathrm{T}$.  Then, for any $p \geq 2$, 
\[ \E \left| X^\ast B Y - \rho \tr B \right|^p \leq K_p \left( ( \mu_4 \tr (B B^\ast) )^{p/2} + \mu_{2p} \tr (B B^\ast)^{p/2} \right). \]
\end{lemma}

The proof of Lemma \ref{lemma:lde} is based on the proof of \cite[Lemma 2.7]{BS}.  In fact, when $\rho = 1$, we recover \cite[Lemma 2.7]{BS}.  

We will need the following results.  

\begin{lemma}[(3.3.41) of \cite{HJ}] \label{lemma:hj}
For $N \times N$ Hermitian $A = (a_{ij})$ with eigenvalues $\lambda_1, \lambda_2, \ldots, \lambda_N$, and $f$ convex, we have
$$ \sum_{i=1}^N f(a_{ii}) \leq \sum_{i=1}^N f(\lambda_i). $$
\end{lemma}

\begin{lemma}[Lemma A.1 of \cite{BS}] \label{lemma:nonneg}
For $X = (x_1, x_2, \ldots, x_N)^\mathrm{T}$ iid standardized complex entries, $B$ $N \times N$ Hermitian nonnegative definite matrix, we have, for any $p \geq 1$,
$$ \E|X^\ast B X|^p \leq K_p \left( (\tr B)^p + \E|x_1|^{2p} \tr B^p \right). $$
\end{lemma}

We are now ready to prove Lemma \ref{lemma:lde}.  

\begin{proof}[Proof of Lemma \ref{lemma:lde}]
Let $\{\mathcal{F}_i\}_{i=0}^N$ denote the sequence of increasing $\sigma$-algebras defined by
$$ \mathcal{F}_i = \sigma(x_1,y_1, x_2, y_2, \ldots, x_i, y_i) $$
for $i=1,2,\ldots, N$.  Following the usual convention, we let $\mathcal{F}_0$ denote the trivial $\sigma$-algebra.  We will continually make use of this filtration throughout the proof.  

We begin by writing
$$ X^\ast B Y - \rho \tr B =  \sum_{i=1}^N (\bar{x}_i y_i - \rho)b_{ii} + \sum_{i=2}^N \bar{x}_i \sum_{j < i} y_j b_{ij} + \sum_{j=2}^N y_j \sum_{i < j} \bar{x}_i b_{ij} $$
and hence
\begin{align}
	\E &\left| X^\ast B Y - \rho \tr B \right|^p \nonumber \\
	& \leq K_p \left( \E \left| \sum_{i=1}^N (\bar{x}_i y_i - \rho)b_{ii}\right|^p + \E \left| \sum_{i=2}^N \bar{x}_i \sum_{j < i} y_j b_{ij} \right|^p + \E\left| \sum_{j=2}^N y_j \sum_{i < j} \bar{x}_i b_{ij}\right|^p \right). \label{eq:lde:3term}
\end{align}  

We will bound each of the three terms on the right-hand side of \eqref{eq:lde:3term} separately.  We begin with the first term.  By Lemma \ref{lemma:rosenthal},
\begin{align*}
	\E \left| \sum_{i=1}^N (\bar{x}_i y_i - \rho) b_{ii} \right|^p &\leq K_p \left( \left( \sum_{i=1}^N \E|\bar{x}_i y_i - \rho|^2 |b_{ii}|^2 \right)^{p/2} + \sum_{i=1}^N \E| \bar{x}_i y_i - \rho|^p |b_{ii}|^p \right) \\
	&\leq K_p \left( (\mu_4 \tr (B B^\ast) )^{p/2} + \mu_{2p} \sum_{i=1}^N |b_{ii}|^p \right).
\end{align*}
Here we have used
$$ \left( \E|\bar{x}_1 y_1 - \rho|^p \right)^{1/p} \leq \left( \E|x_1|^{2p} \right)^{1/p} + \left( \E|y_1|^{2p} \right)^{1/p} + 1 \leq 3 \mu_{2p}^{1/p}. $$
From Lemma \ref{lemma:hj}, we have
$$ \sum_{i=1}^N |b_{ii}|^p \leq \sum_{i=1}^N (B B^\ast)_{ii}^{p/2} \leq \sum_{i=1}^N \lambda_i(B B^\ast)^{p/2} = \tr (B B^\ast)^{p/2}, $$
where $\lambda_1(B B^\ast), \ldots, \lambda_N(B B^\ast)$ denote the eigenvalues of $B B^\ast$.  Combining the bounds above yields
$$ \E \left| \sum_{i=1}^N (\bar{x}_i y_i - \rho) b_{ii} \right|^p \leq K_p \left( ( \mu_4 \tr B B^\ast)^{p/2} + \mu_{2p} \tr (B B^\ast)^{p/2} \right). $$

We now consider the second term on the right-hand side of \eqref{eq:lde:3term}.  By Lemma \ref{lemma:rosenthal},
$$ \E \left| \sum_{i=2}^N \bar{x}_i \sum_{j < i} y_j b_{ij} \right|^p \leq K_p \left( \E \left( \sum_{i=2}^N \left| \sum_{j < i} y_j b_{ij} \right|^2 \right)^{p/2} + \mu_p \E \sum_{i=2}^N \left| \sum_{j < i} y_j b_{ij} \right|^p \right).  $$
We will bound each of the terms on the right-hand side separately.  For the first term, we write
\begin{align*}
	\E \left( \sum_{i=2}^N \left| \sum_{j<i} y_j b_{ij} \right|^2 \right)^{p/2} &= \E \left( \sum_{i=2}^N \left| \sum_{j=1}^N \E (y_j b_{ij} | \mathcal{F}_{i-1}) \right|^2 \right)^{p/2}.
\end{align*}
Applying Lemma \ref{lemma:dilworth} and Lemma \ref{lemma:nonneg}, we have
\begin{align*}
	\E \left( \sum_{i=2}^N \left| \sum_{j=1}^N \E (y_j b_{ij} | \mathcal{F}_{i-1}) \right|^2 \right)^{p/2} &\leq K_p \E \left( \sum_{i=2}^N \left| \sum_{j=1}^N y_j b_{ij} \right|^2 \right)^{p/2} \\
		&\leq K_p \E \left( Y^\ast B^\ast B Y \right)^{p/2} \\
		&\leq K_p \left( (\tr B B^\ast)^{p/2} + \mu_{2p} \tr (B B^\ast)^{p/2} \right) \\
		&\leq K_p \left( (\mu_4 \tr B B^\ast)^{p/2} + \mu_{2p} \tr (B B^\ast)^{p/2} \right).
\end{align*}
For the second term, we apply Lemma \ref{lemma:rosenthal} and obtain
\begin{align*}
	\E \sum_{i=2}^N \left| \sum_{j < i} y_j b_{ij} \right|^p &\leq K_p \sum_{i=2}^N \left( \left( \sum_{j < i} |b_{ij}|^2 \right)^{p/2} + \mu_p \sum_{j < i} |b_{ij}|^p \right) \\
	 	&\leq K_p (1 + \mu_p) \sum_{i=2}^N \left( \sum_{j < i} |b_{ij}|^2 \right)^{p/2}.  
\end{align*}
We now note that
$$  \sum_{i=2}^N \left( \sum_{j < i} |b_{ij}|^2 \right)^{p/2} \leq \sum_{i=1}^N ( (B B^\ast)_{ii})^{p/2} \leq \tr (B B^\ast)^{p/2} $$
by Lemma \ref{lemma:hj}.  Thus
$$ \mu_p \E \sum_{i=2}^N \left| \sum_{j < i} y_j b_{ij} \right|^p \leq K_p \mu_{2p} \tr (B B^\ast)^{p/2} $$
since $\mu_p(1+\mu_p) \leq 2 \mu_p^2 \leq 2 \mu_{2p}$.  Combining the two bounds above, we obtain
$$ \E \left| \sum_{i=2}^N \bar{x}_i \sum_{j < i} y_j b_{ij} \right|^p \leq K_p \left( (\mu_4 \tr B B^\ast)^{p/2} + \mu_{2p} \tr (B B^\ast)^{p/2} \right). $$

The third term on the right-hand side of \eqref{eq:lde:3term} is similarly bounded.  The proof of the lemma is complete.  
\end{proof}

\section{Properties of the limiting measure} \label{app:prop}

This section is devoted to studying the limiting distribution of the singular values of $\frac{1}{\sqrt{N}}Y_N - zI$, where $z \in \mathbb{C}$ and $\{Y_N\}_{N \geq 1}$ is a sequence of random matrices that satisfy condition {\bf C0} with atom variables $(\xi_1, \xi_2)$.  In particular, this section contains the proof of Theorem \ref{thm:support}.  Throughout this section, we fix $\rho := \E[\xi_1 \xi_2]$ with $-1 < \rho < 1$.  Let $\mathcal{E}_\rho$ be the ellipsoid defined in \eqref{eq:def:ellipsoid}.  We let $\sqrt{-1}$ denote the imaginary unit and reserve $i$ as an index.  

\begin{remark}
Many of the results in this section also hold when $\rho = \pm 1$ (although the proofs are different).  In particular, Theorem \ref{thm:support} holds when $\rho = \pm 1$; see Remark \ref{rem:rho1} for further details.  
\end{remark}

Let $a_N(\eta,z)$ be the Stieltjes transform of $\nu_{\frac{1}{\sqrt{N}}Y_N - zI}$ (defined in \eqref{nu}).  That is, for each $z \in \mathbb{C}$, 
$$ a_N(\eta,z) := \int_{\mathbb{R}} \frac{1}{u - \eta} \nu_{\frac{1}{\sqrt{N}}Y_N - z I}(du) $$
for $\eta \in \mathbb{C}^{+} := \{ w \in \mathbb{C} : \Im(w) > 0 \}$.   We study the limiting distribution of the singular values by characterizing the limiting Stieltjes transform.  We begin with the following lemma.  

\begin{lemma}[Self-consistent equation] \label{lemma:abc} 
Let $z,\eta \in \mathbb{C}$ with $\Im(\eta) > 0$.  Fix $-1 < \rho < 1$.  If $a,b,c \in \mathbb{C}$ satisfy 
\begin{equation} \label{eq:abcmatrix}
	\begin{bmatrix} a & b \\ c & a \end{bmatrix} = \begin{bmatrix} -\eta -a & - \rho c - z \\ - \rho b - \bar{z} & - \eta - a \end{bmatrix}^{-1},
\end{equation}
then
\begin{equation} \label{eq:abcscalar}
	\frac{1}{(a + \eta)a} + 1 = \frac{\Re(z)^2}{(\eta + (1+\rho) a)^2} + \frac{\Im(z)^2}{(\eta + (1-\rho)a)^2}. 
\end{equation}
\end{lemma}
\begin{proof}
We rewrite \eqref{eq:abcmatrix} as the system of equations
\begin{equation} \label{eq:abcsystem}
	\frac{a}{a^2-bc} = -\eta -a, \quad \frac{-b}{a^2 -bc} = -\rho c - z, \quad \frac{-c}{a^2 - bc} = - \rho b - \bar{z}.
\end{equation}
Since $\Im(\eta) > 0$, the first equation implies $a \neq 0$.  Thus, we obtain
\begin{align*}
	- b \left( \frac{\eta + a}{a} \right) &= \rho c + z \\
	- c \left( \frac{\eta+a}{a} \right) &= \rho b + \bar{z}.
\end{align*}
Solving these equations for $b$ and $c$ yields
\begin{align}
	b q &= a^2 \rho \bar{z} - a(\eta+a) z \label{eq:bqexp} \\
	c q &= a^2 \rho z - a(\eta+a) \bar{z}, \label{eq:cqexp}
\end{align}
where $q : = (\eta+a)^2 - \rho^2 a^2$.  Thus, we have
\begin{align*}
	bc q^2 = a^2 |z|^2 ( a^2 \rho^2 + (\eta+a)^2) - a^3(\eta+a) \rho (z^2 + \bar{z}^2), 
\end{align*}
and hence
$$ (a^2 - bc)q^2 = a^2 q^2 - a^2 |z|^2 ( a^2 \rho^2 + (\eta+a)^2) + a^3(\eta+a) \rho (z^2 + \bar{z}^2). $$
Equation \eqref{eq:abcscalar} can now be obtained by combining the calculation above with the first equation from \eqref{eq:abcsystem} and noting that 
\begin{align*}
	 |z|^2 ( a^2 \rho^2 + (\eta+a)^2) - &a(\eta+a) \rho (z^2 + \bar{z}^2) \\
	&= \Re(z)^2(\eta + (1-\rho)a)^2 + \Im(z)^2 (\eta + (1+\rho)a)^2.
\end{align*}
The proof of the lemma is complete.  
\end{proof}

\begin{remark} \label{bcequation}  
One can also use \eqref{eq:bqexp} and \eqref{eq:cqexp} to solve for $b$ and $c$.  Indeed, from \eqref{eq:bqexp} it follows that
$$ b = a \left( \frac{a \rho \bar{z} - (\eta + a) z }{q} \right) = -a \left( \frac{\Re(z)}{\eta + a(1+\rho)} + \sqrt{-1} \frac{\Im(z)}{\eta + a(1-\rho)} \right). $$
Similarly, from \eqref{eq:cqexp}, we have
$$ c = -a \left( \frac{\Re(z)}{\eta + a(1+\rho)} - \sqrt{-1} \frac{\Im(z)}{\eta + a(1-\rho)} \right). $$
\end{remark}
\begin{remark} \label{bcbound}
Fix $z \in \C$.  If \eqref{eq:abcmatrix} holds for all $\eta$ with $\Im(\eta) > 0$, then $a,b,c$ can be viewed as functions of $\eta$.  In this case, an upper bound for $a$ can be obtained (see \eqref{eq:CCsupbnd}).  In fact, in view of Lemma \ref{lemma:continuous}, $a$ can be uniformly bounded from above for all $\Im(\eta) > 0$.  Thus, one can use \eqref{eq:abcscalar} and Remark \ref{bcequation} to obtain uniform upper bounds on $b,c$ for all $\Im(\eta) > 0$.  
\end{remark}

We will also need the following lemma for Stieltjes transforms of probability measures on the real line.  
\begin{lemma} \label{lemma:properties}
Let $\nu$ be a probability measure on the real line.  Let $m$ be the Stieltjes transform of $\nu$.  That is,
$$ m(\eta) = \int_{\mathbb{R}} \frac{1}{u - \eta} \nu(du), \quad \eta \in \mathbb{C}^+. $$
Then
\begin{equation} \label{eq:stconv0}
	\lim_{y \rightarrow \infty} \sup_{x \in \mathbb{R}} |m(x + \sqrt{-1}y)| = 0
\end{equation}
and
\begin{equation} \label{eq:stconv1}
	\lim_{y \rightarrow \infty} \sup_{|x| \leq \sqrt{y}} \left| (x+\sqrt{-1}y) m(x+\sqrt{-1}y) + 1 \right| = 0.
\end{equation}
\end{lemma}
\begin{proof}
Equation \eqref{eq:stconv0} follows from the trivial bound $|m(\eta)| \leq |\Im(\eta)|^{-1}$.  We now prove \eqref{eq:stconv1}.  We note that
$$ (x+\sqrt{-1}y) m(x+\sqrt{-1}y) + 1 = \int_{\mathbb{R}} \frac{u}{(u-x) + \sqrt{-1}y} \nu(du) $$
and hence
\begin{align*}
	&\left| (x+\sqrt{-1}y) m(x+\sqrt{-1}y) + 1 \right| \\
	&\qquad\qquad \leq \int_{|u| \leq 2 \sqrt{y}} \frac{|u|}{\sqrt{(u-x)^2 + y^2}} \nu(du) + \int_{|u| > 2 \sqrt{y}} \frac{|u-x| + |x|}{\sqrt{(u-x)^2 + y^2}} \nu(du) \\
	&\qquad\qquad \leq \frac{2}{\sqrt{y}} + \nu( (-\infty, -2 \sqrt{y}) \cup (2 \sqrt{y}, \infty)) + \frac{|x|}{y}.
\end{align*}
Thus, we have
$$ \sup_{|x| \leq \sqrt{y}} \left| (x+\sqrt{-1}y) m(x+\sqrt{-1}y) + 1 \right| \leq \frac{3}{\sqrt{y}} + \nu( (-\infty, -2 \sqrt{y}) \cup (2 \sqrt{y}, \infty)) $$
and the claim follows.  
\end{proof}

We will use Lemma \ref{lemma:abc} and Lemma \ref{lemma:unique} to study the limit of $a_N(\eta,z)$ for each $z \in \mathbb{C}$ and $\eta \in \mathbb{C}^+$.  Indeed, it follows from the calculations in \cite{NgO} (see also \cite{Nell}) that $a_N(\eta,z)$ converges almost surely as $N \rightarrow \infty$ to a solution of 
\begin{equation} \label{eq:relation}
	\frac{1}{a(\eta,z)(\eta + a(\eta,z))} + 1 = \frac{\Re(z)^2}{(\eta + (1+\rho) a(\eta,z))^2} + \frac{\Im(z)^2}{(\eta + (1 - \rho) a(\eta,z))^2}.  
\end{equation}

\begin{lemma}[Existence and uniqueness] \label{lemma:unique}
Fix $-1 < \rho < 1$.  For each $z \in \mathbb{C}$, there exists a unique probability measure $\nu_z$ on the real line such that
\begin{equation} \label{eq:def:mzw}
	a(\eta,z) := \int \frac{1}{u-\eta} \nu_z(du) 
\end{equation}
is a solution of \eqref{eq:relation} for all $\eta \in \mathbb{C}^+$.  
\end{lemma}
\begin{proof}
Fix $z \in \mathbb{C}$.  Since almost surely $\| \frac{1}{\sqrt{N}} Y_N - z I\| = O_z(1)$ by Lemma \ref{lemma:normbnd}, the sequence of measures 
\begin{equation} \label{eq:def:measures}
	\left\{ \nu_{\frac{1}{\sqrt{N}}Y_N - z I} \right\}_{N \geq 1} 
\end{equation}
is almost surely tight.  Existence now follows from a subsequence argument and by applying \cite[Theorem B.9]{BSbook} and Lemma \ref{lemma:abc}.  

We now prove uniqueness.  Suppose $\nu_z$ and $\nu_z'$ are two probability measures on the real line whose Stieltjes transforms
\begin{align*}
	a(\eta,z) &= \int \frac{1}{u-\eta} \nu_z(du), \\
	s(\eta,z) &= \int \frac{1}{u-\eta} \nu'_z(du)
\end{align*}
satisfy \eqref{eq:relation} for all $\eta \in \mathbb{C}^+$.  Seeking a contradiction, assume $v_z \neq v_z'$.  Since $a(\eta,z)$ and $s(\eta,z)$ are analytic functions of $\eta$ in the upper half plane, it follows from \cite[Theorem B.8]{BSbook} that the set
$$ E:= \left\{ \eta \in \mathbb{C}^+ : m(\eta,z) = s(\eta,z) \right\} $$
has no accumulation point in $\mathbb{C}^+$.  Define the set $Q \subset \mathbb{C}^+ \setminus E$ such that $\eta \in Q$ if and only if
$$ \frac{|\eta + s(\eta,z) + a(\eta,z)|}{|a(\eta,z)(\eta+a(\eta,z))s(\eta,z)(\eta + s(\eta,z))|} \geq 16 \frac{ R^2 \left( |\eta|  + |a(\eta,z)| + |s(\eta,z)| \right)}{|\eta + r a(\eta,z)|^2 |\eta + rs(\eta,z)|^2} $$
for each $r \in \{1 + \rho, 1 - \rho\}$ and $R := \max \{ \Re(z)^2, \Im(z)^2 \}$. By taking $\Im(\eta)$ sufficiently large, it follows from Lemma \ref{lemma:properties} that $Q$ contains an open disk $D$ of radius $\eps > 0$.  Thus, by analytic continuation (and \cite[Theorem B.8]{BSbook}), it suffices to show that $a(\eta,z) = s(\eta,z)$ for all $\eta \in D$.  

Indeed, consider \eqref{eq:relation} for both functions $a(\eta,z)$ and $s(\eta,z)$.  We will subtract one equation from the other.  We first note that
\begin{align*}
	&\frac{1}{a(\eta,z)(\eta + a(\eta,z))} - \frac{1}{s(\eta,z) (\eta + s(\eta,z))} \\
	&\qquad = (s(\eta,z)-a(\eta,z)) \frac{\eta+s(\eta,z) + a(\eta,z)}{a(\eta,z)(\eta + a(\eta,z)) s(\eta,z)(\eta + s(\eta,z))}.
\end{align*}
We also have
\begin{align*}
	&\left| \frac{1}{(\eta + r a(\eta,z))^2} - \frac{1}{(\eta + rs(\eta,z))^2} \right| \\
	&\qquad \leq 4 |s(\eta,z) - a(\eta,z)| \frac{ |\eta| + |s(\eta,z)| + |a(\eta,z)|}{|\eta + ra(\eta,z)|^2 |\eta + rs(\eta,z)|^2} 
\end{align*}
for each $r \in \{1 + \rho, 1 - \rho\}$.  Thus, for $\eta \in D$, we obtain
\begin{align*}
	&|s(\eta,z) - a(\eta,z)| \frac{|\eta+s(\eta,z) + a(\eta,z)|}{|a(\eta,z)(\eta + a(\eta,z)) s(\eta,z)(\eta + s(\eta,z))|} \\
	&\qquad \leq \frac{  |s(\eta,z) - a(\eta,z)| }{2} \frac{|\eta+s(\eta,z) + a(\eta,z)|}{|a(\eta,z)(\eta + a(\eta,z)) s(\eta,z)(\eta + s(\eta,z))|}.
\end{align*}
Since $\Im(\eta + s(\eta,z) + a(\eta,z)) > 0$, we conclude that
\begin{align*}
	|s(\eta,z) - a(\eta,z)| \leq \frac{|s(\eta,z) - a(\eta,z)|}{2}
\end{align*}
for all $\eta \in D$, and the claim follows.  
\end{proof}

For the remainder of the section, we fix $-1 < \rho < 1$ and let $v_z$ denote the unique probability measure from Lemma \ref{lemma:unique}.  Let $a(\eta,z)$ be its Stieltjes transform defined by  \eqref{eq:def:mzw} for all $\eta \in \mathbb{C}^+$.  It follows from Lemma \ref{lemma:abc}, Lemma \ref{lemma:unique}, and the calculations in \cite{NgO} (see also \cite{Nell}) that $a_N(\eta,z)$ converges almost surely to $a(\eta,z)$ as $N \rightarrow \infty$ for each fixed $z \in \mathbb{C}$ and $\eta \in \mathbb{C}^+$.  By \cite[Theorem B.9]{BSbook}, the sequence of measures given in \eqref{eq:def:measures} converge almost surely to $\nu_z$ for each fixed $z \in \mathbb{C}$.  We now derive some properties of $v_z$.  

\begin{lemma}[Properties of $\nu_z$] \label{lemma:continuous} 
Fix $-1 < \rho < 1$.  For each $z \in \mathbb{C}$, $\nu_z$ is compactly supported and has continuous, bounded density.  
\end{lemma}
\begin{proof}
Fix $z \in \mathbb{C}$.  Since almost surely $\| \frac{1}{\sqrt{N}} Y_N - z I \| = O_z(1)$ by Lemma \ref{lemma:normbnd}, it follows that $\nu_z$ is compactly supported.  

We now verify that $\nu_z$ has bounded density.  Consider the Stieltjes transform $a(\eta,z)$ as a solution of \eqref{eq:relation}.  We claim that for any $C' > 0$ there exists a corresponding $C > 0$ such that 
\begin{equation} \label{eq:CCsupbnd}
	\sup_{\eta \in \mathbb{C}^+, |\eta| \leq C'} |a(\eta,z)| \leq C. 
\end{equation}
Indeed, suppose there exists $\eta \in \mathbb{C}^+$ such that $|\eta| \leq C'$ and $|a(\eta,z)| \geq C$ for some sufficiently large constant $C > 0$.  We take $C$ large enough to satisfy 
\begin{equation} \label{eq:rccprime}
	rC - C' \geq \frac{rC}{2}
\end{equation}
and
\begin{equation} \label{eq:4z2bnd}
	\frac{4|z|^2}{r^2C^2} + \frac{2}{rC^2} < 1,
\end{equation}
where $r := \min\{1-\rho, 1+\rho\}$.  From \eqref{eq:rccprime}, we obtain the bounds
\begin{align*}
	\frac{1}{|\eta + (1+\rho) a(\eta,z)|} &\leq \frac{2}{rC}, \quad \frac{1}{|\eta + (1-\rho)a(\eta,z)|} \leq \frac{2}{rC}, \\
	\frac{1}{|\eta+a(\eta,z)|} &\leq \frac{2}{rC}. 
\end{align*}
Applying the bounds above to \eqref{eq:relation} yields
$$ 1 \leq \frac{4|z|^2}{r^2C^2} + \frac{2}{rC^2}. $$
This contradicts our choice of $C$ in \eqref{eq:4z2bnd}, and \eqref{eq:CCsupbnd} follows.  

Choose $C'$ sufficiently large such that $\nu_z$ is supported on $[-C'/2, C'/2]$.  Let $C > 0$ be the corresponding constant such that \eqref{eq:CCsupbnd} holds.  For any finite interval $I \subset \mathbb{R}$, it follows from \cite[Theorem B.8]{BSbook} that
\begin{equation} \label{eq:nuzac}
	\nu_z( I ) \leq 2 C |I|. 
\end{equation}
Here we used the fact that the continuity points of the function $x \mapsto \nu_z((-\infty, x])$ are dense in $\mathbb{R}$.  It follows from \eqref{eq:nuzac} that $\nu_z$ has bounded density.  As the roots of a polynomial depend continuously on the coefficients (see \cite{CC,T}), \eqref{eq:relation} and \cite[Theorem B.10]{BSbook} imply that $\nu_z$ has continuous density.  
\end{proof}

\begin{lemma} \label{lemma:support}
Fix $-1 < \rho < 1$ and $z \notin \mathcal{E}_\rho$.  Then there exists $c > 0$ such that
$$ \nu_z([-c,c]) = 0. $$
\end{lemma}
\begin{proof}
Fix $-1 < \rho < 1$ and $z \notin \mathcal{E}_\rho$.  Since $\nu_z$ is the almost sure limit of the measures in \eqref{eq:def:measures}, it suffices to show that there exists $c > 0$ such that $\nu_z([0,c]) = 0.$  

From \eqref{eq:relation}, $a(\eta,z)$ can be continuously extended to the closed upper plane $\{ \eta \in \mathbb{C} : \Im(\eta) \geq 0 \}$.  We claim that $a(0,z) = 0$.  Suppose to the contrary.  Taking the sequence $\eta = \sqrt{-1} y$ with $y \searrow 0$, we obtain
\begin{equation} \label{eq:limy0}
	1 + \lim_{y \searrow 0} \left(a(\sqrt{-1}y,z) \right)^2 = \frac{\Re(z)^2}{(1+\rho)^2} + \frac{\Im(z)^2}{(1-\rho)^2}. 
\end{equation}
However, since $\Re \left( a_N(\sqrt{-1}y,z)\right) = 0$ for all $y > 0$, it follows that 
$$ \lim_{y \searrow 0} \left(a(\sqrt{-1}y,z) \right)^2 \leq 0. $$  
Thus, \eqref{eq:limy0} contradictions the assumption that $z \notin \mathcal{E}_\rho$.  We conclude that $a(0,z) = 0$.  

From Lemma \ref{lemma:continuous}, $\nu_z$ has bounded, continuous density $p_z$.  We now derive some properties of $p_z$.  From \eqref{eq:relation}, we obtain
\begin{equation} \label{eq:relationoverw}
	\frac{1}{ \frac{a(\eta,z)}{\eta}(1 + \frac{a(\eta,z)}{\eta})} + \eta^2 = \frac{ \Re(z)^2}{ \left(1 + (1+\rho) \frac{a(\eta,z)}{\eta}\right)^2} + \frac{ \Im(z)^2}{ \left( 1 + (1-\rho) \frac{a(\eta,z)}{\eta} \right)^2}
\end{equation}
for $\eta \in \mathbb{C}^+$.  We now claim that $\eta^{-1}a(\eta,z)$ is bounded for all $\eta \in \mathbb{C}^+$.  In order to reach a contradiction, assume $\eta^{-1}a(\eta,z)$ is not bounded.  From \eqref{eq:CCsupbnd} and Lemma \ref{lemma:continuous}, it must be the case that $|\eta^{-1} a(\eta,z)|$ tends to infinity as $\eta$ tends to zero through the upper half plane.  Thus, by multiplying \eqref{eq:relationoverw} by $\eta^{-2} a^2(\eta,z)$ and taking the limit $\eta \rightarrow 0$, we obtain
$$ 1 = \frac{ \Re(z)^2}{(1+\rho)^2} + \frac{\Im(z)^2}{(1-\rho)^2} $$
since $a(0,z) = 0$.  This is clearly a contradiction for $z \notin \mathcal{E}_\rho$.  We conclude that $\eta^{-1} a(\eta,z)$ is bounded for all $\eta \in \mathbb{C}^+$.  Thus, there exists $C > 0$ such that
$$ \Im \left( a(\sqrt{-1}y,z) \right) \leq C y $$
for all $y > 0$.  Equivalently, for all $y > 0$, we have
$$ \int_{\mathbb{R}} \frac{1}{u^2 + y^2} p_z(u) du \leq C. $$
By a change of variables, we obtain
\begin{equation} \label{eq:inteps}
	\int_{0}^\infty \frac{1}{u + y^2} q_z(u) du \leq C
\end{equation}
for all $y > 0$, where $q_z$ is the probability density given by
$$ q_z(u) := \left\{
     \begin{array}{ll}
       \frac{1}{\sqrt{u}} p_z(\sqrt{u}), & u > 0 \\
       0, & u \leq 0
     \end{array}
   \right. .$$

Let $\mu_z$ be the probability measure with density $q_z$.  Let $s(\eta,z)$ be the Stieltjes transform of $\mu_z$.  That is,
$$ s(\eta,z) := \int_{\mathbb{R}} \frac{1}{ u - \eta} \mu_z(du) = \int_{\mathbb{R}} \frac{1}{u-\eta} q_z(u) du, \quad \eta \in \mathbb{C}^+. $$
By definition of $\mu_z$, it follows that $a(\eta,z) = \eta s(\eta^2,z)$ for all $z \in \mathbb{C}$ and $\eta \in \mathbb{C}^+$ with $\Im(\eta^2) \neq 0$.  Thus, $a(\eta,z)$ satisfies 
\begin{equation} \label{eq:relationszw}
	\frac{1}{s(\eta,z) (1 + s(\eta,z))} + \eta = \frac{\Re(z)^2}{(1+(1+\rho)s(\eta,z))^2} + \frac{\Im(z)^2}{(1+(1-\rho)s(\eta,z))^2} 
\end{equation}
for $\eta \in \mathbb{C}^+$.  By \eqref{eq:relationoverw} and the argument above for $\eta^{-1} a(\eta,z)$, it follows that, for $z \notin \mathcal{E}_\rho$, $s(\eta,z)$ is bounded for all $\eta \in \mathbb{C}^+$.  By \cite[Theorem B.8]{BSbook}, we conclude that  the density $q_z$ is bounded.  Moreover, from \eqref{eq:relationszw}, it follows that $s(\eta,z)$ can be continuously extended to $\{ \eta \in \mathbb{C} : \Im(\eta) \geq 0, \eta \neq 0 \}$.

By \eqref{eq:inteps} and Lemma \ref{lemma:sequence} below, we conclude that there exists a sequence of positive real numbers $\{x_k\}_{k \geq 1}$ with $\lim_{k \rightarrow \infty} x_k = 0$ such that $ \lim_{k \rightarrow \infty} q_z(x_k) = 0$.  By \cite[Theorem B.10]{BSbook}, we equivalently have
\begin{equation} \label{eq:subseqxk}
	\lim_{k \rightarrow \infty} \Im \left( s(x_k,z) \right) = 0.
\end{equation}

In order to prove the lemma, it suffices to show that there exists $c > 0$ such that $\mu_z([0,c]) = 0$.  In order to reach a contradiction, suppose for all $c > 0$, $\mu_z([0,c]) > 0$.  We write $s(\eta,z) = g(\eta) + \sqrt{-1} h(\eta)$, for real-valued functions $g,h$.  Choose $\eta = x \in \mathbb{R}$ with $h(x) > 0$.  We now compare the real and imaginary parts for both sides of \eqref{eq:relationszw} and let $x$ approach the boundary of the support (which we have assumed to be $0$).  By allowing $x$ to approach $0$ along the subsequence in \eqref{eq:subseqxk}, we have that $h(x) \rightarrow 0$.  Since $s(\eta,z)$ is bounded, we conclude (by possibly taking a further subsequence) that $g(x) \rightarrow g \in \mathbb{R}$ as $x$ approaches the boundary.  We obtain
\begin{align}
	\frac{1}{g(1+g)} &= \frac{\Re(z)^2}{(1+(1+\rho)g)^2} + \frac{\Im(z)^2}{(1+(1-\rho)g)^2} \label{eq:rerel} \\
	\frac{1+2g}{g^2(1+g)^2} &= \frac{2 \Re(z)^2 (1+\rho)}{(1+(1+\rho)g)^3} + \frac{2 \Im(z)^2 (1-\rho)}{(1+ (1-\rho)g)^3}. \label{eq:imrel}
\end{align}
Since $\mu_z$ is supported on the non-negative real line, it follows that $g > 0$.  Suppose $\rho = 0$.  From \eqref{eq:rerel} and \eqref{eq:imrel}, we obtain
$$ \frac{1+2g}{g^2 (1+g)} = \frac{2}{g(1+g)}, $$
a contradiction.  For the remainder of the proof, we assume $\rho \neq 0$.  Returning to \eqref{eq:rerel} and \eqref{eq:imrel}, we have
$$ \frac{2}{g(1+g)} = \frac{2\Re(z)^2}{(1+(1+\rho)g)^3} + \frac{2\Im(z)^2}{(1+(1-\rho)g)^3} + \frac{1+2g}{g(1+g)^2} $$
and hence
\begin{equation} \label{eq:1g1g}
	\frac{1}{g(1+g)^2} = \frac{2\Re(z)^2}{(1+(1+\rho)g)^3} + \frac{2\Im(z)^2}{(1+(1-\rho)g)^3}. 
\end{equation}
Combining this equation with \eqref{eq:imrel}, we obtain
\begin{equation} \label{eq:1g21g}
	\frac{1}{g^2(1+g)} = 2 \rho \left[ \frac{\Re(z)^2}{(1+(1+\rho)g)^3} - \frac{\Im(z)^2}{(1+(1-\rho)g)^3} \right]. 
\end{equation}
From \eqref{eq:1g1g} and \eqref{eq:1g21g}, we have
\begin{align*}
	\frac{1}{g^2(1+g)^2} &= 4 \rho \frac{\Re(z)^2}{(1+(1+\rho)g)^4} \\
	\frac{-1}{g^2(1+g)^2} &= 4 \rho \frac{\Im(z)^2}{(1+(1-\rho)g)^4}.
\end{align*}
Since $\rho \neq 0$, we conclude that
$$ 0 = \frac{\Re(z)^2}{(1+(1+\rho)g)^4} + \frac{\Im(z)^2}{(1+(1-\rho)g)^4}. $$
This implies $z = 0 \in \mathcal{E}_\rho$, a contradiction.  The proof of the lemma is complete.  
\end{proof}

\begin{remark}
The measure $\mu_z$, defined in the proof of Lemma \ref{lemma:support} above, is the almost sure limit of the empirical spectral measures built from the eigenvalues of $(\frac{1}{\sqrt{N}}Y_N - zI)^\ast ( \frac{1}{\sqrt{N}} Y_N - zI)$.  In fact, equation \eqref{eq:relationszw} has appeared in the work of Girko (see for instance \cite[Lemma 8.1]{Gten}).  Girko discusses the support of $\mu_z$ in \cite[Section 10]{Gten}. 
\end{remark}

\begin{remark}
In the case $\rho = 0$, the exact interval of support of the measure $\nu_z$ can be computed for all $z \in \mathbb{C}$.  See \cite[Remark 3.1]{GTcirc} for further details.  
\end{remark}

\begin{lemma} \label{lemma:sequence}
Let $C > 0$.  Suppose $f:\mathbb{R} \rightarrow [0,\infty)$ is a probability density function that satisfies 
$$ \int_{0}^{\infty} \frac{1}{x + \eps} f(x) dx  \leq C $$
for all $\eps > 0$.  Then there exists a sequence of positive real numbers $\{x_k\}_{k \geq 1}$ with $\lim_{k \rightarrow \infty} x_k = 0$ such that
$$ \lim_{k \rightarrow \infty} f(x_k) = 0. $$
\end{lemma}

Lemma \ref{lemma:sequence} follows from a simple indirect proof; we leave the details to the reader.  Using Lemma \ref{lemma:support}, we now verify Theorem \ref{thm:support}.

\begin{proof}[Proof of Theorem \ref{thm:support}]
Since $\nu_z$ is the almost sure limit of the measures in \eqref{eq:def:measures}, it suffices to show that there exists $c > 0$ such that
$$ \nu_z([0,c]) = 0 $$
for all $z \in \mathbb{C}$ with $\dist(z,\mathcal{E}_\rho) \geq \delta$.  For each $z \in \mathbb{C}$, we define
$$ x_z := \sup \{ x \geq 0 : \nu_z([0,x]) = 0 \}. $$
By Lemma \ref{lemma:continuous} the set above is nonempty, and hence $x_z \geq 0$ for all $z \in \mathbb{C}$.  

We remind the reader that the least singular value of $\frac{1}{\sqrt{N}}Y_N - z I$ is trivially bounded almost surely by Lemma \ref{lemma:detresolventbnd} for $|z|$ sufficiently large because we have the almost sure bound $\|\frac{1}{\sqrt{N}}Y_N \| = O(1)$ from Lemma \ref{lemma:normbnd}.  Thus, it suffices to prove the theorem for all $z$ in a compact set $D \subset \{z \in \mathbb{C} : \dist(z,\mathcal{E}_\rho) \geq \delta \}$.  

We now claim that $x_z$ is continuous in $z$.  Indeed, since $\nu_z$ is the almost sure limit of the measures in \eqref{eq:def:measures}, we obtain almost surely 
$$ \nu_{\frac{1}{\sqrt{N}} Y_N - z I} ( [0,x_z]) \longrightarrow \nu_z([0,x_z]) = 0 $$
as $N \rightarrow \infty$.  Thus, by Weyl's perturbation bound (see for instance \cite{B}), for $|z - z'| \leq x_z$, we have almost surely
$$ \nu_{\frac{1}{\sqrt{N}} Y_N - z'I} ([0,x_z - |z- z'|]) \longrightarrow 0 $$
as $N \rightarrow \infty$.  We conclude that $\nu_{z'}([0,x_z - |z - z'|]) = 0$ and hence 
\begin{equation} \label{eq:xzprime}
	x_{z'} \geq x_z - |z-z'|. 
\end{equation}
We note that \eqref{eq:xzprime} trivially holds when $|z - z'| > x_z$.  Repeating the argument with $z$ and $z'$ reversed, we obtain
$$ |x_z - x_{z'}| \leq |z - z'|. $$ 
We conclude that $x_z$ is continuous in $z$.  Since $D$ is compact, it suffices to show that $x_z > 0$ for all $z \in D$.  The claim now follows from Lemma \ref{lemma:support}.  
\end{proof}


\begin{thebibliography}{99}

\bibitem{A}  G. Anderson, \emph{Convergence of the largest singular value of a polynomial in independent Wigner matrices}, Ann. Probab. Volume 41, Number 3B (2013), 2103--2181.

\bibitem{AGG}  P. Arbenz, W. Gander, G. Golub, \emph{Restricted rank modification of the symmetric eigenvalue problem: Theoretical considerations}, Linear Algebra and its Applications \textbf{104} (1988), 75--95.

\bibitem{Bcirc} Z.~D.~Bai, {\it Circular law}, Ann. Probab. \textbf{25} (1997), 494--529.

\bibitem{BS} Z. D. Bai, J. Silverstein, {\em No eigenvalues outside the support of the limiting spectral distribution of large-dimensional sample covariance matrices}, Ann. Probab. Volume 26, Number 1 (1998), 316--345. 

\bibitem{BSbook} Z. D. Bai, J. Silverstein, {\em Spectral analysis of large dimensional random matrices}, second edition. Mathematics Monograph Series \textbf{2}, Science Press, 
Beijing (2010).

\bibitem{BP} Z.~D.~Bai, G.~M.~Pan, \emph{Limiting behavior of eigenvectors of large Wigner matrices}, J. Stat. Phys. \textbf{146} (2012), 519--549.

\bibitem{BT} D.~Bau, L.~Trefethen, \emph{Numerical linear algebra}, Society for Industrial and Applied Mathematics (SIAM), Philadelphia, PA, 1997.

\bibitem{BR} F. Benaych-Georges, R. Rao, {\em The eigenvalues and eigenvectors of finite, low rank perturbations of large random matrices},
Adv. Math., \textbf{227}, No. 1, (2011), 494--521.

\bibitem{BGM} F. Benaych-Georges, A. Guionnet, M. Maida,  {\em Fluctuations of the extreme eigenvalues of finite rank deformations of random 
matrices}, Elec. J. Probab., \textbf{16} (2011), 1621--1662.

\bibitem{BGM2} F. Benaych-Georges, A. Guionnet, M. Maida, {\em Large deviations of the extreme eigenvalues of random deformations of matrices},
to appear in Probab. Theory Related Fields.

\bibitem{BGR} F. Benaych-Georges, J. Rochet, \emph{Outliers in the Single Ring Theorem}, available at {\tt arXiv:1308.3064 [math.PR]}.

\bibitem{B} R.~Bhatia, \textit{Matrix analysis}, volume 169, Springer Verlag (1997).

\bibitem{BL} P.~Biane, F.~Lehner, \emph{Computation of some examples of Brown's spectral measure in free probability}, Colloq. Math. 90, no. 2, 181--211 (2001).

\bibitem{BC} C. Bordenave, D. Chafa\"i, {\em Around the circular law.} Probability Surveys \textbf{9} (2012). 1--89

\bibitem{BCC} C. Bordenave, P. Caputo, D. Chafai, \emph{Spectrum of Markov generators on sparse random graphs}, available at {\tt arXiv:1202.0644} (2012).

\bibitem{Bmart} D. L. Burkholder, \textit{Distribution function inequalities for martingales}, Ann. Probab. \textbf{1} 19--42 (1973).

\bibitem{CDF1} M. Capitaine, C. Donati-Martin, D. F\'eral, {\em The largest eigenvalue of finite rank deformation of large Wigner matrices: 
convergence and non universality of the fluctuations}, Ann. Probab., \textbf{37}, (1), (2009), 1--47.

\bibitem{CDF} M. Capitaine, C. Donati-Martin, D. F\'eral,  {\em Central limit theorems for eigenvalues of deformations of Wigner matrices},
Ann. I.H.P.-Prob.et Stat., \textbf{48}, No. 1 (2012), 107--133.

\bibitem{CDFF} M. Capitaine, C. Donati-Martin, D. F\'eral, M. F\'evrier, 
{\em Free convolution with a semi-circular distribution and eigenvalues of spiked deformations of Wigner matrices}, Elec. J. Probab., Vol. 16, No. 64, 1750--1792 (2011).

\bibitem{CC} F. Cucker, A. G. Corbalan, \emph{An alternate proof of the continuity of the roots of a polynomial}, American Mathematical Monthly, Vol. 96, No. 4 (1989), 342--345.

\bibitem{D} S. J. Dilworth, \textit{Some probabilistic inequalities with applications to functional analysis}, Banach Spaces. Contemp. Math. \textbf{144} 53--67 (1993). 

\bibitem{Ed-cir} A.~Edelman, {\it The Probability that a random real Gaussian matrix has $k$ real eigenvalues, related distributions, and the circular Law}, J. Multivariate Anal. \textbf{60}, 203--232 (1997). 

\bibitem{FP} D. F\'eral, S. P\'ech\'e,  {\em The largest eigenvalue of rank one deformation of large Wigner matrices}  Comm. Math. Phys.  \textbf{272},  
no. 1, (2007), 185--228.

\bibitem{FK} Z. F\"uredi, J. Koml\'os,  {\em The eigenvalues of random symmetric matrices}, Combinatorica \textbf{1}, (1981), 233--241.

\bibitem{FS} Y. V. Fyodorov, H-J. Sommers, \emph{Statistics of S-matrix poles in few-channel chaotic scattering: crossover from isolated to overlapping resonances}, 
JETP Lett., Volume 63 (1996) Issue 12, 1026--1030.

\bibitem{FS2} Y. V. Fyodorov, H-J. Sommers, \emph{Statistics of resonance poles, phase shifts and time delays in quantum chaotic scattering}, J.Math.Phys. v.38 (1997), 1918--1981.

\bibitem{FKstat} Y. V. Fyodorov, B.A. Khoruzhenko, \emph{Systematic analytical approach to correlation functions of resonances in quantum chaotic scattering},  Phys. Rev. Lett., Volume 83 (1999), Issue 1, 65--68.

\bibitem{Gi} J. ~Ginibre, {\it Statistical ensembles of complex, quaternion and real matrices}, J. Math. Phys. \textbf{6} (1965), 440--449.

\bibitem{Gorig} V.~L.~Girko, \textit{Elliptic law}, Theory of Probability and Its Applications, Vol. 30, No. 4 (1985).

\bibitem{Gten} V.~L.~Girko, \textit{The elliptic law: ten years later I}, Random Oper. and Stoch. Equ., Vol. 3, No. 3, pp. 257--302 (1995).

\bibitem{GTcirc} F.~G\"otze, T.~Tikhomirov, \textit{The circular law for random matrices}, Ann. Probab. Volume 38, Number 4 1444--1491 (2010).  

\bibitem{HJ} R. A. Horn, C. R. Johnson, \textit{Topics in Matrix Analysis}, Cambridge Univ. Press (1991).  

\bibitem{HJ2} R.~A.~Horn, C.~R.~Johnson, \textit{Matrix Analysis}, Cambridge Univ. Press (1991).  

\bibitem{HT} U. Haagerup, S. Thorbj\o rnsen, \textit{ A new application of random matrices: $Ext(C_{red}^*(F_2))$ is not a group.} Ann. of Math. (2) \textbf{162} 711--775 (2005).

\bibitem{KY} A. Knowles, J. Yin, {\em The isotropic semicircle law and deformation of Wigner matrices}, 
available at {\tt arXiv:1110.6449}.

\bibitem{KY2} A. Knowles, J. Yin, \emph{The outliers of a deformed Wigner matrix}, available at {\tt arXiv:1207.5619}.

\bibitem{M} M. L.~Mehta,  {\em Random matrices and the statistical theory of energy levels}, Acad. Press (1967).

\bibitem{M:B} M.~L.~Mehta, {\it Random Matrices}, third edition.  Elsevier/Academic Press, Amsterdam (2004).

\bibitem{Nell} A.~Naumov, \emph{Elliptic law for real random matrices}, available at {\tt arXiv:1201.1639 [math.PR]}.  

\bibitem{NgO} H.~Nguyen, S.~O'Rourke.  \emph{The elliptic law}, available at {\tt  arXiv:1208.5883 [math.PR]}. 

\bibitem{PZ} G.~Pan, W.~Zhou, {\it Circular law, extreme singular values and potential theory}, Journal of Multivariate Analysis, \textbf{101} 645--656 (2010).

\bibitem{P} S. P\'ech\'e, {\em The largest eigenvalue of small rank perturbations of Hermitian random matrices}.  Probab. Theory Related Fields \textbf{134},  
no. 1, (2006), 127--173.

\bibitem{PRS} A. Pizzo, D. Renfrew, A. Soshnikov,  {\em On finite rank deformations of Wigner matrices}, to appear in Annales de L'Institut 
Henri Poincar\'{e} Probabilit\'{e}s et Statistiques, available at {\tt arXiv:1103.3731 [math.PR]}.

\bibitem{PRS2} D. Renfrew, A. Soshnikov, \emph{On Finite Rank Deformations of Wigner Matrices II: Delocalized Perturbations}, 
available at {\tt arXiv:1203.5130 math.PR}.

\bibitem{R} W. Rudin, \emph{Principles of mathematical analysis}. Third edition. International Series in Pure and Applied Mathematics. McGraw-Hill Book Co., New York-Auckland-D\"{u}sseldorf, (1976). 

\bibitem{TVcirc}  T. Tao, V. Vu, \emph{Random matrices: The Circular Law}, Communication in Contemporary Mathematics \textbf{10} (2008), 261--307.

\bibitem{TVbull} T. Tao, V. Vu, {\it  From the Littlewood-Offord problem to the circular law: universality of the spectral distribution of random matrices},
 {Bull.~Amer.~Math.~Soc.}  (N.S.) \textbf{46} (2009), no. 3, 377--396.

\bibitem{TVesd} T.~Tao, V.~Vu, \emph{Random matrices: Universality of ESDs and the circular law}, Ann. Probab. Volume 38, Number \textbf{5} (2010), 2023--2065. 

\bibitem{Tout} T.~Tao, \emph{Outliers in the spectrum of iid matrices with bounded rank perturbations}, Probab. Theory Related Fields \textbf{155} (2013), 231--263.  

\bibitem{T} E.~E.~Tyrtyshnikov, \textit{A Brief Introduction to Numerical Analysis}, Birkh\"{a}user Boston (1997).  

\bibitem{Voic} D.~Voiculescu, \emph{Limit laws for Random matrices and free products}, Inv. math, Volume 104, Issue 1, pp 201--220 (1991).

\bibitem{W} E.~P.~Wigner, {\it On the distributions of the roots of certain symmetric matrices}, Ann. Math. \textbf{67}, (1958) 325--327.

\end{thebibliography}
\end{document}